\documentclass[a4paper,11pt]{article}

\usepackage[a4paper,top=1in,bottom=1in,left=0.75in,right=0.75in]{geometry}
\usepackage{amsmath,amssymb,amscd,amsthm,mathrsfs,bm}
\usepackage{enumitem}
\usepackage{graphicx}
\usepackage{epstopdf}
\usepackage{booktabs}
\usepackage{threeparttable}
\usepackage{xcolor}
\usepackage[numbers,sort&compress]{natbib}

\usepackage[T1]{fontenc}
\usepackage{mathpazo,microtype}
\usepackage{hyperref}
\definecolor{myblue}{rgb}{0.2,0,0.9}
\definecolor{blue-violet}{rgb}{0.54,0.17,0.89}

\hypersetup{hidelinks,pdftitle={Consumption and Investment When Welfare Floors Change over Time},pdfauthor={Gugyum Ha, Junkee Jeon, Takwon Kim}}

\numberwithin{equation}{section}

\theoremstyle{plain}
\newtheorem{as}{Assumption}[section]

\newtheorem{thm}{Theorem}[section]
\newtheorem{lem}[thm]{Lemma}
\newtheorem{prop}[thm]{Proposition}
\newtheorem{cor}[thm]{Corollary}

\theoremstyle{definition}
\newtheorem{pr}{Problem}[section]

\theoremstyle{remark}
\newtheorem{rem}{Remark}[section]

\newcommand{\cL}{\mathscr{L}}

\begin{document}

\title{\LARGE \textbf{Consumption and Investment When Welfare Floors Change over Time}}

\author{
Gugyum Ha \footnote{E-mail: \href{mailto:ggha@sogang.ac.kr}{\texttt{ggha@sogang.ac.kr}}\;Department of Mathematics, Sogang University, Seoul 04107, Republic of Korea.}
	\and
    Junkee Jeon \footnote{E-mail: \href{mailto:junkeejeon@khu.ac.kr}{\texttt{junkeejeon@khu.ac.kr}}\;Department of Applied Mathematics, Kyung Hee University, Yongin 17104, Republic of Korea.}
	\and
    Takwon Kim\footnote{Corresponding author. E-mail: \href{mailto:takwon@sungshin.ac.kr}{\texttt{takwon@sungshin.ac.kr}}\;School of Mathematics, Statistics and Data Science, Sungshin Women's University, Seoul 02844, Republic of Korea.}
}

\date{}

\maketitle
\pagestyle{plain}
\pagenumbering{arabic}

\begin{abstract}
We study consumption and investment over a finite horizon under time dependent welfare floors and a terminal wealth floor. A cumulative multiplier yields a parabolic obstacle problem. For relative risk aversion above one, the intervention boundary exists at every interior date; below one, a running record criterion characterizes active dates, so the boundary can disappear and reenter. We establish boundary regularity, verify optimal controls, and identify the exact lower endpoint of feasible wealth. With stochastic income in a complete market, an outside option determines the floor and links the boundary to limited commitment and borrowing capacity.
\end{abstract}

\begin{flushleft}
\textbf{Short title:} Welfare floors that change over time.
\end{flushleft}

\section{Introduction}\label{sec:intro}

Dynamic lower bounds on continuation utility arise whenever an allocation must remain acceptable at every future date and not only at the start. Limited commitment is the leading example. If a household can walk away from a financial arrangement, the continuation utility it is promised must be at least the value of its outside option at each date. This is the participation constraint of \citet{KehoeLevine1993}, \citet{Kocherlakota1996}, and \citet{AlvarezJermann2000}, and of the continuous time formulations in \citet{MiaoZhang2015} and \citet{ChoiKooLimYoo2024}. A welfare floor can also describe a minimum acceptable standard of living or another sustainability requirement. In each of these settings the constraint is imposed on continuation value itself.

The floor need not be constant. The value of an outside option may vary with the time remaining, and at the horizon the welfare requirement becomes a minimum wealth condition through bequest utility. We ask how a welfare floor that changes over time, together with the terminal requirement it implies, shapes consumption, portfolio choice, and feasible borrowing over a finite horizon. With a moving floor, whether the constraint can bind at all may depend on the date.

We study an agent with CRRA preferences who receives labor income and chooses consumption and risky investment over a fixed horizon. At every date, continuation utility must stay above a deterministic welfare floor that may vary with time. Following the martingale approach of \citet{KaratzasLehoczkyShreve1987} and \citet{CoxHuang1989}, together with the treatment of dynamic constraints in \citet{HePages1993} and \citet{ElKarouiJeanblanc1998}, we represent the continuum of welfare constraints by a nondecreasing cumulative multiplier. The duality theory behind this representation has been extended to random endowments, constraints that depend on the wealth path, and intermediate consumption in incomplete markets by \citet{CvitanicSchachermayerWang2001}, \citet{BouchardPham2004}, \citet{Mostovyi2015}, and \citet{GuasoniWang2020}. The dual problem is a singular control problem of monotone follower type on a finite horizon, and its connection with optimal stopping goes back to \citet{KaratzasShreve1984}. Homogeneity reduces the state dimension, and the gradient constraint leads to a one dimensional parabolic obstacle problem with a time dependent obstacle and a kinked terminal condition. The general theory of such variational inequalities can be found in \citet{KinderlehrerStampacchia1980} and \citet{PetrosyanShahgholianUraltseva2012}.

Our main technical contribution is a description of the free boundary geometry that does not rely on monotonicity in time. Unit relative risk aversion separates two cases. When relative risk aversion exceeds one, the contact set is nonempty at every interior date. When it is below one, contact occurs only at dates where a discounted transformation of the welfare floor is locally increasing and sets a strict running record relative to its past values. The free boundary can therefore diverge, disappear on a nontrivial interval, and later reenter. We prove existence, uniqueness, and monotonicity in the state variable for the solution of the obstacle problem. Near the boundary we bound the time derivative of the solution by its space derivative, and this estimate gives quadratic growth, local Lipschitz continuity, and local $C^{1,\alpha}$ regularity of the boundary for every $\alpha<1/2$, together with its behavior at the endpoints of the active dates. Related regularity results for parabolic obstacle problems and American exercise boundaries appear in \citet{BlanchetDolbeaultMonneau2006}, \citet{Villeneuve1999}, \citet{ChenChadam2007}, and \citet{LambertonMikou2008}.

The optimal cumulative multiplier is the minimal reflected process that keeps the normalized dual state below the intervention boundary, and it increases only when the welfare constraint binds. We construct the dual value from the obstacle solution, prove the moment bounds needed for verification, and recover the optimal primal policies. The moment bounds follow from a representation of the multiplier as a running supremum, in the spirit of the stochastic representation theorem of \citet{BankElKaroui2004}. In wealth space the dual free boundary becomes an endogenous credit limit that varies with time, and this limit is the exact lower endpoint of the set of feasible wealth levels. Even when no intervention boundary exists before the horizon, the terminal wealth floor can keep feasible wealth strictly above the natural debt limit.

With stochastic labor income the model has a limited commitment interpretation. We take the welfare floor to be the value of an outside option in which the household leaves the arrangement and consumes a fraction of its labor income. Under complete markets and CRRA homogeneity, scaling by current income and a change of measure reduce this problem to the baseline model, so the same free boundary analysis applies. A more attractive outside option raises the endogenous lower bound on the wealth to income ratio and reduces borrowing capacity. In the numerical examples its effect on consumption and risky investment is largest near the credit limit.

The closest paper is \citet{JeonKwak2024Welfare}, who study a constant welfare floor in an infinite horizon consumption and investment problem. A floor that changes over time on a finite horizon brings in terminal welfare, genuine nonstationarity, and an endogenous set of active dates on which the free boundary is finite. The feasible wealth domain then changes with time, whereas the stationary model has a single minimum wealth level. Our analysis is also related to free boundary problems on a finite horizon in portfolio choice and mathematical finance, including \citet{DaiYi2009}, \citet{Grandits2015}, \citet{DeAngelisStabile2019}, \citet{YangJeon2026}, and \citet{KimYangJeon2026}. Free boundaries arising from singular control in irreversible or partially reversible investment are studied in \citet{RiedelSu2011} and \citet{DeAngelisFerrari2014}. Dynamic constraints on the wealth path lead to related free boundaries in the drawdown problem of \citet{ElieTouzi2008} and in the regularity analysis of \citet{FedericoGassiatGozzi2015}. Preferences that compare consumption with a reference level give a related singular control structure in the infinite horizon model of \citet{DengLiPhamYu2022}. In the present problem, whether an intervention boundary exists depends on the date. We characterize the active dates by a running record condition and map the dual boundary exactly to the primal credit limit.

Section~\ref{sec:model} formulates the model. Section~\ref{sec:optimization} derives the dual formulation as a singular control problem and the reduced variational inequality. Section~\ref{sec:parabolic-obstacle} analyzes the obstacle problem and its free boundary. Section~\ref{sec:verification-duality} establishes verification and strong duality and represents the endogenous credit limit. Section~\ref{sec:implications} treats stochastic income and limited commitment. Section~\ref{sec:conclusion} concludes, and the Appendix contains the proofs.

\section{Model}\label{sec:model}

\subsection{Financial market and wealth dynamics}

The market consists of a riskless asset with rate $r>0$ and a risky asset $S=(S_t)_{0\le t\le T}$ satisfying
\begin{align*}
\frac{dS_t}{S_t}=\mu\,dt+\sigma\,dB_t,
\qquad 0\le t\le T,
\end{align*}
where $\mu\neq r$ and $\sigma>0$ are constants, and $B=(B_t)_{0\le t\le T}$ is a Brownian motion on a standard probability space $(\Omega,\mathcal F,\mathbb P)$ endowed with the augmented filtration $\mathbb F=\{\mathcal F_t\}_{0\le t\le T}$ generated by $B$.

The agent receives labor income at the constant rate $\epsilon>0$ over $[0,T]$.
Given an $\mathbb F$-progressively measurable pair $(c,\pi)$, where $c_t>0$ denotes the consumption rate and $\pi_t$ denotes the dollar amount invested in the risky asset at time $t$, the associated wealth process $X^{c,\pi}=(X_t^{c,\pi})_{0\le t\le T}$ satisfies
\begin{align}
dX_t^{c,\pi}
=
\big(rX_t^{c,\pi}+\pi_t(\mu-r)-c_t+\epsilon\big)\,dt
+\sigma\pi_t\,dB_t,
\qquad
X_0^{c,\pi}=x.
\label{eq:wealth_finite}
\end{align}

The time-$t$ present value of future labor income is
\begin{align}
h(t)
:=
\int_t^T e^{-r(s-t)}\epsilon\,ds
=
\frac{\epsilon}{r}\big(1-e^{-r(T-t)}\big),
\qquad 0\le t\le T.
\label{eq:human_wealth_finite}
\end{align}
We assume
\begin{align*}
x>-h(0)=-\frac{\epsilon}{r}\big(1-e^{-rT}\big).
\end{align*}

\subsection{Preferences, bequest utility, and continuation value}

Consumption and bequest utility are CRRA with $\gamma>0$, $\gamma\neq1$. For $c>0$ and $x>0$,
\begin{align}
u(c)=\frac{c^{1-\gamma}}{1-\gamma},
\label{eq:utility_crra}
\end{align}
and
\begin{align}
U_B(x)=u(\kappa_B x)=\frac{(\kappa_B x)^{1-\gamma}}{1-\gamma},
\label{eq:bequest_crra}
\end{align}
where $\kappa_B>0$ is a constant scaling parameter for bequest utility.

Given a strategy $(c,\pi)$, the agent evaluates it through
\begin{align}
\mathbb E\bigg[
\int_0^T e^{-\beta t}u(c_t)\,dt
+
e^{-\beta T}U_B(X_T^{c,\pi})
\bigg],
\label{eq:objective_finite}
\end{align}
where $\beta>0$ is the subjective discount rate.

The continuation value is
\begin{align}
{\cal W}_t
:=
\mathbb E_t\bigg[
\int_t^T e^{-\beta(s-t)}u(c_s)\,ds
+
e^{-\beta(T-t)}U_B(X_T^{c,\pi})
\bigg],
\label{eq:continuation_value_finite}
\end{align}
where $\mathbb E_t[\cdot]=\mathbb E[\cdot\mid\mathcal F_t]$.

\subsection{Dynamic welfare constraint}

Let $P:[0,T]\to\mathbb R$ be a deterministic welfare floor and impose
\begin{align}
{\cal W}_t\ge P(t),
\qquad \forall t\in[0,T].
\label{eq:welfare_constraint_finite}
\end{align}

\begin{as}\label{as:P_finite}
The welfare floor $P:[0,T]\to\mathbb R$ belongs to $C^{1,1}([0,T])$, that is, $P$ is continuously differentiable and $P'$ is Lipschitz continuous.
Moreover,
\begin{align}
(1-\gamma)P(t)>0,
\qquad \forall t\in[0,T].
\label{eq:P_sign_condition}
\end{align}
Equivalently,
\begin{align*}
P(t)>0 \quad \text{for all } t\in[0,T] \qquad \text{if } \gamma\in (0,1),
\end{align*}
whereas
\begin{align*}
P(t)<0 \quad \text{for all } t\in[0,T] \qquad \text{if } \gamma\in (1,\infty).
\end{align*}
\end{as}
The Lipschitz continuity of $P'$ is needed for the optimal regularity in Lemma~\ref{lem:existence-R}, for the local contact bound in Lemma~\ref{lem:necessary-contact-R}, and for the estimates of the time derivative in Lemma~\ref{lem:time-space-derivative-estimate} and Proposition~\ref{prop:boundary-higher-regularity}, which require $P''\in L^\infty$. The sign condition \eqref{eq:P_sign_condition} costs no generality. If $\gamma>1$ and $P(t)\ge0$ for some $t$, no strategy satisfies \eqref{eq:welfare_constraint_finite} because ${\cal W}_t<0$. If $\gamma\in(0,1)$ and $P(t)\le0$, the constraint is void at $t$ because ${\cal W}_t>0$.

\begin{rem}
Since ${\cal W}_T=U_B(X_T^{c,\pi})$, the constraint \eqref{eq:welfare_constraint_finite} at $t=T$ reads $U_B(X_T^{c,\pi})\ge P(T)$. Assumption \ref{as:P_finite} ensures a unique $\underline{x}_B>0$ satisfying
\begin{align}
U_B(\underline{x}_B)=P(T),
\qquad\text{that is,}\qquad
\underline{x}_B
=
\frac{1}{\kappa_B}\Big((1-\gamma)P(T)\Big)^{\frac{1}{1-\gamma}},
\label{eq:xB_bar_definition}
\end{align}
and since $U_B$ is strictly increasing, the terminal constraint is equivalent to $X_T^{c,\pi}\ge \underline{x}_B$. Thus the dynamic welfare constraint contains a terminal wealth floor $\underline{x}_B$ induced by the welfare floor.
\end{rem}

Under the integrability condition below, the continuation value has a continuous version in the Brownian filtration. We use this version throughout, so the constraint \eqref{eq:welfare_constraint_finite} holds at all dates simultaneously outside one null set. Its dynamics are
\begin{align}
d{\cal W}_t
=
\big(\beta{\cal W}_t-u(c_t)\big)\,dt+\varsigma_t^{\mathcal W}\,dB_t,
\qquad
{\cal W}_T=U_B(X_T^{c,\pi}),
\label{eq:W_dynamics_finite}
\end{align}
for some $\mathbb F$-progressively measurable process $\varsigma^{\mathcal W}=(\varsigma_t^{\mathcal W})_{0\le t\le T}$.

\subsection{Admissible strategies}

A progressively measurable pair $(c,\pi)$ is admissible, written $(c,\pi)\in{\cal A}(x)$, if the following conditions hold.
\begin{itemize}
\item[(a)] 
\begin{align}
\int_0^T c_s\,ds<\infty
\quad\mbox{a.s.},
\qquad
\int_0^T \pi_s^2\,ds<\infty
\quad\mbox{a.s.}
\label{eq:admissible_finite_1}
\end{align}

\item[(b)]
\begin{align}
\mathbb E\bigg[
\int_0^T e^{-\beta t}|u(c_t)|\,dt
+
e^{-\beta T}|U_B(X_T^{c,\pi})|
\bigg]
<\infty,
\label{eq:admissible_finite_2}
\end{align}
so that the objective functional and the continuation value are well defined.

\item[(c)] The continuation value process satisfies the dynamic welfare constraint \eqref{eq:welfare_constraint_finite}.

\item[(d)] The wealth process stays above the natural debt limit,
\begin{align}
X_t^{c,\pi}>-h(t)
=
-\frac{\epsilon}{r}\big(1-e^{-r(T-t)}\big)
\quad\mbox{a.s.}\quad \forall t\in[0,T].
\label{eq:admissible_finite_3}
\end{align}
\end{itemize}

Since $h(T)=0$, the natural debt limit implies
\begin{align*}
X_T^{c,\pi}>0
\qquad \mbox{a.s.},
\end{align*}
and the terminal welfare constraint further requires
\begin{align*}
X_T^{c,\pi}\ge \underline{x}_B.
\end{align*}

\subsection{Primal problem}

\begin{pr}\label{pr:main_finite}
For a given initial wealth $x>-h(0)$, define
\begin{align*}
V(x)
=
\sup_{(c,\pi)\in{\cal A}(x)}
\mathbb E\bigg[
\int_0^T e^{-\beta t}u(c_t)\,dt
+
e^{-\beta T}U_B(X_T^{c,\pi})
\bigg].
\end{align*}
\end{pr}

\section{Optimization}\label{sec:optimization}

This section derives a dual formulation of Problem~\ref{pr:main_finite}. The welfare multiplier enters both the consumption term and the terminal wealth term.

\subsection{Static budget constraint}

Let
\[
\theta:=\frac{\mu-r}{\sigma},
\]
and define the state price density by
\begin{equation}\label{eq:state_price_density_finite}
{\cal H}_t
=
\exp\!\left(
-\Big(r+\frac{\theta^2}{2}\Big)t-\theta B_t
\right),
\qquad 0\le t\le T.
\end{equation}
For an admissible strategy, It\^o's formula and the identity $h'-rh+\epsilon=0$, which follows from \eqref{eq:human_wealth_finite}, show that $\mathcal H_t(X_t^{c,\pi}+h(t))+\int_0^t\mathcal H_sc_s\,ds$ is a nonnegative local martingale. It is therefore a supermartingale. Evaluating its expectation at $T$ shows that any admissible strategy $(c,\pi)\in{\cal A}(x)$ satisfies
\begin{equation}\label{eq:budget_finite}
\mathbb E\bigg[
\int_0^T {\cal H}_t(c_t-\epsilon)\,dt
+
{\cal H}_T X_T^{c,\pi}
\bigg]
\le x.
\end{equation}

For $y>0$, define
\begin{equation}\label{eq:Y_process_finite}
{\cal Y}_t^y:=y e^{\beta t}{\cal H}_t,
\qquad 0\le t\le T.
\end{equation}
Then ${\cal Y}^y$ solves
\[
\frac{d{\cal Y}_t^y}{{\cal Y}_t^y}
=
(\beta-r)\,dt-\theta\,dB_t,
\qquad
{\cal Y}_0^y=y.
\]

\subsection{Heuristic derivation}

Let $y>0$ and a nonnegative optional process $\eta=(\eta_t)_{0\le t\le T}$ be multipliers for the budget constraint \eqref{eq:budget_finite} and the welfare constraint \eqref{eq:welfare_constraint_finite}, and form the Lagrangian by adding $y\big(x-\mathbb E[\int_0^T{\cal H}_t(c_t-\epsilon)\,dt+{\cal H}_TX_T^{c,\pi}]\big)+\mathbb E[\int_0^T\eta_t({\cal W}_t-P(t))\,dt]$ to the objective \eqref{eq:objective_finite}.
Applying Fubini's theorem formally to the last term and defining the cumulative multiplier
\[
{\cal Z}_t:=\int_0^t e^{\beta s}\eta_s\,ds,
\qquad 0\le t\le T,
\]
which is nonnegative and nondecreasing, the Lagrangian of Problem~\ref{pr:main_finite} becomes, by \eqref{eq:Y_process_finite},
\begin{equation}\label{eq:Lagrangian_finite}
\begin{aligned}
\mathfrak L
=
&\,
\mathbb E\bigg[
\int_0^T e^{-\beta t}
\Big(
(1+{\cal Z}_t)u(c_t)
-
{\cal Y}_t^y c_t
+
{\cal Y}_t^y\epsilon
\Big)\,dt
\\
&\hspace{1.6cm}
+
e^{-\beta T}
\Big(
(1+{\cal Z}_T)U_B(X_T^{c,\pi})
-
{\cal Y}_T^y X_T^{c,\pi}
\Big)
-
\int_{[0,T]} e^{-\beta t}P(t)\,d{\cal Z}_t
\bigg]
+
yx.
\end{aligned}
\end{equation}
We extend the class of cumulative multipliers to general nondecreasing processes, which allows singular continuous increments and atoms at $0$ and $T$. The Stieltjes integrals are therefore taken over closed intervals.
\subsection{Pointwise optimization}

Let $\tilde u$ be the convex dual of $u$:
\[
\tilde u(q):=\sup_{c\ge0}\big(u(c)-qc\big),
\qquad q>0.
\]
Under the CRRA specification
\[
u(c)=\frac{c^{1-\gamma}}{1-\gamma},
\qquad \gamma>0,\ \gamma\neq 1,
\]
we have
\begin{equation}\label{eq:tilde_u_CRRA}
I(q)=(u')^{-1}(q)=q^{-1/\gamma},
\qquad
\tilde u(q)=\frac{\gamma}{1-\gamma}q^{-\frac{1-\gamma}{\gamma}}.
\end{equation}

Because the welfare constraint at time $T$ implies $X_T^{c,\pi}\ge \underline{x}_B$, we define the constrained convex dual of $U_B$ by
\[
\tilde U_B(q):=\sup_{x\ge \underline{x}_B}\big(U_B(x)-qx\big),
\qquad q>0.
\]
Since
\[
U_B(x)=\frac{(\kappa_B x)^{1-\gamma}}{1-\gamma},
\qquad
U_B'(x)=\kappa_B^{1-\gamma}x^{-\gamma},
\]
the constrained maximizer is
\[
\hat x_B(q)
=
\max\left\{
\underline{x}_B,\,
\kappa_B^{\frac{1-\gamma}{\gamma}}q^{-1/\gamma}
\right\},
\]
and therefore
\begin{equation}\label{eq:tilde_UB_CRRA}
\tilde U_B(q)
=
\begin{cases}
\displaystyle
\frac{\gamma}{1-\gamma}\,
\kappa_B^{\frac{1-\gamma}{\gamma}}
q^{-\frac{1-\gamma}{\gamma}},
& \text{if } q\le \kappa_B^{1-\gamma}\underline{x}_B^{-\gamma},
\\[2ex]
\displaystyle
P(T)-q\underline{x}_B,
& \text{if } q> \kappa_B^{1-\gamma}\underline{x}_B^{-\gamma}.
\end{cases}
\end{equation}
The kink in \eqref{eq:tilde_UB_CRRA} comes from the terminal wealth floor.

Optimizing \eqref{eq:Lagrangian_finite} pointwise over $c$ and $X_T$ yields
\[
\hat c_t
=
I\!\left(\frac{{\cal Y}_t^y}{1+{\cal Z}_t}\right),
\qquad
\hat X_T
=
\max\bigg\{
\underline{x}_B,\,
\frac{1}{\kappa_B}
I\!\left(
\frac{{\cal Y}_T^y}{\kappa_B(1+{\cal Z}_T)}
\right)
\bigg\},
\]
and hence
\[
\mathfrak L\le \mathfrak L(y,{\cal Z})+yx,
\]
where
\begin{equation}\label{eq:dual_functional_finite}
\begin{split}
\mathfrak L(y,\mathcal Z):=\mathbb E\bigg[&
\int_0^T e^{-\beta t}\left\{
(1+\mathcal Z_t)\widetilde u\!\left(\frac{\mathcal Y_t^y}{1+\mathcal Z_t}\right)
+\epsilon\mathcal Y_t^y\right\}\,dt\\
&+e^{-\beta T}(1+\mathcal Z_T)
\widetilde U_B\!\left(\frac{\mathcal Y_T^y}{1+\mathcal Z_T}\right)
-\int_{[0,T]}e^{-\beta t}P(t)\,d\mathcal Z_t\bigg].
\end{split}
\end{equation}

\subsection{The dual problem and weak duality}

For $z\ge0$, let $\Pi(z)$ denote the set of all finite valued, nonnegative, nondecreasing, c\`adl\`ag, $\mathbb F$-adapted cumulative multipliers ${\cal Z}$ on $[0,T]$ such that ${\cal Z}_{0-}=z$ and $\mathbb E[(1+\mathcal Z_T)^m]<\infty$ for every $m>1$. The condition ${\cal Z}_{0-}=z$ specifies the value before any initial intervention, so a jump at time $0$ is allowed. The moment condition ensures integrability of the dual criterion, and Lemma~\ref{lem:reflection-moments} verifies it for the optimal cumulative multiplier.
For each $y>0$, define
\begin{equation}\label{eq:dual_problem_finite}
J(y)
:=
\inf_{{\cal Z}\in\Pi(0)} \mathfrak L(y,{\cal Z}).
\end{equation}
The dual objective for the primal problem is
\begin{equation}\label{eq:full_dual_problem_finite}
\inf_{y>0}\big(J(y)+yx\big).
\end{equation}

\begin{prop}\label{pro:weak-duality}
For any $y>0$, ${\cal Z}\in \Pi(0)$, and $(c,\pi)\in{\cal A}(x)$,
\[
\mathbb E\bigg[
\int_0^T e^{-\beta t}u(c_t)\,dt
+
e^{-\beta T}U_B(X_T^{c,\pi})
\bigg]
\le
\mathfrak L(y,{\cal Z})+yx.
\]
\end{prop}
Consequently,
\begin{equation}\label{eq:weak_duality_finite}
V(x)\le \inf_{y>0}\big(J(y)+yx\big).
\end{equation}
Strong duality is established in Section~\ref{sec:verification-duality}. The same argument, applied conditionally on $\mathcal F_t$ to the continuation problem on $[t,T]$ with multipliers in $\Pi_t(0)$, gives ${\cal W}_t\le\mathfrak J(t,\xi,0)+\xi X_t^{c,\pi}$ almost surely for every $\xi>0$, where $\mathfrak J$ is the dynamic dual value defined in \eqref{eq:three_dim_singular_finite} below. This conditional form is used in Corollary~\ref{cor:feasible-wealth-domain}.

\subsection{Dynamic formulation of the dual problem}

To study the dual problem dynamically, define, for $(t,y,z)\in[0,T]\times(0,\infty)\times[0,\infty)$,
\begin{equation}\label{eq:Y_tsy}
{\cal H}_s^t:=\frac{{\cal H}_s}{{\cal H}_t},
\qquad
{\cal Y}_s^{t,y}:=y e^{\beta(s-t)}{\cal H}_s^t,
\qquad t\le s\le T.
\end{equation}
Then ${\cal Y}^{t,y}$ solves
\[
\frac{d{\cal Y}_s^{t,y}}{{\cal Y}_s^{t,y}}
=
(\beta-r)\,ds-\theta\,dB_s,
\qquad
{\cal Y}_t^{t,y}=y.
\]

Let $\Pi_t(z)$ denote the set of all finite valued, nonnegative, nondecreasing, c\`adl\`ag, $\mathbb F$-adapted processes ${\cal Z}=({\cal Z}_s)_{t\le s\le T}$ such that the value before intervention is ${\cal Z}_{t-}=z$ and $\mathbb E[(1+\mathcal Z_T)^m]<\infty$ for every $m>1$. The moment condition is the same as in the definition of $\Pi(z)$.  This convention permits a jump $\Delta{\cal Z}_t={\cal Z}_t-z\ge0$.
Define
\begin{equation}\label{eq:three_dim_singular_finite}
\begin{aligned}
\mathfrak{J}(t,y,z)
:=
\inf_{{\cal Z}\in \Pi_t(z)}
\mathbb E_t\bigg[
&\int_t^T e^{-\beta(s-t)}
\bigg(
(1+{\cal Z}_s)\tilde u\!\left(\frac{{\cal Y}_s^{t,y}}{1+{\cal Z}_s}\right)
+
{\cal Y}_s^{t,y}\epsilon
\bigg)\,ds
\\
&\qquad
+
e^{-\beta(T-t)}(1+{\cal Z}_T)\tilde U_B\!\left(\frac{{\cal Y}_T^{t,y}}{1+{\cal Z}_T}\right)
-
\int_{[t,T]} e^{-\beta(s-t)}P(s)\,d{\cal Z}_s
\bigg].
\end{aligned}
\end{equation}
By construction,
\[
J(y)=\mathfrak{J}(0,y,0).
\]
The dual problem is nonstationary because of the finite horizon and the time dependent welfare floor.

\subsection{Dimension reduction and reduced variational inequality}

Let
\[
{\cal D}:=[0,T)\times(0,\infty)\times[0,\infty).
\]
If $\mathfrak{J}$ is sufficiently smooth, then dynamic programming suggests that it should satisfy
\begin{equation}\label{eq:HJB_finite}
\min\bigg\{
\partial_t \mathfrak{J}
+\frac{\theta^2}{2}y^2\partial_{yy}\mathfrak{J}
+(\beta-r)y\partial_y\mathfrak{J}
-\beta \mathfrak{J}
+(1+z)\tilde u\!\left(\frac{y}{1+z}\right)
+y\epsilon,\,
\partial_z\mathfrak{J}-P(t)
\bigg\}
=0
\end{equation}
on ${\cal D}$, with terminal condition
\begin{equation}\label{eq:HJB_terminal_finite}
\mathfrak{J}(T,y,z)
=
(1+z)\tilde U_B\!\left(\frac{y}{1+z}\right),
\qquad (y,z)\in(0,\infty)\times[0,\infty).
\end{equation}

To reduce the state dimension, set $\xi:=y/(1+z)$.

\begin{prop}\label{pro:dimension_reduction}
Define
\[
Q(t,\xi):=\mathfrak{J}(t,\xi,0),
\qquad (t,\xi)\in[0,T]\times(0,\infty).
\]
Then
\begin{equation}\label{eq:J_equals_Q_reduction}
\mathfrak{J}(t,y,z)=(1+z)\,Q\!\left(t,\frac{y}{1+z}\right)
\end{equation}
for every $(t,y,z)\in[0,T]\times(0,\infty)\times[0,\infty)$.
\end{prop}

Define the reduced differential operator
\begin{equation}\label{eq:LQ_operator}
{\cal L} \phi
:= 
\frac{\theta^2}{2}\xi^2 \partial_{\xi\xi}\phi
+(\beta-r)\xi \partial_\xi \phi
-\beta \phi.
\end{equation}
If $Q$ is sufficiently smooth, then differentiating \eqref{eq:J_equals_Q_reduction} gives
\[
\partial_t \mathfrak{J}=(1+z)Q_t,
\qquad
\partial_y \mathfrak{J}=Q_\xi,
\qquad
\partial_{yy}\mathfrak{J}=\frac{1}{1+z}Q_{\xi\xi},
\qquad
\partial_z \mathfrak{J}=Q-\xi Q_\xi.
\]
Substituting these identities into \eqref{eq:HJB_finite}, we obtain the reduced variational inequality
\begin{equation}\label{eq:HJB_Q_reduced}
\min\bigg\{
\partial_t Q(t,\xi)+{\cal L} Q(t,\xi)+\tilde u(\xi)+\epsilon\xi,\,
Q(t,\xi)-\xi Q_\xi(t,\xi)-P(t)
\bigg\}
=0
\end{equation}
for $(t,\xi)\in[0,T)\times(0,\infty)$, with terminal condition
\begin{equation}\label{eq:Q_terminal_condition}
Q(T,\xi)=\tilde U_B(\xi),
\qquad \xi\in(0,\infty).
\end{equation}

Define
\begin{align}
N(t,\xi):=Q(t,\xi)-\xi Q_\xi(t,\xi)-P(t),
\qquad
(t,\xi)\in [0,T]\times(0,\infty),
\label{eq:N_def}
\end{align}
so that $N\ge0$ by \eqref{eq:HJB_Q_reduced}. The waiting and increasing regions in the reduced $(t,\xi)$-space are
\[
{\bf WR}^{\rm dual}
:=
\big\{(t,\xi)\in[0,T)\times(0,\infty): N(t,\xi)>0\big\},
\quad
{\bf IR}^{\rm dual}
:=
\big\{(t,\xi)\in[0,T)\times(0,\infty): N(t,\xi)=0\big\}.
\]
On ${\bf WR}^{\rm dual}$ the first term in \eqref{eq:HJB_Q_reduced} vanishes, and a direct computation using $Q-\xi Q_\xi=N+P$ yields
\begin{align}
\partial_t N(t,\xi)+{\cal L}N(t,\xi)-\xi \tilde u'(\xi)+\tilde u(\xi)
+P'(t)-\beta P(t)=0
\label{eq:N_PDE}
\end{align}
there. Under the CRRA specification, $-\xi\tilde u'(\xi)+\tilde u(\xi)=\frac{1}{1-\gamma}\xi^{1-\frac1\gamma}$, so the formal obstacle problem for $N$ is
\begin{align}
\max\bigg\{
\partial_t N(t,\xi)+{\cal L}N(t,\xi)
+\frac{1}{1-\gamma}\xi^{1-\frac{1}{\gamma}}
+P'(t)-\beta P(t),\,
-N(t,\xi)
\bigg\}
=0,
\qquad
(t,\xi)\in [0,T)\times(0,\infty),
\label{eq:N_obstacle_problem_CRRA}
\end{align}
and the terminal condition $N(T,\xi)=\tilde U_B(\xi)-\xi\tilde U_B'(\xi)-P(T)$ reads
\begin{align}
N(T,\xi)
=
\begin{cases}
\displaystyle
\frac{\kappa_B^{\frac{1-\gamma}{\gamma}}}{1-\gamma}\,
\xi^{1-\frac{1}{\gamma}}-P(T),
& \text{if } \xi\le \kappa_B^{1-\gamma}\underline{x}_B^{-\gamma},
\\[2ex]
\displaystyle
0,
& \text{if } \xi> \kappa_B^{1-\gamma}\underline{x}_B^{-\gamma},
\end{cases}
\qquad
\xi\in(0,\infty).
\label{eq:N_terminal_condition_CRRA}
\end{align}

\section{Parabolic Obstacle Problem with a Time Dependent Obstacle}\label{sec:parabolic-obstacle}
\subsection{Obstacle problem and transformation}

For the analysis of the free boundary, we reverse time and introduce the logarithmic
spatial variable by setting
\[
    \tau:=T-t,\qquad v:=\ln\xi,\qquad
    p(\tau):=P(T-\tau),
\]
for $(\tau,v)\in\Omega_T:=(0,T)\times\mathbb R$. Formally, $R(\tau,v):=N(T-\tau,e^v)+p(\tau)$ satisfies the obstacle problem \eqref{VI:R}. From here on, $R$ denotes the unique solution of \eqref{VI:R} provided by Lemma~\ref{lem:existence-R}, and $N(t,\xi):=R(T-t,\ln\xi)-P(t)$. Theorem~\ref{thm:dual-verification} shows that this $N$ coincides with the function in \eqref{eq:N_def}. Define
\[
    \mathscr L
    :=
    \frac{\theta^2}{2}\partial_{vv}
    +\left(\beta-r-\frac{\theta^2}{2}\right)\partial_v
    -\beta,
\]
and set
\[
    f(v):=\frac{1}{1-\gamma}e^{(1-\frac1\gamma)v},
    \qquad
    B(v):=
    \frac{\kappa_B^{\frac{1-\gamma}{\gamma}}}{1-\gamma}
    e^{(1-\frac1\gamma)v}.
\]
The obstacle problem for $R\ge p$ is
\begin{equation}\label{VI:R}
    \begin{cases}
        \partial_\tau R-\mathscr LR=f(v)
        &\quad \text{if }R>p\ \text{ in }\Omega_T,\\[0.5ex]
        \partial_\tau R-\mathscr LR\ge f(v)
        &\quad \text{if }R=p\ \text{ in }\Omega_T,\\[0.5ex]
        R(0,v)=\max\{B(v),p(0)\}
        &\quad \text{for }v\in\mathbb R.
    \end{cases}
\end{equation}
The function $B$ is strictly decreasing. In particular, if
\[
    v_B
    :=
    \ln\!\left(\kappa_B^{1-\gamma}\underline x_B^{-\gamma}\right),
\]
then $B(v_B)=p(0)$, and hence
\[
    \{v\in\mathbb R:R(0,v)=p(0)\}=[v_B,\infty).
\]

The unbounded spatial domain requires a growth class. Put
$a:=1-1/\gamma$ and $b_0:=\beta-r-\theta^2/2$. Let
$L_s^v=v+b_0s+\theta B_s^\circ$, where $B^\circ$ is a Brownian motion, and let
$\mathcal S_\tau$ be the stopping times taking values in $[0,\tau]$.

\begin{lem}\label{lem:existence-R}
There is a unique solution
\[
R\in W_{\infty,\mathrm{loc}}^{1,2}(\Omega_T)\cap C(\overline{\Omega_T})
\]
of \eqref{VI:R} in the class
\begin{equation}\label{eq:R-growth-class}
|R(\tau,v)|\le C(1+e^{av}),\qquad (\tau,v)\in[0,T]\times\mathbb R,
\end{equation}
for a finite constant $C$. It has the representation
\begin{equation}\label{eq:R-stopping-representation}
\begin{split}
R(\tau,v)=\sup_{\rho\in\mathcal S_\tau}\mathbb E\bigg[&
\int_0^\rho e^{-\beta s}f(L_s^v)\,ds
+e^{-\beta\rho}p(\tau-\rho)\mathbf1_{\{\rho<\tau\}}\\
&+e^{-\beta\tau}\max\{B(L_\tau^v),p(0)\}
\mathbf1_{\{\rho=\tau\}}\bigg].
\end{split}
\end{equation}

Finally, the following comparison principle holds. Fix $2\le q<\infty$. Let $\mathcal N$ be $(s_0,s_1)\times(v_0,\infty)$, $(s_0,s_1)\times\mathbb R$, or a bounded cylinder $(s_0,s_1)\times(v_1,v_2)$ with $0\le s_0<s_1\le T$, and let $\zeta\in W^{1,2}_{q,\mathrm{loc}}(\mathcal N)\cap C(\overline{\mathcal N})$ satisfy \eqref{eq:R-growth-class} on $\mathcal N$. If $\zeta\ge p$ and $(\partial_\tau-\mathscr L)\zeta\ge f$ a.e. in $\mathcal N$ and $\zeta\ge R$ on the parabolic boundary of $\mathcal N$, then $\zeta\ge R$ in $\mathcal N$. If $(\partial_\tau-\mathscr L)\zeta\le f$ a.e. in $\mathcal N\cap\{\zeta>p\}$ and $\zeta\le R$ on the parabolic boundary, then $\zeta\le R$ in $\mathcal N$.
\end{lem}
For $\gamma>1$ the solution is bounded, $p\le R\le0$, so the growth class is restrictive only when $\gamma<1$, where $a<0$ and $R$ may grow as $v\to-\infty$.
The representation also shows that extending $p$ to a slightly larger
time interval does not change $R$ at earlier times.
This is used in Remark~\ref{rem:endpoint-T} to treat the endpoint $\tau=T$.

\subsection{Analysis of the free boundary}
Assumption~\ref{as:P_finite} gives $p\in C^{1,1}([0,T])$. We first characterize the contact region, whose structure depends on the range of $\gamma$, and then parametrize it by a single threshold using monotonicity in the state variable. Finally we bound the time derivative of $R$ by its space derivative. This estimate yields quadratic growth away from the free boundary, local $1/2$-H\"older continuity of the threshold, and the identification of its graph with the free boundary.

\subsubsection{Investigation of the contact region}
We begin with a pointwise necessary condition for the contact between
the solution $R$ and its obstacle $p$. The contact mechanism differs
according to the range of $\gamma$. When $\gamma>1$, the source term
$f$ tends to $-\infty$ as $v\to+\infty$. Since
$p'(\tau)+\beta p(\tau)$ is bounded, for sufficiently large $v$ we have
\[
    f(v)<p'(\tau)+\beta p(\tau)
\]
uniformly in $\tau$. This allows us to construct a supersolution
touching the obstacle and thereby establish contact at large states.

When $\gamma\in(0,1)$, on the other hand, $f(v)\to0^+$ as
$v\to+\infty$. Hence, for any sufficiently small $\lambda_0>0$,
we have $f(v)<\lambda_0$ for all sufficiently large $v$. We therefore
seek a temporal auxiliary function $\rho_\tau$ satisfying
\[
    \rho_\tau(\tau)=p(\tau),\qquad
    \rho_\tau(s)\ge p(s)\quad (0\le s\le\tau),\qquad
    \rho_\tau'(s)+\beta\rho_\tau(s)=\lambda_0.
\]
Such a function can be constructed for sufficiently small $\lambda_0>0$
when $p'(\tau)+\beta p(\tau)>0$ and $e^{\beta\tau}p(\tau)$ is strictly larger than all of its previous
values, that is, $e^{\beta\tau}p(\tau)>e^{\beta s}p(s)$ for every $0\le s<\tau$. We refer to the latter condition as the running record condition.

\begin{lem}
\label{lem:necessary-contact-R}
Suppose first that $\gamma\in (0,1)$, and define
$\Gamma:(0,T)\to\mathbb R\cup\{+\infty\}$ by
\[
    \Gamma(\tau)
    :=
    \begin{cases}
        \displaystyle
        \frac{\gamma}{\gamma-1}
        \ln\Bigl[(1-\gamma)
        \bigl\{p'(\tau)+\beta p(\tau)\bigr\}\Bigr],
        &\text{if }p'(\tau)+\beta p(\tau)>0,\\[0.5em]
        +\infty,
        &\text{if }p'(\tau)+\beta p(\tau)\le0.
    \end{cases}
\]
Then $\{R=p\}\subset
\{(\tau,v)\in\Omega_T:v>\Gamma(\tau)\}$.

Suppose next that $\gamma\in (1,\infty)$, and define
$\Gamma:(0,T)\to\mathbb R\cup\{-\infty\}$ by
\[
    \Gamma(\tau)
    :=
    \begin{cases}
        \displaystyle
        \frac{\gamma}{\gamma-1}
        \ln\Bigl[(1-\gamma)
        \bigl\{p'(\tau)+\beta p(\tau)\bigr\}\Bigr],
        &\text{if }p'(\tau)+\beta p(\tau)<0,\\[0.5em]
        -\infty,
        &\text{if }p'(\tau)+\beta p(\tau)\ge0.
    \end{cases}
\]
Then $\{R=p\}\subset
\{(\tau,v)\in\Omega_T:v>\Gamma(\tau)\}$.
\end{lem}
\subsubsection{The case \texorpdfstring{$\gamma\in (0,1)$}{0<gamma<1}}

For $\gamma\in(0,1)$, contact is governed by the local
condition in Lemma~\ref{lem:necessary-contact-R} together with the
running record condition introduced above. The economic
interpretation is
that utility is positive, so the continuation value at any date
strictly exceeds the discounted value $e^{-\beta(s-t)}P(s)$ of every
later floor. The floor can bind at $t$ only if $e^{-\beta t}P(t)$
exceeds all its later values and is decreasing at $t$.

\begin{lem}
\label{lem:strict-record-condition}
Suppose that $\gamma\in (0,1)$, and define
$q(\tau):=e^{\beta\tau}p(\tau)$ for $0\le\tau\le T$. If
$R(\tau,v)=p(\tau)$ for some $(\tau,v)\in\Omega_T$, then
\[q'(\tau)>0 
\quad \text{and} \quad q(\tau)>q(s) 
\quad \text{for every }\ 0\le s<\tau.\]
\end{lem}
Define the active time set by
\[
    \mathcal T_0
    :=
    \bigl\{\tau\in(0,T):q'(\tau)>0,\ 
    q(\tau)>q(s)\ \text{ for every }\ 0\le s<\tau\bigr\}.
\]

\begin{cor}
\label{cor:no-contact-outside-T0}
Suppose that $\gamma\in (0,1)$. For every
$\tau\in(0,T)\setminus\mathcal T_0$, we have
$R(\tau,v)>p(\tau)$ for every $v\in\mathbb R$.
\end{cor}
This is the contrapositive of Lemma~\ref{lem:strict-record-condition}.
\begin{lem}
\label{lem:T0-open}
The set $\mathcal T_0$ is open.
\end{lem}
We next prove the converse of
Lemma~\ref{lem:strict-record-condition}. In fact, contact occurs
locally uniformly in time for all sufficiently large states.

\begin{lem}
\label{lem:contact-existence-gamma-less-one}
Suppose that $\gamma\in (0,1)$. For every
$\tau_0\in\mathcal T_0$, there exist
$\nu=\nu(\tau_0)>0$ and $V=V(\tau_0)\in\mathbb R$ such that
$R(\tau,v)=p(\tau)$ whenever $|\tau-\tau_0|<\nu$ and $v\ge V$.
\end{lem}
Combining Corollary~\ref{cor:no-contact-outside-T0} and
Lemma~\ref{lem:contact-existence-gamma-less-one}, we obtain
\[
    \bigl\{\tau\in(0,T):R(\tau,v)=p(\tau)
    \text{ for some }v\in\mathbb R\bigr\}=\mathcal T_0.
\]

\subsubsection{The case \texorpdfstring{$\gamma\in (1,\infty)$}{gamma>1}}

The necessary condition for contact
in Lemma~\ref{lem:necessary-contact-R} shows that, at each
time \(\tau\), contact can occur only for \(v>\Gamma(\tau)\). Thus,
when \(\Gamma(\tau)\) is finite, the region \(v\le\Gamma(\tau)\) is
automatically contained in the noncontact set. Unlike in the case
\(\gamma\in(0,1)\), however, \(\Gamma(\tau)\) may equal \(-\infty\)
when \(\gamma\in(1,\infty)\), in which case this condition gives no
information on the noncontact set. We therefore establish its
nonemptiness on each time slice directly by constructing an auxiliary
function and applying the comparison principle.

\begin{lem}
\label{lem:bounds-gamma-greater-one}
Suppose that $\gamma\in(1,\infty)$. Then $R\le0$ in $\Omega_T$.
Moreover, there exists $v_*\in\mathbb R$ such that
\[
    (-\infty,v_*]\subset\{v\in\mathbb R:R(\tau,v)>p(\tau)\}
    \qquad\text{for every }\tau\in[0,T].
\]
\end{lem}
When $\gamma\in (1,\infty)$, the source term satisfies
$f(v)\to-\infty$ as $v\to+\infty$. This guarantees contact with
the obstacle at every interior time.

\begin{lem}
\label{lem:contact-existence-gamma-greater-one}
Suppose that $\gamma\in (1,\infty)$. There exists $V_0\in\mathbb R$ such that
\[
    R(\tau,v)=p(\tau)
    \qquad
    \text{for all }\ (\tau,v)\in[0,T]\times[V_0,\infty).
\]
\end{lem}
\subsubsection{State monotonicity and parametrization of the free boundary}

\begin{lem}
\label{lem:state-monotonicity-R}
Let $R$ be the solution of \eqref{VI:R}. Then
$\partial_vR\le0$ in $\Omega_T$, and
$\partial_vR<0$ in $\{R>p\}$.
\end{lem}
Corollary~\ref{cor:no-contact-outside-T0} and
Lemmas~\ref{lem:contact-existence-gamma-less-one} and
\ref{lem:contact-existence-gamma-greater-one}
identify the time slices on which the contact set
is nonempty. By the state monotonicity established in
Lemma~\ref{lem:state-monotonicity-R}, each nonempty contact slice is
determined by a single threshold. We now introduce this threshold. After establishing its continuity, we will show that its graph coincides with the free boundary.
Accordingly, define
\[
    \mathcal T_\gamma
    :=
    \begin{cases}
        \mathcal T_0,&\gamma\in (0,1),\\
        (0,T),&\gamma\in (1,\infty),
    \end{cases}
\]
and for each $\tau\in\mathcal T_\gamma$,
\begin{equation}\label{V}
    \mathcal V(\tau)
    :=
    \inf\bigl\{v\in\mathbb R:R(\tau,v)=p(\tau)\bigr\}.
\end{equation}

\begin{thm}
\label{thm:contact-geometry-R}
Let $R$ be the solution of \eqref{VI:R}.

\begin{enumerate}
    \item For every $\tau\in(0,T)$, the contact set is nonempty
    at time $\tau$ if and only if $\tau\in\mathcal T_\gamma$.

    \item For every $\tau\in\mathcal T_\gamma$, 
    $\mathcal V(\tau)$ defined by \eqref{V} is finite, and
    \[
        \{v\in\mathbb R:R(\tau,v)>p(\tau)\}
        =(-\infty,\mathcal V(\tau)),
        \qquad
        \{v\in\mathbb R:R(\tau,v)=p(\tau)\}
        =[\mathcal V(\tau),\infty).
    \]

    \item The function $\mathcal{V}$ satisfies
    $\mathcal V(\tau)>\Gamma(\tau)$ for every
    $\tau\in\mathcal T_\gamma$.
\end{enumerate}
\end{thm}
Consequently, 
\begin{equation}\label{non-contact}
  \{(\tau,v)\in \mathcal T_\gamma\times\mathbb R: R(\tau,v)>p(\tau)\}
=\{(\tau,v)\in \mathcal T_\gamma\times\mathbb R: v<\mathcal{V}(\tau)\},  
\end{equation}
and this characterization gives the following result for $\mathcal{V}$.

\begin{lem}\label{lem:free-boundary-lsc}
Let $\mathcal V:\mathcal T_\gamma\to\mathbb R$ be the parametrization
defined in \eqref{V}. Then $\mathcal V$ is lower semicontinuous on
$\mathcal T_\gamma$.
\end{lem}
Lower semicontinuity makes $\{(\tau,v):v<\mathcal V(\tau)\}$ open.
Continuity also makes $\{(\tau,v):v>\mathcal V(\tau)\}$ open, so the free boundary is exactly the graph of $\mathcal V$.
We establish this continuity by proving local $1/2$-H\"older regularity.

\subsubsection{Regularity of the free boundary}
\begin{lem}\label{lem:time-space-derivative-estimate}
Let $\tau_0\in\mathcal T_\gamma$. Then there exist $\kappa>0$,
$\underline v<\overline v$, and $C>0$ such that
$[\tau_0-\kappa,\tau_0+\kappa]\subset\mathcal T_\gamma$,
\[
    \underline v<\mathcal V(\tau)<\overline v
    \qquad\text{for }|\tau-\tau_0|\le\kappa,
\]
and
\[
    \bigl|\partial_\tau R(\tau,v)-p'(\tau)\bigr|
    \le
    C\bigl(-\partial_vR(\tau,v)\bigr)
\]
for every
    $(\tau,v)\in
    (\tau_0-\kappa,\tau_0+\kappa)
    \times(\underline v,\overline v)
    \cap\{R>p\}$.
\end{lem}
\begin{lem}\label{lem:free-boundary-holder}
Let $\mathcal V:\mathcal T_\gamma\to\mathbb R$ be the parametrization
defined in \eqref{V}. Then
\[
    \mathcal V\in C_{\rm loc}^{1/2}(\mathcal T_\gamma).
\]
Consequently,
\[
    \partial\{R>p\}
    =
    \bigl\{(\tau,\mathcal V(\tau)):\tau\in\mathcal T_\gamma\bigr\}
    \qquad\text{in }\Omega_T.
\]
\end{lem}
\begin{prop}\label{prop:boundary-higher-regularity}
The free boundary is locally Lipschitz on $\mathcal T_\gamma$ and,
more precisely,
\[
\mathcal V\in C_{\mathrm{loc}}^{1,\alpha}(\mathcal T_\gamma)
\qquad\text{for every }0<\alpha<\tfrac12.
\]
If, in addition, $p\in C^{m+1}((0,T))$ for some integer
$m\ge2$, then $\mathcal V\in C_{\mathrm{loc}}^{m,\alpha}(\mathcal T_\gamma)$
for every $0<\alpha<\tfrac12$. In particular, if $p\in C^\infty((0,T))$, then
$\mathcal V\in C^\infty(\mathcal T_\gamma)$.
These statements are local on $\mathcal T_\gamma$ and do not extend
across endpoints where $\mathcal V$ diverges.
\end{prop}

\subsubsection{Endpoint behavior of the free boundary}

We finally describe the behavior of the free boundary at the endpoints
of its time domain. We extend $\Gamma$ to $\tau=0$ by the formulas of
Lemma~\ref{lem:necessary-contact-R} with $\tau=0$, so that
$\Gamma(0)\in\mathbb R\cup\{+\infty\}$ if $\gamma\in(0,1)$ and
$\Gamma(0)\in\mathbb R\cup\{-\infty\}$ if $\gamma\in(1,\infty)$.
For $\gamma\in (0,1)$, we use the convention
$\mathcal V(\tau)=+\infty$ when $\tau\notin\mathcal T_0$.

\begin{lem}\label{lem:boundary-T0}
Suppose that $\gamma\in (0,1)$. For every
$\tau_*\in\partial\mathcal T_0\cap(0,T)$,
\[
    \lim_{\substack{\tau\to\tau_*\\ \tau\in\mathcal T_0}}
    \mathcal V(\tau)=+\infty.
\]
\end{lem}
\begin{lem}\label{lem:endpoint-zero}
The parametrization of the free boundary $\mathcal{V}$ satisfies
\[
    \lim_{\tau\to 0^+}\mathcal V(\tau)
    =
    \max\{v_B,\Gamma(0)\},
\]
where the equality is understood in the extended real sense.
\end{lem}
\begin{rem}\label{rem:endpoint-T}
The results of this section obtained under Assumption~\ref{as:P_finite} alone also hold at the endpoint $\tau=T$.
Extend $p$ to $[0,T+\delta]$ for some $\delta>0$ so that the extension belongs to $C^{1,1}([0,T+\delta])$ and satisfies $(1-\gamma)p>0$.
By \eqref{eq:R-stopping-representation}, $R(\tau,\cdot)$ depends only on the restriction of $p$ to $[0,\tau]$, so the solution of the extended problem coincides with $R$ on $[0,T]\times\mathbb R$.
Since $\tau=T$ is an interior time for the extended problem, the corresponding local results apply there.
In particular, $\partial_vR(T,v)<0$ whenever $R(T,v)>p(T)$, the contact set at $\tau=T$ is $[\mathcal V(T),\infty)$ when it is nonempty, and $\mathcal V$ is continuous at $T$ in that case.
The $C^{m,\alpha}$ and $C^\infty$ conclusions of Proposition~\ref{prop:boundary-higher-regularity} also hold at $T$, whenever $\mathcal V(T)$ is finite, provided that $p$ admits a $C^{m+1}$ or $C^\infty$ extension, respectively, to a neighborhood of $T$.
For $\gamma\in(0,1)$, whether $T$ belongs to the extended set $\mathcal T_0$ depends only on $q'(T)$ and on $q$ restricted to $[0,T]$, and hence not on the extension.
If $T$ is a boundary point of the extended set $\mathcal T_0$, Lemma~\ref{lem:boundary-T0} gives $\mathcal V(\tau)\to+\infty$ as $\tau\to T$.
\end{rem}

\subsection{Examples of free boundary geometry}
\label{subsec:illustrative-free-boundaries}

We briefly illustrate how the geometry of the free boundary depends on
risk aversion and on the time profile of the welfare floor.  Let
\[
s:=\frac{T-t}{T}\in[0,1],
\qquad
P_i(t):=P(T)e^{-\beta(T-t)}m_i(s),
\]
where
\[
\begin{aligned}
m_1(s)&=1+0.55s,\\
m_2(s)&=1+1.20s-s^2,\\
m_3(s)&=1+8\left(\frac{s^3}{3}-0.40s^2+0.1375s\right).
\end{aligned}
\]
Thus $P_i(T)=P(T)=U_B(\underline x_B)$ in every case.  Under the calibration below, $P(T)\simeq0.4472$ for $\gamma=0.5$ and $P(T)=-20$ for $\gamma=2$.  The numerical
illustrations use $T=20$, $\beta=0.04$, $r=0.02$, $\mu=0.06$,
$\sigma=0.20$, $\epsilon=0.50$, $\kappa_B=1$, $\underline x_B=0.05$,
and $\gamma\in\{0.5,2\}$.

The three profiles separate the two cases of risk aversion.
If $\gamma>1$, the contact set is nonempty at every
interior date, so the free boundary is finite throughout $(0,T)$ even
for the nonmonotone profile $P_2$.  If $\gamma\in(0,1)$, the
set $\mathcal T_0$ is determined by $q'>0$ together with the running record condition for
$q(\tau)=e^{\beta\tau}p(\tau)$.  Under profile $P_1$ every date is active because
$q$ is strictly increasing.  For $P_2$,
$q$ increases only for $0<\tau<0.6T$, which gives a finite boundary for
$0.4T<t<T$.  For $P_3$, $q$ first increases, then decreases, and
regains its previous record at $\tau=0.7T$. Hence
\[
\mathcal T_0=(0,0.25T)\cup(0.7T,T),
\]
which corresponds in calendar time to the two windows
$t\in(0,0.3T)\cup(0.75T,T)$ with a finite boundary.  At the interior endpoints where these windows close or open again, $\mathcal V$ diverges to $+\infty$ by Lemma~\ref{lem:boundary-T0}. The limit at the terminal endpoint is given by Lemma~\ref{lem:endpoint-zero}.  The dashed
curve in Figure~\ref{fig:illustrative-free-boundaries} is the local
necessary bound $\Gamma$ from Lemma~\ref{lem:necessary-contact-R}.
For $\gamma<1$ the condition $\mathcal V>\Gamma$ is
necessary, but the running record condition decides whether a finite
boundary exists at all.

\begin{figure}[!t]
\centering
\begin{tabular}{cc}
\includegraphics[width=0.475\textwidth]{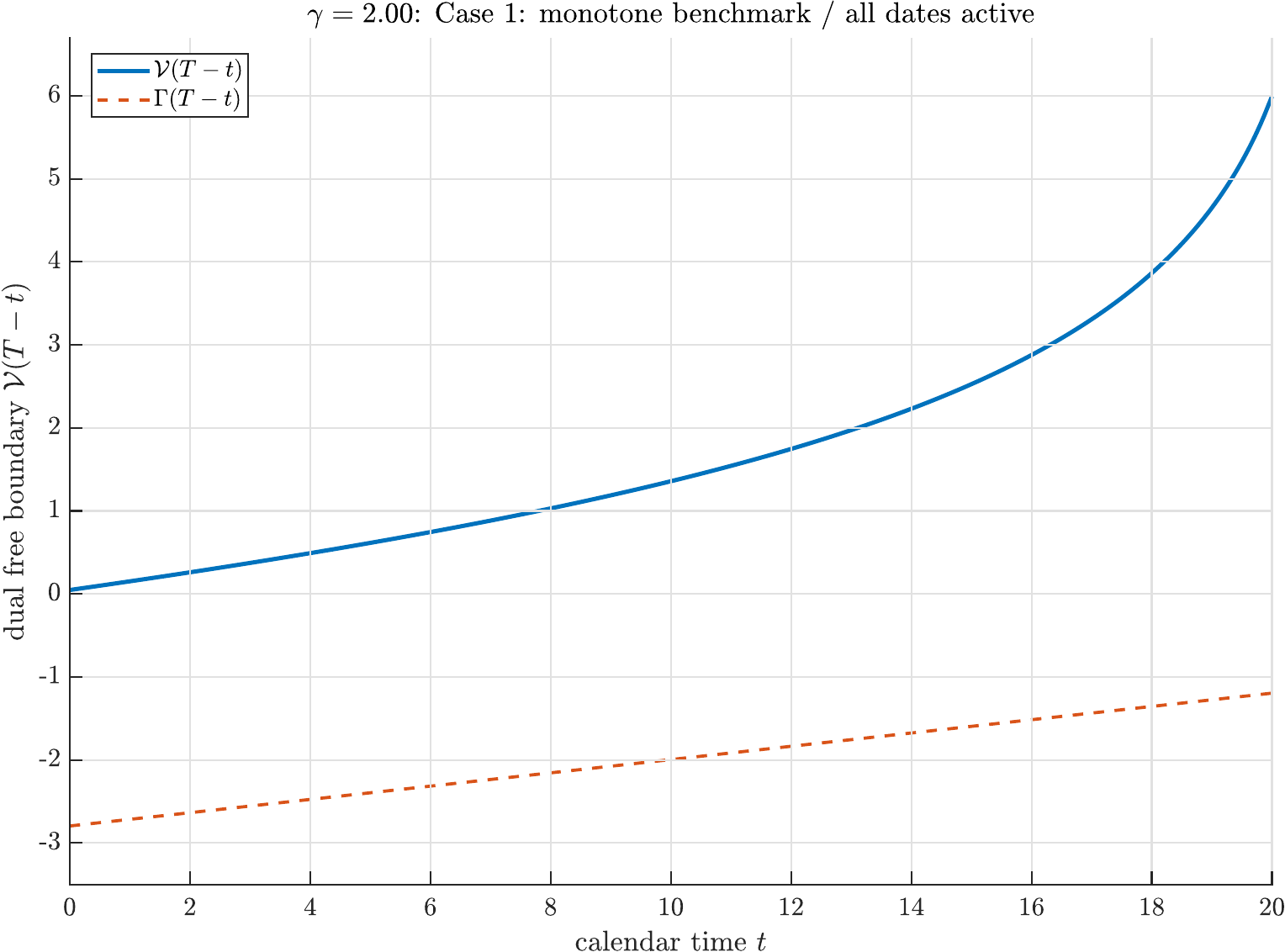}
&
\includegraphics[width=0.475\textwidth]{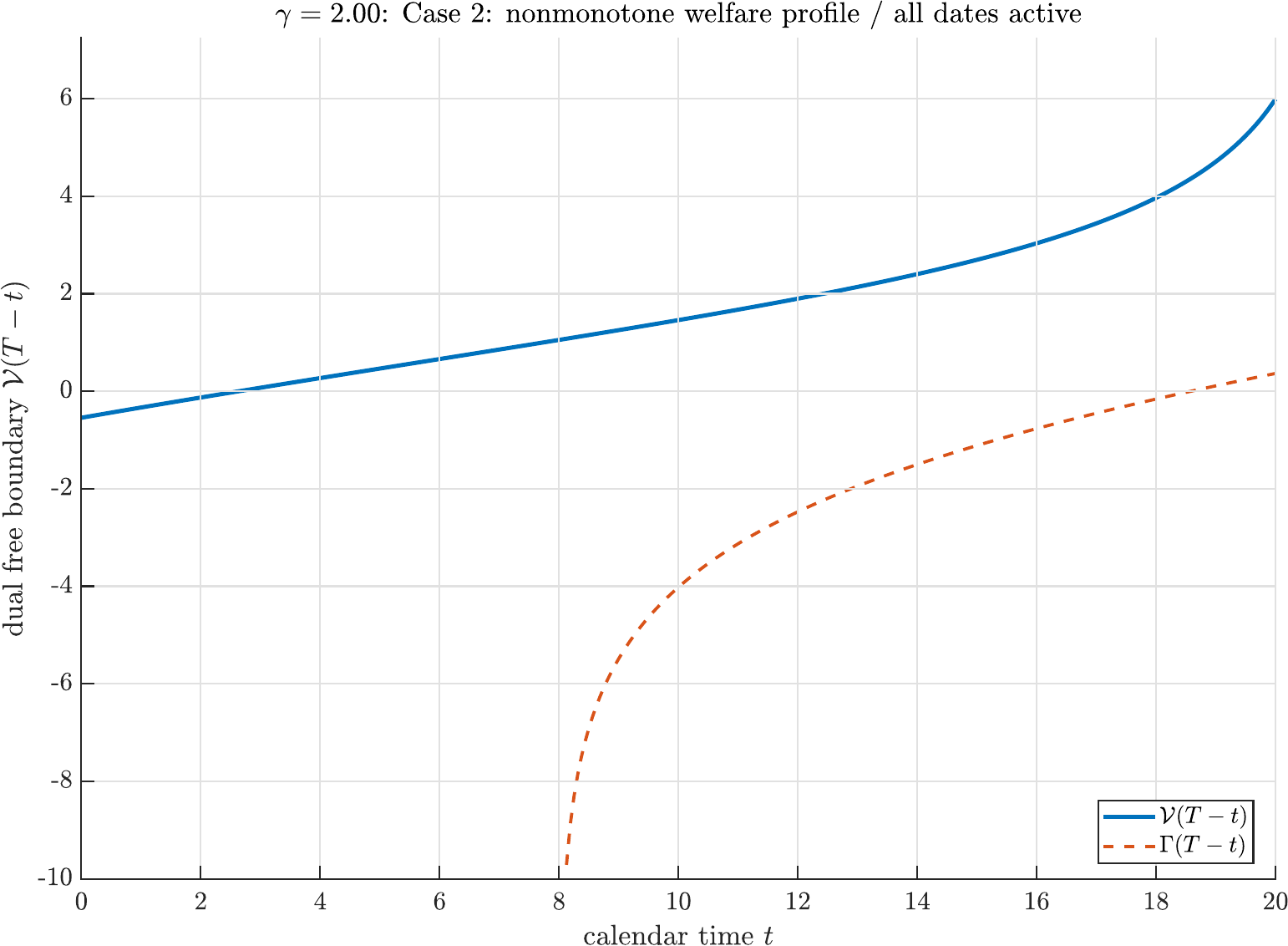}
\\[-0.4em]
{\small (a) $\gamma=2$, profile $P_1$}
&
{\small (b) $\gamma=2$, profile $P_2$}
\\[0.4em]
\includegraphics[width=0.475\textwidth]{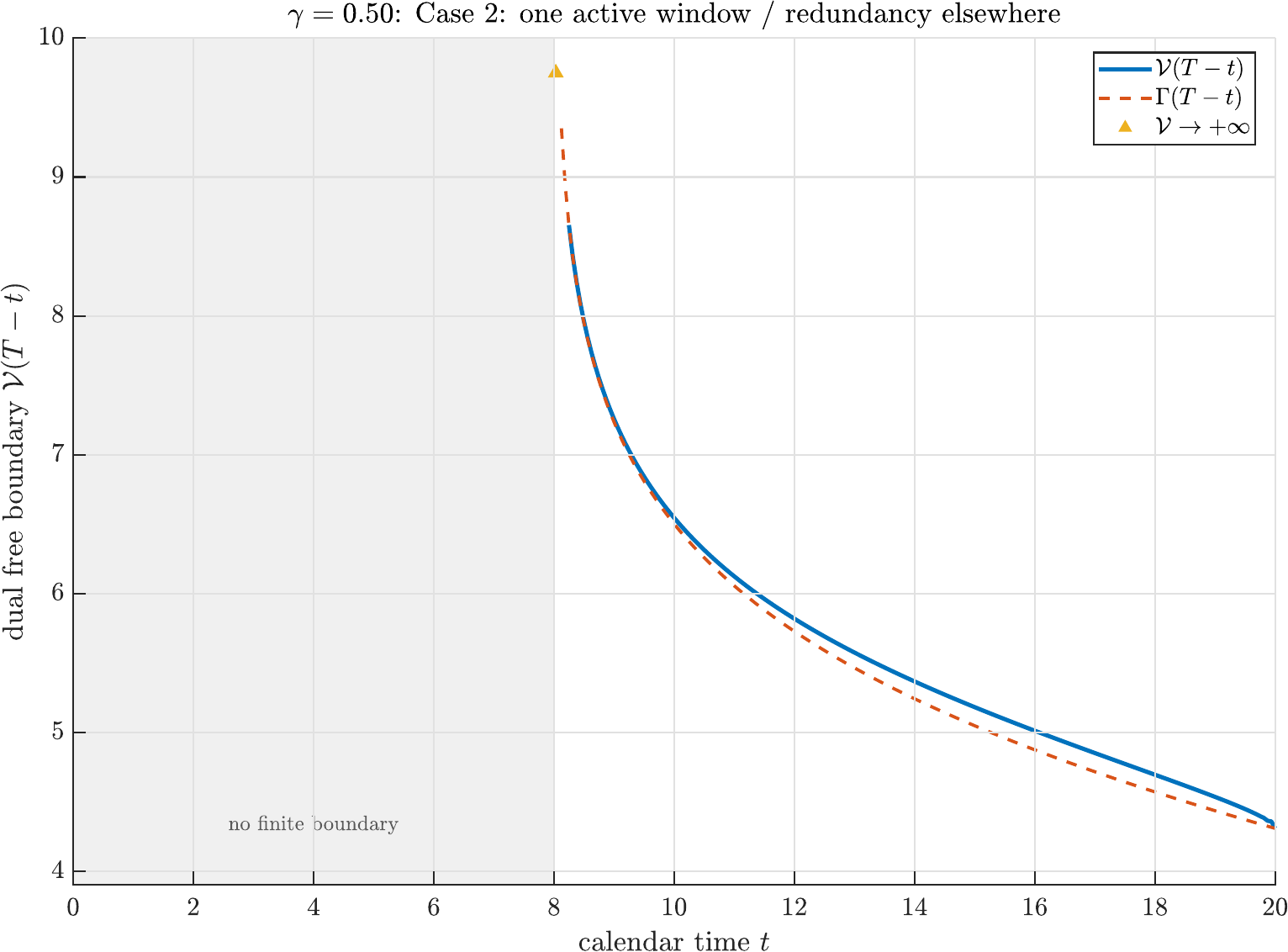}
&
\includegraphics[width=0.475\textwidth]{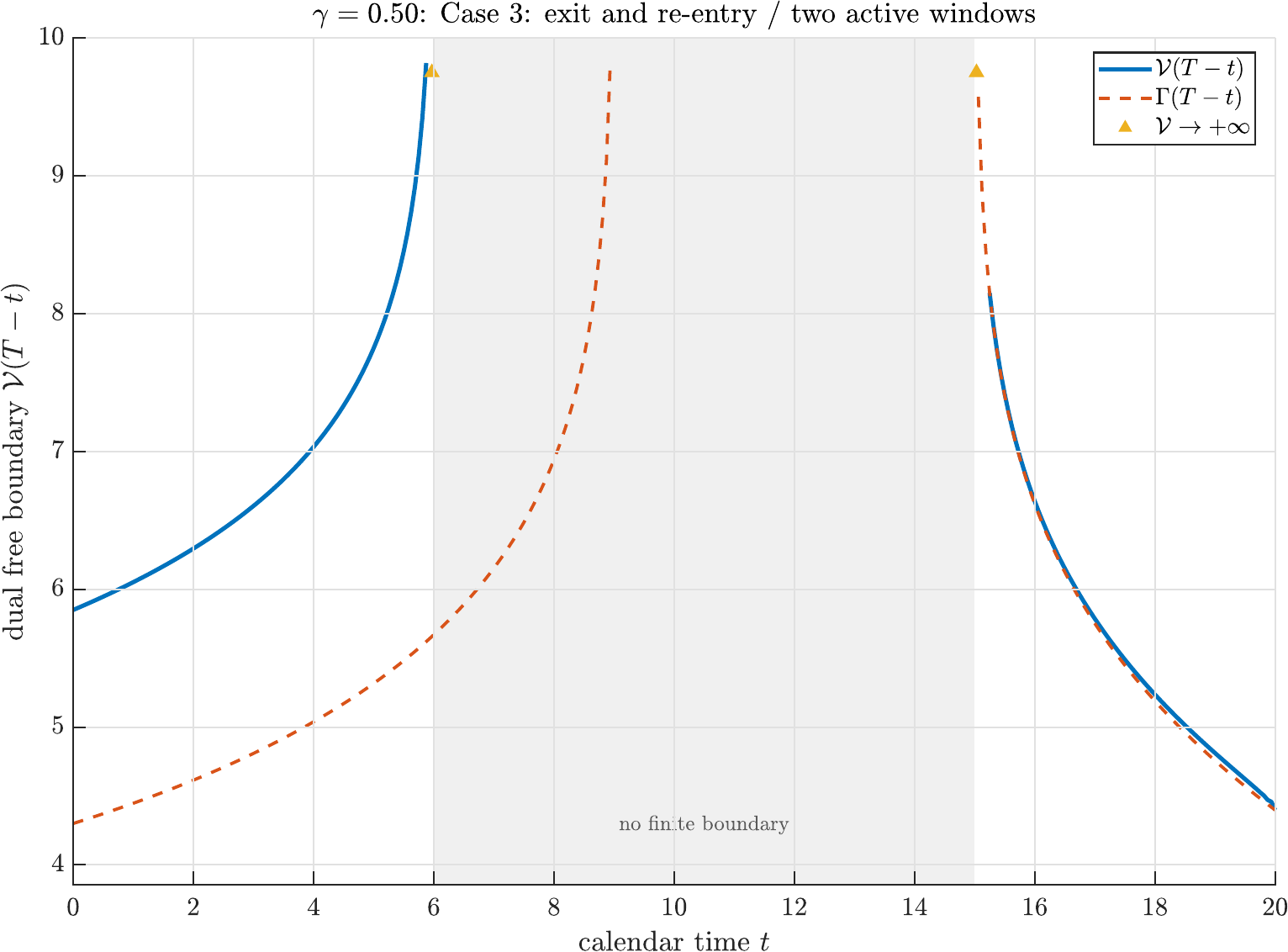}
\\[-0.4em]
{\small (c) $\gamma=0.5$, profile $P_2$}
&
{\small (d) $\gamma=0.5$, profile $P_3$}
\end{tabular}
\caption{Free boundaries under the three welfare
profiles.  Solid curves are $\mathcal V(T-t)$ and dashed curves are
$\Gamma(T-t)$.  Shaded intervals indicate dates with no finite free
boundary, and triangles mark the divergence of $\mathcal V$ at the endpoints
of the active time set.}
\label{fig:illustrative-free-boundaries}
\end{figure}

\section{Verification and Duality}\label{sec:verification-duality}

We now verify the optimal cumulative multiplier, reconstruct the primal controls, and establish strong duality using the analytic properties obtained from the obstacle problem.

\subsection{Construction of the dual value and intervention boundary}

Define
\begin{equation}\label{eq:G-def-verification}
G(t,\xi):=R(T-t,\log\xi)=N(t,\xi)+P(t).
\end{equation}
The obstacle solution determines the dual value, including the linear term
which the identity $Q-\xi Q_\xi=G$ alone leaves undetermined. Set
\begin{equation}\label{eq:unconstrained-coefficients}
\begin{split}
a&:=1-\frac1\gamma,\qquad
\lambda:=\frac{\theta^2}{2}a^2+b_0a-\beta,\\
K(t)&:=\int_0^{T-t}e^{\lambda s}\,ds
+\kappa_B^{(1-\gamma)/\gamma}e^{\lambda(T-t)},\\
G^0(t,\xi)&:=\frac{K(t)}{1-\gamma}\xi^a,\qquad
Q^0(t,\xi):=\frac{\gamma K(t)}{1-\gamma}\xi^a+h(t)\xi,\qquad
D(t,\xi):=G(t,\xi)-G^0(t,\xi).
\end{split}
\end{equation}
Here $Q^0$ is the dual value with neither interim nor terminal welfare
constraints, so that $Q^0-\xi Q^0_\xi=G^0$.

\begin{prop}\label{prop:construct-Q}
The function $D$ is nonnegative and, for every $m>0$,
\begin{equation}\label{eq:D-small-state}
\sup_{0\le t\le T}D(t,\xi)\le C_m\xi^m,\qquad 0<\xi\le1.
\end{equation}
Consequently the formula
\begin{equation}\label{eq:Q-constructive}
\widehat Q(t,\xi)
:=Q^0(t,\xi)-\xi\int_0^\xi\frac{D(t,\eta)}{\eta^2}\,d\eta
\end{equation}
is well defined. It satisfies \eqref{eq:HJB_Q_reduced} and
\eqref{eq:Q_terminal_condition}, is continuous up to $T$, and belongs to
$C^{1,2}([0,T)\times(0,\infty))$, with one sided time derivatives at $0$.
Furthermore,
\begin{equation}\label{eq:Q-reconstruction-identities}
\widehat Q-\xi\widehat Q_\xi=G,\qquad
\widehat Q_{\xi\xi}=-\frac{G_\xi}{\xi}\ge0,
\end{equation}
and, uniformly in $t\in[0,T]$,
\begin{equation}\label{eq:Q-growth-explicit}
|\widehat Q(t,\xi)|+\xi|\widehat Q_\xi(t,\xi)|
\le C(1+\xi+\xi^a).
\end{equation}
\end{prop}
We use $\widehat Q$ as the verification candidate. Theorem~\ref{thm:dual-verification}
identifies it with the reduced value $Q$ defined in Section~\ref{sec:optimization}.

Define the terminal dual kink
\begin{equation}\label{eq:qB-terminal}
q_B:=\kappa_B^{1-\gamma}\underline x_B^{-\gamma}=e^{v_B}.
\end{equation}
For $0\le t<T$, set
\begin{equation}\label{eq:dual-threshold-b-exact}
b(t):=\inf\{\xi>0:N(t,\xi)=0\},\qquad\inf\varnothing:=+\infty,
\end{equation}
and set
\begin{equation}\label{eq:dual-threshold-terminal}
b(T):=q_B.
\end{equation}
For interior dates,
\begin{equation}\label{eq:dual-threshold-from-V}
b(t)=
\begin{cases}
\exp\{\mathcal V(T-t)\},&T-t\in\mathcal T_\gamma,\\
+\infty,&T-t\notin\mathcal T_\gamma.
\end{cases}
\end{equation}
At $t=0$ the boundary is determined by $R(T,\cdot)$. By Remark~\ref{rem:endpoint-T}, \eqref{eq:dual-threshold-from-V} also holds at $t=0$ with the extended set $\mathcal T_\gamma$. The left limit of $b$ at $T$ may exceed
$q_B$, which permits a terminal multiplier jump enforcing $X_T\ge\underline x_B$.

\subsection{The reflected cumulative multiplier}

Fix $(t,y,z)\in[0,T]\times(0,\infty)\times[0,\infty)$ and write
$A_{t-}:=1+z$. A ratio with denominator $+\infty$ is interpreted as zero.
Define
\begin{equation}\label{eq:A-star-reflection}
 A_s^{*,t,y,z}
 :=(1+z)\vee
   \sup_{t\le u\le s}
   \frac{\mathcal Y_u^{t,y}}{b(u)},
 \qquad t\le s\le T,
\end{equation}
and
\begin{equation}\label{eq:Z-star-reflection} 
 \mathcal Z_s^{*,t,y,z}:=A_s^{*,t,y,z}-1,
 \qquad
 \Xi_s^{*,t,y,z}
 :=\frac{\mathcal Y_s^{t,y}}{A_s^{*,t,y,z}}. 
\end{equation}
The running supremum is right continuous and gives the smallest nondecreasing multiplier that keeps the normalized dual state below $b$.

\begin{lem}\label{lem:reflection-moments}
There is a deterministic $b_*>0$ such that $b(t)\ge b_*$ for every
$t\in[0,T]$, with the convention of the extended real line. The reciprocal $1/b$ is
continuous on $[0,T)$, has a left limit at $T$, and can jump upwards only at
$T$. The process in \eqref{eq:A-star-reflection} satisfies
\begin{equation}\label{eq:A-moment-bound}
A_s^{*,t,y,z}\le 1+z+b_*^{-1}\sup_{t\le u\le T}\mathcal Y_u^{t,y},
\qquad
\mathbb E\big[(A_T^{*,t,y,z})^m\big]<\infty\quad(m>0).
\end{equation}
All positive and negative moments of the associated consumption and all
positive moments of terminal wealth are finite, uniformly over time for
each fixed initial state. Moreover,
\begin{equation}\label{eq:boundary-zero-occupation}
\int_t^T\mathbf1_{\{\Xi_s^{*,t,y,z}=b(s)<\infty\}}\,ds=0
\qquad\text{a.s.}
\end{equation}
\end{lem}

\begin{prop}
\label{prop:dual-shadow-reflection}
Under Assumption~\ref{as:P_finite},
$\mathcal Z^{*,t,y,z}\in\Pi_t(z)$ and
\begin{equation}\label{eq:state-constraint-reflection}
 \Xi_s^{*,t,y,z}\le b(s),\qquad t\le s\le T.
\end{equation}
Moreover, $d\mathcal Z^{*,t,y,z}$ is carried by the set on which the reflected
state is on the intervention boundary:
\begin{equation}\label{eq:Skorokhod-support}
 \int_{[t,T]}
 \mathbf 1_{\{\Xi_s^{*,t,y,z}<b(s)\}}
 \,d\mathcal Z_s^{*,t,y,z}=0
\end{equation}
and
\begin{equation}\label{eq:N-complementarity-exact}
 \int_{[t,T]}
 N\bigl(s,\Xi_s^{*,t,y,z}\bigr)
 \,d\mathcal Z_s^{*,t,y,z}=0.
\end{equation}
At a jump time, including $t$ or $T$, the state after intervention equals the
current boundary whenever the jump is nonzero.
\end{prop}
For the original problem we abbreviate
\begin{equation}\label{eq:Z-star-y-abbrev}
 \mathcal Z^{*,y}:=\mathcal Z^{*,0,y,0},\qquad
 A_t^y:=1+\mathcal Z_t^{*,y},\qquad
 \Xi_t^y:=\frac{\mathcal Y_t^y}{A_t^y}.
\end{equation}
In particular,
\begin{equation}\label{eq:Z-star-y-explicit}
 \mathcal Z_t^{*,y}
 =\left(1\vee\sup_{0\le s\le t}
       \frac{\mathcal Y_s^y}{b(s)}\right)-1.
\end{equation}

\subsection{Dual verification}

\begin{thm}\label{thm:dual-verification}
Under Assumption~\ref{as:P_finite}, the cumulative multiplier $\mathcal Z^{*,t,y,z}$ defined in \eqref{eq:Z-star-reflection} is optimal for the singular control
problem \eqref{eq:three_dim_singular_finite}.  More precisely,
\begin{equation}\label{eq:dual-verification-value}
 \mathfrak J(t,y,z)
 =(1+z)Q\!\left(t,\frac{y}{1+z}\right),
\end{equation}
and the infimum is attained by
$\mathcal Z^{*,t,y,z}$.  In particular,
\begin{equation}\label{eq:J-Q-initial}
 J(y)=Q(0,y)=\mathfrak L\bigl(y,\mathcal Z^{*,y}\bigr).
\end{equation}
\end{thm}
\subsection{Primal variables generated by the cumulative multiplier}

For $y>0$, define the pointwise Fenchel optimizers
\begin{equation}\label{eq:cstar-xTstar}
 c_t^y:=I(\Xi_t^y)=(\Xi_t^y)^{-1/\gamma},
 \qquad
 X_T^y:=\hat x_B(\Xi_T^y)
 =\max\left\{\underline x_B,
 \kappa_B^{\frac{1-\gamma}{\gamma}}(\Xi_T^y)^{-1/\gamma}
 \right\}.
\end{equation}
Define their continuation utility by
\begin{equation}\label{eq:Wstar-y}
 \mathcal W_t^y
 :=\mathbb E_t\!\left[
 \int_t^T e^{-\beta(s-t)}u(c_s^y)\,ds
 +e^{-\beta(T-t)}U_B(X_T^y)
 \right].
\end{equation}

\begin{prop}\label{prop:shadow-welfare-identity}
For every $y>0$,
\begin{equation}\label{eq:shadow-welfare-identity}
 \mathcal W_t^y
 =G(t,\Xi_t^y)
 =Q(t,\Xi_t^y)-\Xi_t^yQ_\xi(t,\Xi_t^y)
 =N(t,\Xi_t^y)+P(t),
 \qquad 0\le t\le T.
\end{equation}
Consequently, $\mathcal W_t^y\ge P(t)$ for all $t$, and
\begin{equation}\label{eq:welfare-complementarity}
 \int_{[0,T]}e^{-\beta t}
 \bigl(\mathcal W_t^y-P(t)\bigr)
 \,d\mathcal Z_t^{*,y}=0.
\end{equation}
Thus the cumulative multiplier increases only when the welfare constraint binds.
\end{prop}
\subsection{The map to wealth and the endogenous credit limit}
\label{subsec:dynamic-credit-limit}

To map the dual boundary into wealth space, restart the verified problem at $t\in[0,T]$ and, for $\xi>0$, set
\begin{equation}\label{eq:dynamic-reflected-objects}
 \mathcal Z^{*,t,\xi}:=\mathcal Z^{*,t,\xi,0},\qquad
 \Xi_s^{t,\xi}:=
 \frac{\mathcal Y_s^{t,\xi}}{1+\mathcal Z_s^{*,t,\xi}},
 \qquad t\le s\le T,
\end{equation}
and define
\begin{equation}\label{eq:dynamic-primal-objects}
 c_s^{t,\xi}:=I(\Xi_s^{t,\xi}),\qquad
 X_T^{t,\xi}:=\hat x_B(\Xi_T^{t,\xi}).
\end{equation}
The cost of these claims under the state price density, measured at time $t$, is
\begin{equation}\label{eq:dynamic-budget-map}
 \mathcal B_t(\xi)
 :=\mathbb E_t\!\left[
 \int_t^T \mathcal H_s^t(c_s^{t,\xi}-\epsilon)\,ds
 +\mathcal H_T^tX_T^{t,\xi}
 \right].
\end{equation}
The Markov structure makes $\mathcal B_t(\xi)$ deterministic. Define
\begin{equation}\label{eq:wealth-dual-map}
 x(t,\xi):=-Q_\xi(t,\xi)
\end{equation}
as the image of the normalized dual state $\xi$ in wealth space.

\begin{prop}
\label{prop:dynamic-credit-limit}
Under Assumption~\ref{as:P_finite},
for every $t\in[0,T]$ and $\xi>0$,
\begin{equation}\label{eq:dynamic-budget-envelope}
 \mathcal B_t(\xi)=x(t,\xi)=-Q_\xi(t,\xi).
\end{equation}
The map $\xi\mapsto x(t,\xi)$ is continuous and nonincreasing.  Define the
endogenous time-$t$ credit limit by
\begin{equation}\label{eq:dynamic-credit-limit}
 \underline x(t):=
 \begin{cases}
 -Q_\xi(t,b(t)),& b(t)<+\infty,\\[1ex]
 \displaystyle\lim_{\xi\to\infty}\bigl[-Q_\xi(t,\xi)\bigr],
 & b(t)=+\infty.
 \end{cases}
\end{equation}
Then the following assertions hold.
\begin{enumerate}
\item[(i)] If $t<T$ and $b(t)<+\infty$, then
\begin{equation}\label{eq:credit-limit-collapse}
 x(t,\xi)=\underline x(t),\qquad \xi\ge b(t),
\end{equation}
whereas $x(t,\cdot)$ is strictly decreasing on $(0,b(t))$ and
\begin{equation}\label{eq:credit-limit-boundary-limit}
 \lim_{\xi\to b(t)^-}x(t,\xi)=\underline x(t).
\end{equation}
Thus $\xi=b(t)$ maps to the primal credit limit $x=\underline x(t)$.

\item[(ii)] If $t<T$ and $b(t)=+\infty$, then $x(t,\cdot)$ is strictly
decreasing on $(0,\infty)$ and
\begin{equation}\label{eq:credit-limit-infinite-boundary}
 x(t,\xi)\to\underline x(t)
 \qquad\text{as }\xi\to\infty.
\end{equation}
Hence an infinite dual boundary corresponds to a limiting wealth boundary.

\item[(iii)] At terminal time,
\begin{equation}\label{eq:terminal-wealth-dual-map}
 x(T,\xi)=\hat x_B(\xi),\qquad
 \underline x(T)=\underline x_B,
\end{equation}
and the terminal dual boundary $b(T)=q_B$ maps to the terminal wealth floor
$\underline x_B$.

\item[(iv)] For every $t\in[0,T]$,
\begin{equation}\label{eq:dynamic-credit-limit-lower-bound}
 \underline x(t)
 \ge \ell(t)
 :=-h(t)+e^{-r(T-t)}\underline x_B
 >-h(t).
\end{equation}
If no welfare intervention occurs after $t$ and before the horizon, in the sense that
$b(s)=+\infty$ for every $s\in[t,T)$, then
\begin{equation}\label{eq:credit-limit-no-future-intervention}
 \underline x(t)=\ell(t)
 =-h(t)+e^{-r(T-t)}\underline x_B.
\end{equation}
Thus the absence of an interior welfare boundary does not generally restore the natural debt limit $-h(t)$, because $X_T\ge\underline x_B$ remains binding at the horizon.
\end{enumerate}
Write
\begin{equation}\label{eq:initial-credit-limit-abbrev}
 \underline x:=\underline x(0).
\end{equation}
\end{prop}
\begin{cor}
\label{cor:feasible-wealth-domain}
Let $\mathcal A_t(x)$ denote the continuation admissible class on $[t,T]$
obtained from $\mathcal A(x)$ by starting the wealth process from $X_t=x$ and
imposing the welfare constraints on $[t,T]$.  Under Assumption~\ref{as:P_finite},
\begin{equation}\label{eq:feasible-domain-endpoint}
 \inf\{x:\mathcal A_t(x)\neq\varnothing\}=\underline x(t).
\end{equation}
No admissible continuation exists for $x<\underline x(t)$, while each $x>\underline x(t)$ is generated by a unique $\xi\in(0,b(t))$ (interpreted as $(0,\infty)$ if $b(t)=+\infty$). If $t<T$ and $b(t)<\infty$, the endpoint is also attainable, and the feasible set is $[\underline x(t),\infty)$. If $b(t)=\infty$, \eqref{eq:feasible-domain-endpoint} identifies the infimum only. At $T$, the feasible set is $[\underline x_B,\infty)$. In particular, $V(x)=-\infty$ for $x<\underline x(0)$, the endpoint $x=\underline x(0)$ is covered by the preceding assertions, and Theorem~\ref{thm:strong-duality} below treats $x>\underline x(0)$.
\end{cor}
\subsection{Initial budget map and the scalar dual multiplier}

For the original problem at time $0$, define
\begin{equation}\label{eq:budget-map-B}
 \mathcal B(y)
 :=\mathbb E\!\left[
 \int_0^T \mathcal H_t(c_t^y-\epsilon)\,dt
 +\mathcal H_T X_T^y
 \right]
 =\mathcal B_0(y).
\end{equation}
By Proposition~\ref{prop:dynamic-credit-limit},
\begin{equation}\label{eq:initial-budget-Qxi}
 \mathcal B(y)=-Q_\xi(0,y).
\end{equation}

\begin{prop}
\label{prop:dual-envelope-budget}
The dual value $J(y)=Q(0,y)$ is continuously differentiable and
\begin{equation}\label{eq:Jprime-budget}
 J'(y)=-\mathcal B(y).
\end{equation}
Moreover, $\mathcal B$ is continuous and nonincreasing, is strictly decreasing
whenever $y<b(0)$, and
\begin{equation}\label{eq:endogenous-credit-limit}
 \lim_{y\to\infty}\mathcal B(y)=\underline x(0)=:\underline x,
\end{equation}
with
\begin{equation}\label{eq:credit-limit-lower-bound}
 \underline x\ge-h(0)+e^{-rT}\underline x_B>-h(0).
\end{equation}
Furthermore,
\begin{equation}\label{eq:budget-map-zero-limit}
 \mathcal B(y)\longrightarrow+\infty
 \qquad\text{as }y\to0^+.
\end{equation}
Consequently, for every $x>\underline x$ there exists a unique
$y^*=y^*(x)>0$ satisfying
\begin{equation}\label{eq:y-star-budget}
 \mathcal B(y^*)=x,
\end{equation}
or equivalently
\begin{equation}\label{eq:y-star-budget-FOC}
 J'(y^*)+x=0.
\end{equation}
This $y^*$ is the global minimizer of $J(y)+yx$.
\end{prop}
\subsection{Wealth and portfolio reconstruction}

Fix $x>\underline x$ and choose $y^*$ according to
\eqref{eq:y-star-budget}.  Write
\begin{equation}\label{eq:optimal-primal-from-y}
 c^*:=c^{y^*},\qquad X_T^*:=X_T^{y^*},\qquad
 \mathcal Z^*:=\mathcal Z^{*,y^*},\qquad
 \Xi^*:=\Xi^{y^*}.
\end{equation}
Define
\begin{equation}\label{eq:replicating-wealth}
 X_t^*
 :=\frac{1}{\mathcal H_t}
 \mathbb E_t\!\left[
 \int_t^T \mathcal H_s(c_s^*-\epsilon)\,ds
 +\mathcal H_T X_T^*
 \right],
 \qquad 0\le t\le T.
\end{equation}
By \eqref{eq:y-star-budget}, $X_0^*=x$.

\begin{prop}
\label{prop:wealth-credit-reconstruction}
For the verified primal and dual pair,
\begin{equation}\label{eq:wealth-Qxi-representation}
 X_t^*=-Q_\xi(t,\Xi_t^*)=x(t,\Xi_t^*),
 \qquad 0\le t\le T.
\end{equation}
\begin{equation}\label{eq:optimal-wealth-credit-limit}
 X_t^*\ge\underline x(t)\ge
 -h(t)+e^{-r(T-t)}\underline x_B>-h(t),
 \qquad 0\le t<T,
\end{equation}
with $X_T^*\ge\underline x_B=\underline x(T)$.  Moreover,
\begin{equation}\label{eq:wealth-credit-complementarity}
 \int_{[0,T]}
 \bigl(X_t^*-\underline x(t)\bigr)\,d\mathcal Z_t^*=0.
\end{equation}
Thus the cumulative multiplier acts only at the endogenous credit limit $X_t^*=\underline x(t)$, corresponding to $\Xi_t^*=b(t)$.
\end{prop}
In particular,
\begin{equation}\label{eq:natural-debt-reconstruction}
 X_t^*+h(t)
 =\frac{1}{\mathcal H_t}
 \mathbb E_t\!\left[
 \int_t^T \mathcal H_s c_s^*\,ds
 +\mathcal H_T X_T^*
 \right]>0,
\end{equation}
so the original natural borrowing constraint is satisfied strictly.

To identify the portfolio, let
\begin{equation}\label{eq:martingale-reconstruction}
 M_t^X
 :=\mathbb E_t\!\left[
 \int_0^T\mathcal H_s(c_s^*-\epsilon)\,ds
 +\mathcal H_T X_T^*
 \right].
\end{equation}
Then
\begin{equation}\label{eq:MX-HX}
 M_t^X
 =\mathcal H_tX_t^*
 +\int_0^t\mathcal H_s(c_s^*-\epsilon)\,ds.
\end{equation}
Since the market filtration is Brownian, the martingale representation theorem
gives a progressively measurable process $\varphi$ such that
\begin{equation}\label{eq:MX-representation}
 dM_t^X=\varphi_t\,dB_t.
\end{equation}
Define
\begin{equation}\label{eq:pi-star-explicit}
 \pi_t^*
 :=\frac{1}{\sigma}
 \left(\frac{\varphi_t}{\mathcal H_t}+\theta X_t^*\right).
\end{equation}
It\^o's formula with $d\mathcal H_t=-r\mathcal H_tdt-\theta\mathcal H_t dB_t$ then gives \eqref{eq:wealth_finite}. Equivalently, comparison of diffusion coefficients in
$X_t^*=-Q_\xi(t,\Xi_t^*)$ yields
\begin{equation}\label{eq:portfolio-feedback-baseline}
\pi_t^*=
\frac{\theta}{\sigma}\Xi_t^*Q_{\xi\xi}(t,\Xi_t^*)
=-\frac{\theta}{\sigma}N_\xi(t,\Xi_t^*)
\quad dt\otimes d\mathbb P\text{-a.e.}
\end{equation}
The continuous feedback extension equals zero at a finite credit boundary
by smooth fit. Lemma~\ref{lem:reflection-moments} and the martingale
representation establish portfolio admissibility in the proof below.

\subsection{Strong duality and primal optimality}

\begin{thm}\label{thm:strong-duality}
Let $x>\underline x=\underline x(0)$ and choose $y^*>0$ satisfying
\eqref{eq:y-star-budget}.  Then
$(c^*,\pi^*)\in\mathcal A(x)$ and is optimal for
Problem~\ref{pr:main_finite}.  The budget constraint binds, the welfare
constraint satisfies \eqref{eq:welfare-complementarity}, the optimal wealth satisfies
Proposition~\ref{prop:wealth-credit-reconstruction}, and
\begin{equation}\label{eq:strong-duality-final}
 V(x)
 =J(y^*)+y^*x
 =\inf_{y>0}\{J(y)+yx\}.
\end{equation}
The optimal consumption and terminal wealth are
\begin{equation}\label{eq:optimal-controls-final}
 c_t^*=I\!\left(
 \frac{\mathcal Y_t^{y^*}}{1+\mathcal Z_t^*}
 \right),
 \qquad
 X_T^*=\hat x_B\!\left(
 \frac{\mathcal Y_T^{y^*}}{1+\mathcal Z_T^*}
 \right),
\end{equation}
where the cumulative multiplier is given explicitly by
\begin{equation}\label{eq:Z-star-final}
 \mathcal Z_t^*
 =\left(1\vee\sup_{0\le s\le t}
 \frac{\mathcal Y_s^{y^*}}{b(s)}\right)-1. 
\end{equation}
The optimal consumption and terminal wealth, and hence the optimal wealth and portfolio processes, are unique.
\end{thm}

\section{Stochastic Income and Limited Commitment}\label{sec:implications}

This section develops the limited commitment interpretation of the welfare floor described in the introduction. We show that an outside option based on labor income preserves the reduction to a one dimensional free boundary problem and turns the welfare boundary into an endogenous credit limit.

\subsection{Stochastic income and an outside option based on income}

The welfare constraint has a direct interpretation in terms of household credit
once labor income is stochastic and the reservation value is tied to the
agent's current income.  We follow the complete market formulation of
\citet{ChoiKooLimYoo2024}.  Let labor income satisfy
\begin{equation}\label{eq:income-gbm-imp}
\frac{dI_t}{I_t}
=\mu_I\,dt+\sigma_I\Bigl(\rho_I\,dB_t^1
+\sqrt{1-\rho_I^2}\,dB_t^2\Bigr),
\qquad I_0>0,
\end{equation}
where the stock is driven by $B^1$.  As in \citet{ChoiKooLimYoo2024}, we
assume that a traded contract spans the $B^2$-component of income risk with unit diffusion exposure and zero excess drift. Thus the market is complete, and its vector of market prices of risk is $(\theta,0)$.  If $f_t$ denotes the dollar exposure to this contract, the
wealth dynamics become
\begin{equation}\label{eq:wealth-stoch-income-imp}
dX_t=\bigl(rX_t+\pi_t(\mu-r)-c_t+I_t\bigr)dt
+\sigma\pi_t\,dB_t^1+f_t\,dB_t^2.
\end{equation}

Under CRRA preferences, a consumption plan proportional to income generates a
welfare value proportional to $I_t^{1-\gamma}$ and not to
$I_t$ itself.  The homogeneous counterpart of a restriction of the form ``$P(t)$ times current
income'' is therefore
\begin{equation}\label{eq:income-scaled-welfare-imp}
\mathcal W_t\ge \mathcal O_t=I_t^{1-\gamma}P_\delta(t).
\end{equation}
This is the scaling compatible with CRRA utility and with the reduction to the wealth to income ratio.

A natural specification takes $\mathcal O_t$ to be the value of an
outside option under which the agent consumes a fraction $\delta\in(0,1]$ of
labor income after exit and receives the terminal resource $\underline w_B I_T$ from the outside arrangement,
\begin{equation}\label{eq:outside-value-imp}
\mathcal O_t
:=\mathbb E_t\!\left[
\int_t^T e^{-\beta(s-t)}u(\delta I_s)\,ds
+e^{-\beta(T-t)}u(\kappa_B\underline w_B I_T)
\right].
\end{equation}
This is the finite horizon analogue of the reservation utility based on income in
\citet{ChoiKooLimYoo2024}.  When $0<\gamma<1$, both the agent's
continuation utility and the outside value are positive, so the welfare
constraint has economic content and does not hold trivially because of the sign
of CRRA utility.

Set
\begin{equation}\label{eq:transformed-beta-imp}
\widehat\beta
:=\beta-(1-\gamma)\mu_I
+\frac{\gamma(1-\gamma)}{2}\sigma_I^2.
\end{equation}
Then \eqref{eq:outside-value-imp} has the homogeneous form
\eqref{eq:income-scaled-welfare-imp}, where, with $\tau=T-t$,
\begin{equation}\label{eq:Pdelta-imp}
P_\delta(t)
=\frac{1}{1-\gamma}
\left[
\delta^{1-\gamma}\frac{1-e^{-\widehat\beta\tau}}{\widehat\beta}
+(\kappa_B\underline w_B)^{1-\gamma}e^{-\widehat\beta\tau}
\right],
\end{equation}
with the usual limiting interpretation when $\widehat\beta=0$.  Notice that
$(1-\gamma)P_\delta(t)>0$ for every $\gamma>0$, $\gamma\neq1$.

\subsection{Reduction to the baseline problem}

Define the wealth and consumption ratios
\[
w_t:=\frac{X_t}{I_t},\qquad m_t:=\frac{c_t}{I_t},
\]
and the portfolio coordinates net of income hedging
\[
\widetilde\pi_t
:=\frac{\pi_t}{I_t}-\frac{\sigma_I\rho_I}{\sigma}w_t,
\qquad
\widetilde f_t
:=\frac{f_t}{I_t}-\sigma_I\sqrt{1-\rho_I^2}\,w_t.
\]
Write $\boldsymbol\eta=(\sigma_I\rho_I,\sigma_I\sqrt{1-\rho_I^2})$ and
define the equivalent probability measure
\begin{equation}\label{eq:income-measure-change}
\frac{d\widehat{\mathbb P}}{d\mathbb P}\bigg|_{\mathcal F_t}
=\exp\!\left((1-\gamma)\boldsymbol\eta\cdot\boldsymbol B_t
-\frac12(1-\gamma)^2\sigma_I^2t\right).
\end{equation}
Then $\widehat{\boldsymbol B}_t=\boldsymbol B_t-(1-\gamma)\boldsymbol\eta t$
is Brownian under $\widehat{\mathbb P}$. For any admissible ratio plan,
conditional change of measure gives
\begin{equation}\label{eq:income-utility-reduction}
\frac{\mathcal W_t}{I_t^{1-\gamma}}
=\widehat{\mathbb E}_t\!\left[
\int_t^T e^{-\widehat\beta(s-t)}u(m_s)\,ds
+e^{-\widehat\beta(T-t)}U_B(w_T)\right].
\end{equation}
With $\boldsymbol e_t=(\sigma\widetilde\pi_t,\widetilde f_t)$ and
$\boldsymbol\vartheta=(\theta-\gamma\sigma_I\rho_I,
-\gamma\sigma_I\sqrt{1-\rho_I^2})$, It\^o's formula yields
\begin{equation}\label{eq:income-ratio-dynamics}
dw_t=\{r_Iw_t+\boldsymbol\vartheta\cdot\boldsymbol e_t-m_t+1\}\,dt
+\boldsymbol e_t\cdot d\widehat{\boldsymbol B}_t.
\end{equation}
The transformed parameters are
\begin{equation}\label{eq:transformed-parameters-imp}
\begin{aligned}
r_I&:=r-\mu_I+\theta\sigma_I\rho_I,\\
\Theta_\gamma
&:=\left[(\theta-\gamma\sigma_I\rho_I)^2
+\gamma^2\sigma_I^2(1-\rho_I^2)\right]^{1/2}.
\end{aligned}
\end{equation}
\begin{prop}\label{prop:income-reduction}
Assume $\widehat\beta>0$ and $\Theta_\gamma>0$. Call $(c,\pi,f)$ admissible if $(m,\widetilde\pi,\widetilde f)$ is progressively measurable with $\int_0^T(m_t+\widetilde\pi_t^2+\widetilde f_t^2)\,dt<\infty$ a.s., the integrability condition \eqref{eq:admissible_finite_2} holds for $(m,w_T)$ under $\widehat{\mathbb P}$ with discount rate $\widehat\beta$, $\mathcal W_t\ge\mathcal O_t$ for all $t$, and $X_t>-h_I(t)I_t$ for all $t$, with $h_I$ as in \eqref{eq:natural-credit-ratio-imp} below. Then, in the coordinates $(w,m,\widetilde\pi,\widetilde f)$ and under $\widehat{\mathbb P}$, the problem is the baseline problem of Sections~\ref{sec:model}--\ref{sec:verification-duality} with the substitutions $\epsilon\to1$, $\beta\to\widehat\beta$, $r\to r_I$, $\theta\to\Theta_\gamma$, $P\to P_\delta$, $\underline x_B\to\underline w_B$, driven by the one dimensional Brownian motion $\boldsymbol\vartheta\cdot\widehat{\boldsymbol B}/\Theta_\gamma$. Consequently all results of Sections~\ref{sec:parabolic-obstacle}--\ref{sec:verification-duality} apply to it.
\end{prop}
\begin{proof}
The identity \eqref{eq:income-utility-reduction} and the dynamics \eqref{eq:income-ratio-dynamics} are the conditional change of measure and It\^o computations displayed above. In the market \eqref{eq:income-ratio-dynamics} the state price density under $\widehat{\mathbb P}$ is driven by $\boldsymbol\vartheta\cdot\widehat{\boldsymbol B}/\Theta_\gamma$ with market price of risk $\Theta_\gamma$, and the exposure orthogonal to $\boldsymbol\vartheta$ carries no premium. The optimal consumption and terminal wealth of Section~\ref{sec:verification-duality} depend only on this one dimensional state. Hence the martingale representation of the replicating wealth has zero orthogonal component, and the exposure along $\boldsymbol\vartheta$ is the scalar portfolio of the baseline problem. The $B^2$ coordinate $\widetilde f$ itself need not vanish, as the display following \eqref{eq:portfolio-income-imp} shows. The sign of the interest rate is never used in Sections~\ref{sec:model}--\ref{sec:verification-duality}, only the integral formula for human wealth, so any finite $r_I$ is allowed. Assumption~\ref{as:P_finite} holds for $P_\delta$ because $(1-\gamma)P_\delta>0$ and $P_\delta$ is smooth, and $\widehat\beta>0$, $\Theta_\gamma>0$ are the discount and nondegeneracy hypotheses of the baseline proofs.
\end{proof}

The free boundary of the reduced problem has a particularly simple geometry.  Writing
$p_\delta(\tau):=P_\delta(T-\tau)$,
\begin{equation}\label{eq:g-outside-imp}
p_\delta'(\tau)+\widehat\beta p_\delta(\tau)
=\frac{\delta^{1-\gamma}}{1-\gamma},
\end{equation}
and hence the local contact bound is constant,
\begin{equation}\label{eq:Gamma-outside-imp}
\Gamma_\delta(\tau)=-\gamma\log\delta.
\end{equation}
Moreover,
\begin{equation}\label{eq:q-outside-imp}
q_\delta(\tau):=e^{\widehat\beta\tau}p_\delta(\tau)
=\frac{1}{1-\gamma}\left[
(\kappa_B\underline w_B)^{1-\gamma}
+\delta^{1-\gamma}\frac{e^{\widehat\beta\tau}-1}{\widehat\beta}
\right].
\end{equation}
For $0<\gamma<1$, $q_\delta'(\tau)>0$ on $(0,T)$, so the running record
condition is satisfied at every interior date and therefore
$\mathcal T_0=(0,T)$.  For $\gamma>1$, the existence result above already
implies contact at every interior date.  Thus this outside option
produces a finite free boundary throughout the horizon for every
$\gamma>0$, $\gamma\neq1$.

\subsection{Credit limits, consumption, and portfolio choice}

Let $Q^I$ and $N^I$ denote the reduced dual functions of the problem in
Proposition~\ref{prop:income-reduction}, and let
$b_I(t)=\exp\{\mathcal V_I(T-t)\}$.  The endogenous minimum
wealth to income ratio is
\begin{equation}\label{eq:credit-ratio-imp}
\underline w(t)=-Q^I_\xi(t,b_I(t)),
\qquad
\frac{X_t}{I_t}\ge \underline w(t).
\end{equation}
The natural debt limit is likewise proportional to current income:
\begin{equation}\label{eq:natural-credit-ratio-imp}
\frac{X_t}{I_t}>-h_I(t),
\qquad
h_I(t):=
\frac{1-e^{-r_I(T-t)}}{r_I},
\end{equation}
again with the continuous extension at $r_I=0$.  Hence
$-\underline w(t)$ is the endogenous borrowing capacity per unit of current
income whenever $\underline w(t)<0$.

For an effective normalized dual state $\Xi_t=e^{v_t}\le b_I(t)$, the
ratio of consumption to income has the same form as before,
\begin{equation}\label{eq:cons-income-imp}
\frac{c_t^*}{I_t}=\Xi_t^{-1/\gamma}=e^{-v_t/\gamma}.
\end{equation}
The stock position can also be recovered directly from the transformed dual
solution.  Since
$\Xi Q^I_{\xi\xi}=-N^I_\xi$, the ratio of risky investment to income is
\begin{equation}\label{eq:portfolio-income-imp}
\frac{\pi_t^*}{I_t}
=
\frac{\sigma_I\rho_I}{\sigma}\,w_t
-\frac{\theta-\gamma\sigma_I\rho_I}{\sigma}
 N^I_\xi(t,\Xi_t).
\end{equation}
The remaining exposure to the traded contract is
\[
\frac{f_t^*}{I_t}=\sigma_I\sqrt{1-\rho_I^2}\,w_t
+\gamma\sigma_I\sqrt{1-\rho_I^2}\,N_\xi^I(t,\Xi_t).
\]
The first term in \eqref{eq:portfolio-income-imp} scales the stock exposure with income, and the
second is the optimal exposure in the transformed complete market
problem.  When $\rho_I=0$, risky investment converges to zero as wealth
approaches the credit limit. When $\rho_I\neq0$, the portfolio at the boundary
is a pure income hedge.

A larger $\delta$ raises the outside value because $u$ is increasing. The admissible sets are then nested, so $\underline w(t)$ is nondecreasing in $\delta$ and borrowing capacity cannot increase.  In the numerical examples below, the effect of $\delta$ on consumption and risky investment is concentrated near the credit limit.  Income risk enters through $r_I$ and $\Theta_\gamma$ and
therefore affects both the credit boundary and the split of
\eqref{eq:portfolio-income-imp} into hedging and speculative parts.

\subsection{Numerical comparative statics}\label{subsec:implications-numerics}

We illustrate the implications for credit and policies under the complete market calibration
\[
T=20,\quad \beta=0.04,\quad r=0.02,\quad \mu=0.06,\quad \sigma=0.20,
\quad \mu_I=0.01,\quad \sigma_I=0.10,\quad \rho_I=0,
\]
with $\kappa_B=1$, $\underline w_B=0.05$, and
$\delta\in\{0.65,0.80,0.95\}$.  A larger $\delta$ makes the outside option more
attractive and therefore tightens the endogenous credit limit.

Figure~\ref{fig:imp-credit-limits} shows that a larger $\delta$ raises the minimum
wealth to income ratio $\underline w(t)$ for $\gamma$ on either side of one and reduces debt
capacity per unit of income.  The dash dotted curve is the natural debt limit
$-h_I(t)$, and the dotted curve is the benchmark with only the terminal floor,
\[
\underline w_{\rm term}(t)
=-h_I(t)+e^{-r_I(T-t)}\underline w_B.
\]
The gap between these benchmarks and the endogenous boundary shows that, away from maturity,
most of the tightening of credit comes from the interim welfare requirement.

\begin{figure}[!t]
\centering
\begin{tabular}{cc}
\includegraphics[width=0.45\textwidth]{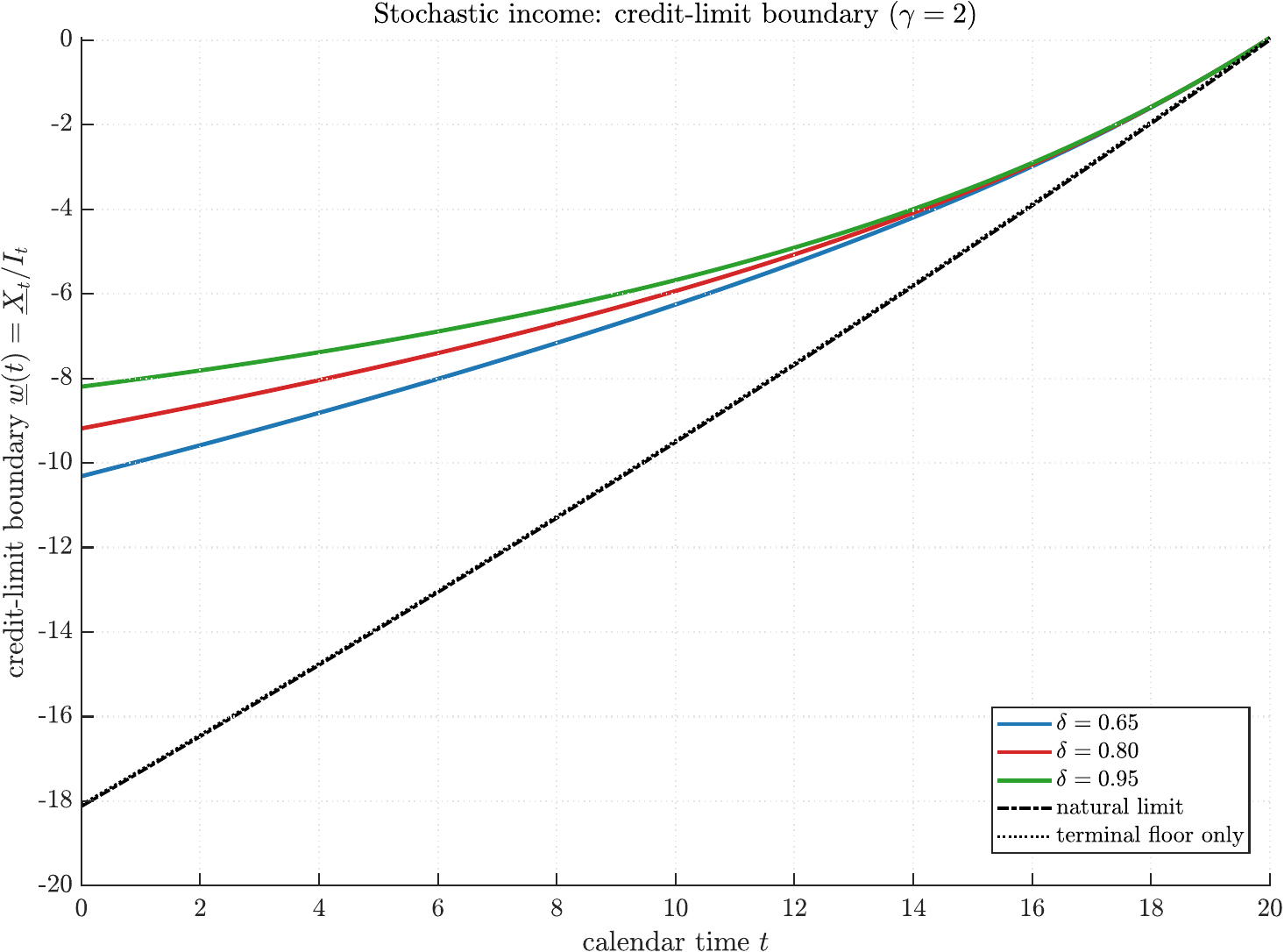}
&
\includegraphics[width=0.45\textwidth]{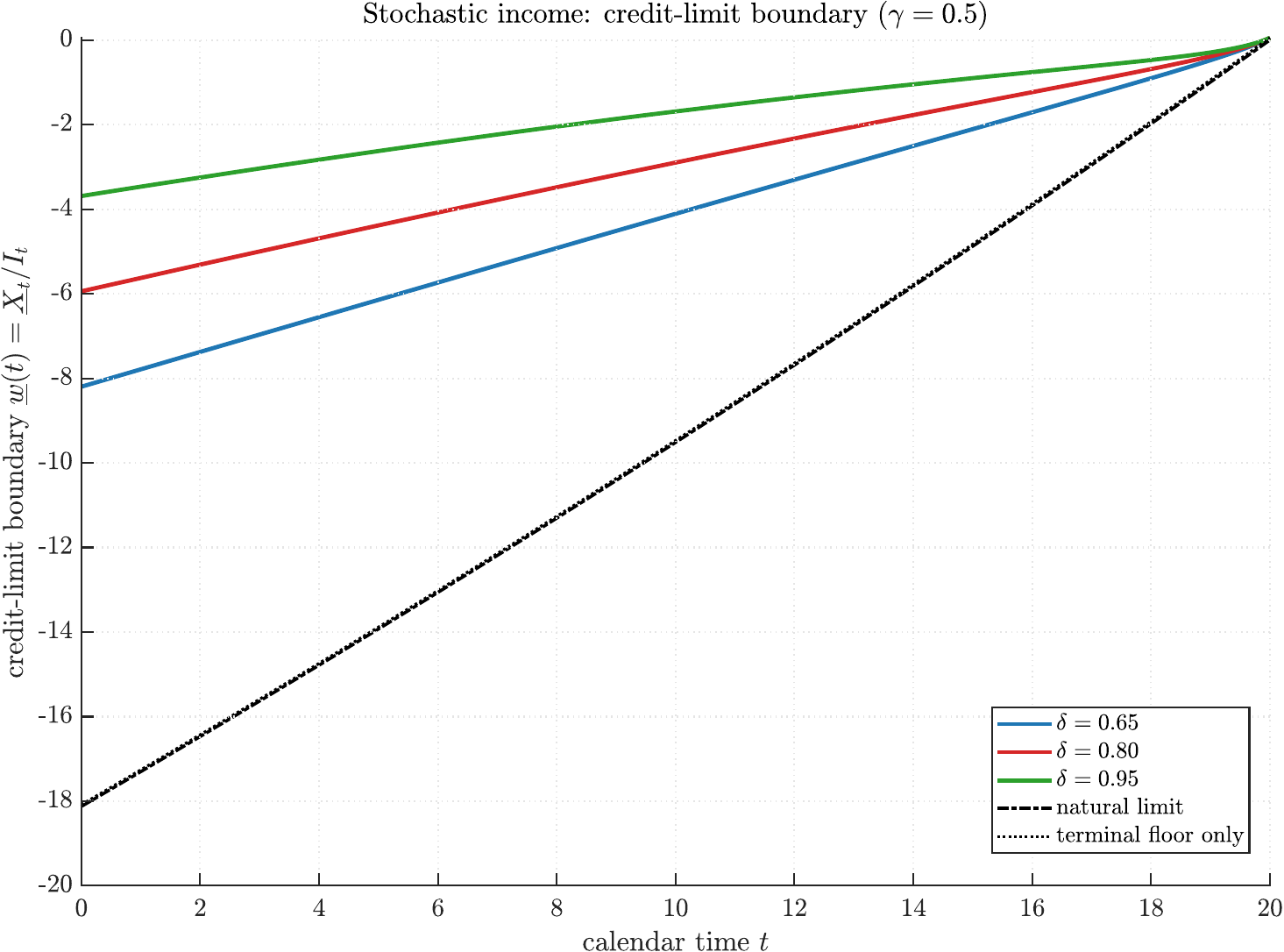}
\\[-0.4em]
{\small (a) $\gamma=2$}
&
{\small (b) $\gamma=0.5$}
\end{tabular}
\caption{Endogenous credit limits with stochastic income.  The colored curves
show the minimum wealth to income ratio $\underline w(t)$ for several values of $\delta$.
A larger $\delta$ makes the outside option more attractive and the credit limit
tighter.  The dash dotted and dotted black curves are the natural debt limit and
the benchmark with only the terminal floor. They differ by
$e^{-r_I(T-t)}\underline w_B\le0.05$ and therefore almost coincide in the figure.}
\label{fig:imp-credit-limits}
\end{figure}

Figure~\ref{fig:imp-policy-ratios} shows the corresponding optimal policies at
$t=0$ for $\gamma=2$.  The effects are concentrated near the endogenous lower end of the domain of the wealth to income ratio and become small as wealth rises.  Because
$\rho_I=0$, the stock position has no income hedging component at
the boundary. Smooth fit gives $N^I_\xi=0$ there, so risky investment converges to zero.
A more attractive outside option moves this point to a higher wealth to income
ratio and reduces risky exposure over the constrained region.

\begin{figure}[!t]
\centering
\begin{tabular}{cc}
\includegraphics[width=0.45\textwidth]{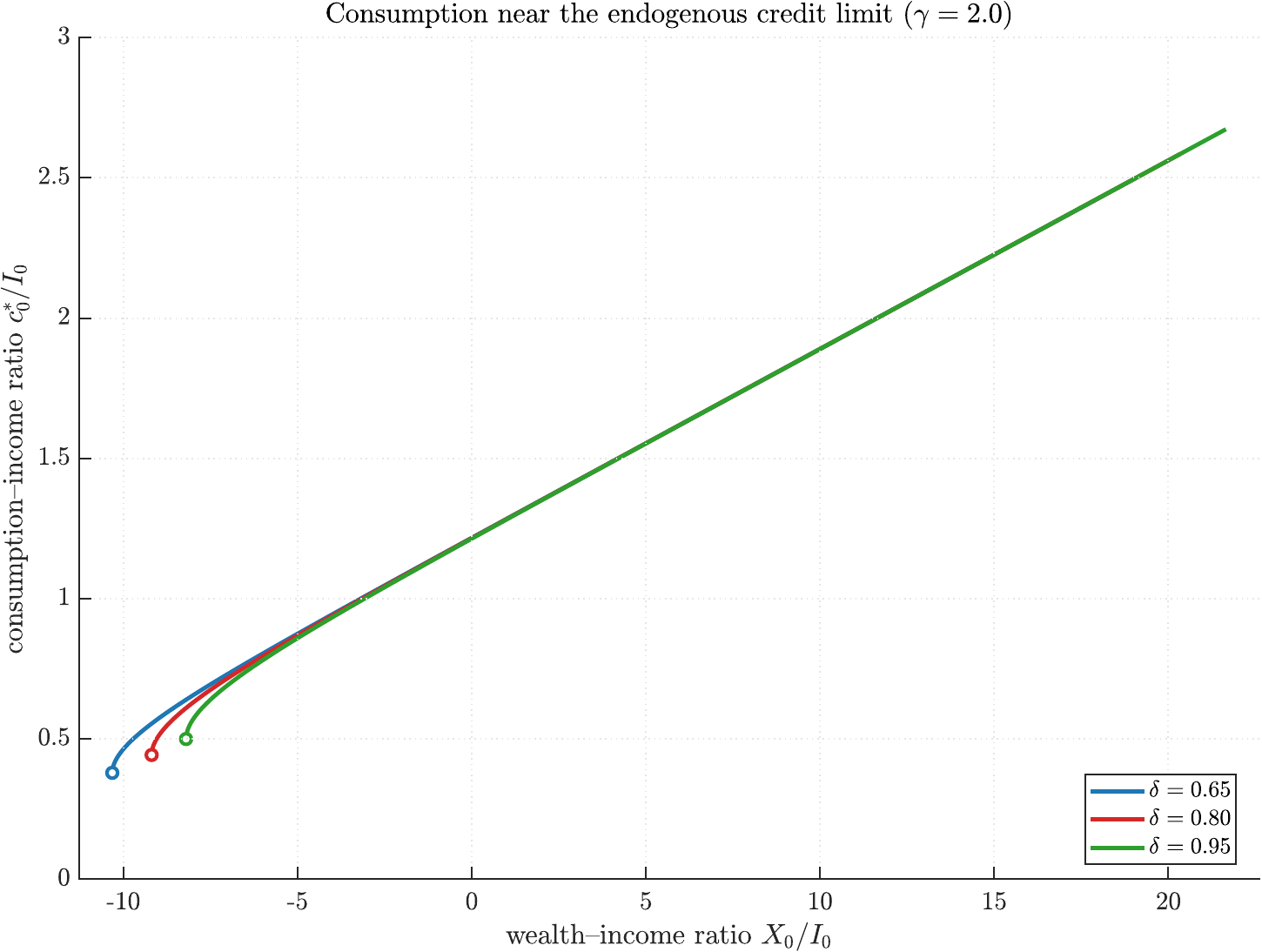}
&
\includegraphics[width=0.45\textwidth]{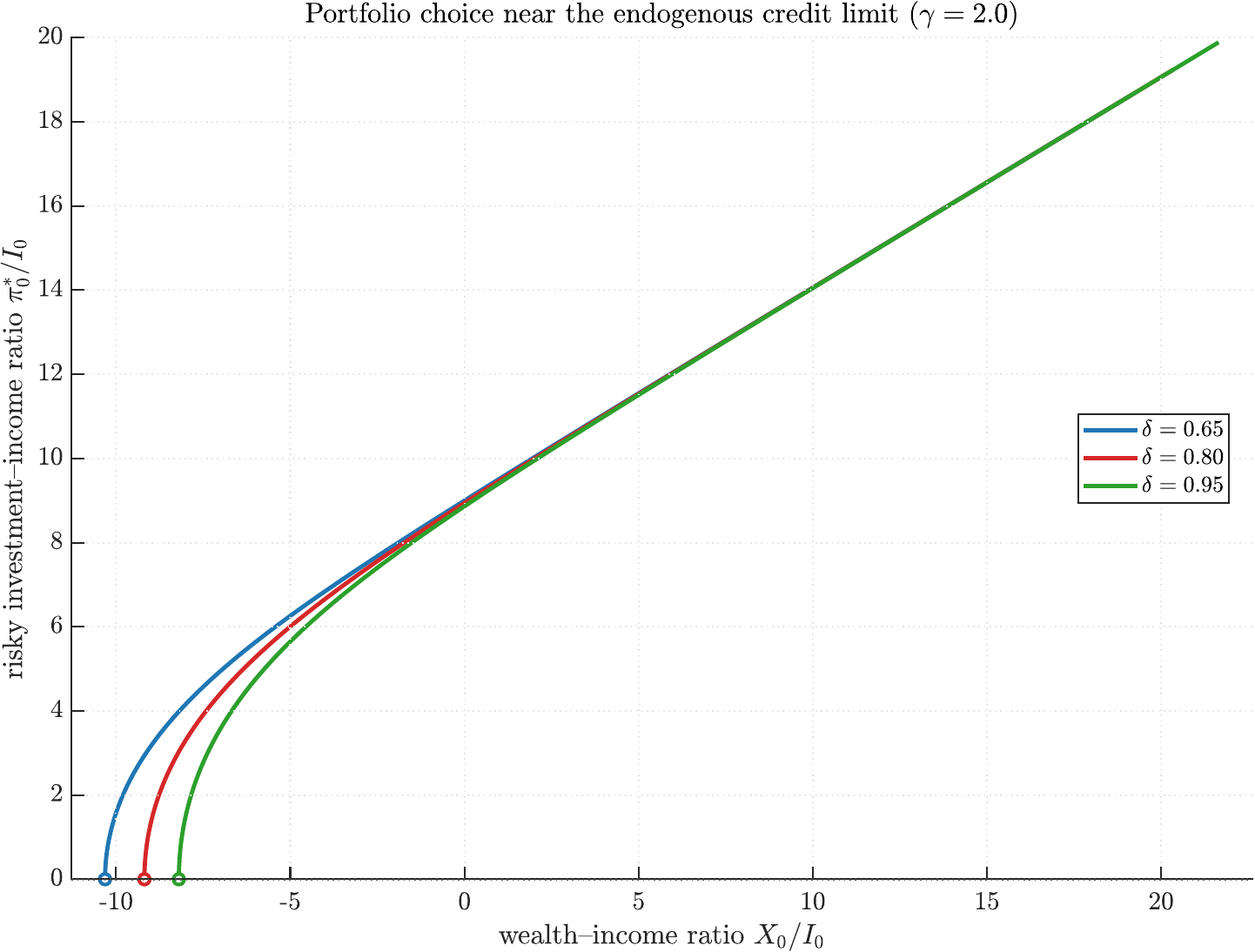}
\\[-0.4em]
{\small (a) Consumption to income ratio}
&
{\small (b) Risky investment to income ratio}
\end{tabular}
\caption{Optimal policy ratios as functions of the initial wealth to income ratio
for $\gamma=2$.  Open circles mark the endogenous credit limits.  The
parameter $\delta$ of the outside option mainly affects policies close to the borrowing
boundary, and the curves nearly merge as wealth relative to
income increases.}
\label{fig:imp-policy-ratios}
\end{figure}

\section{Conclusion}\label{sec:conclusion}

We have studied consumption and portfolio choice on a finite horizon when continuation utility must stay above a floor that changes over time. A cumulative multiplier reduces the problem to a parabolic obstacle problem with a time dependent obstacle and a kinked terminal condition. Its intervention boundary is finite at every interior date when $\gamma>1$. When $0<\gamma<1$, it is finite only on the set of dates defined by the strict running record condition, so it may disappear and later reenter. Verification and strong duality identify the dual boundary with the exact lower endpoint of the feasible wealth domain. With stochastic income the welfare floor is the value of an outside option under limited commitment, and a more attractive outside option raises the endogenous lower bound on the wealth to income ratio. In the numerical examples the resulting changes in consumption and investment are concentrated near the credit limit. Natural extensions include incomplete markets for income risk and outside options that are themselves stochastic.

\appendix
\small

\section{Proofs}\label{sec:proofs}

\subsection{Proof of Proposition \ref{pro:weak-duality}}\label{sec:weak-duality}
\begin{proof}
First suppose that $\mathcal Z$ is bounded, and put $A=1+\mathcal Z$.
For $0\le s\le T$, define the remaining discounted utility
\[
 \mathcal U_s:=\int_s^T e^{-\beta u}u(c_u)\,du
                 +e^{-\beta T}U_B(X_T).
\]
Its optional projection is
${}^{o}\mathcal U_s=e^{-\beta s}\mathcal W_s$. This equality also
holds at stopping times, by the continuous version of the uniformly
integrable utility martingale. Since $\mathcal Z$ is adapted and
c\`adl\`ag, it is optional, and
\[
 \mathbb E\!\int_{[0,T]}|\mathcal U_s|\,d\mathcal Z_s
 \le \|\mathcal Z_T\|_{L^\infty}
 \mathbb E\!\left[\int_0^T e^{-\beta u}|u(c_u)|\,du
                 +e^{-\beta T}|U_B(X_T)|\right]<\infty.
\]
The optional projection identity for an increasing optional integrator
(see \citet[Chapter~VI]{DellacherieMeyer1982}) therefore gives
\[
 \mathbb E\!\int_{[0,T]}\mathcal U_s\,d\mathcal Z_s
 =\mathbb E\!\int_{[0,T]}{}^{o}\mathcal U_s\,d\mathcal Z_s
 =\mathbb E\!\int_{[0,T]}e^{-\beta s}\mathcal W_s\,d\mathcal Z_s.
\]
Fubini's theorem, justified by the same bound on absolute integrability,
and $\mathcal Z_u=\int_{[0,u]}d\mathcal Z_s$ now yield
\begin{equation}\label{eq:W-Z-identity-new}
\begin{split}
\mathcal W_0=\mathbb E\bigg[&
\int_0^T e^{-\beta s}A_su(c_s)\,ds
+e^{-\beta T}A_TU_B(X_T)\\
&-\int_{[0,T]}e^{-\beta s}\mathcal W_s\,d\mathcal Z_s\bigg].
\end{split}
\end{equation}
The Stieltjes integrals include the possible atom at $0$.

Using $\mathcal W_s\ge P(s)$, the budget inequality, and the two
Fenchel inequalities gives
\[
\mathcal W_0\le yx+\mathbb E\bigg[
\int_0^T e^{-\beta s}\{A_s\widetilde u(\mathcal Y_s^y/A_s)
+\epsilon\mathcal Y_s^y\}\,ds
+e^{-\beta T}A_T\widetilde U_B(\mathcal Y_T^y/A_T)
-\int_{[0,T]}e^{-\beta s}P(s)\,d\mathcal Z_s\bigg].
\]
This is the desired inequality for bounded multipliers.

For a general $\mathcal Z\in\Pi(0)$ apply the preceding result to
$\mathcal Z^{(n)}=\min\{\mathcal Z,n\}$. The absolute values of the
perspective terms are bounded by a constant times
\[
A_T+\sup_{s\le T}\mathcal Y_s^y
+A_T^{1/\gamma}\sup_{s\le T}(\mathcal Y_s^y)^a.
\]
This is integrable by the multiplier moment condition and the positive
and negative moments of the lognormal process. The Stieltjes terms are
dominated by $C\mathcal Z_T$. Dominated convergence therefore gives
$\mathfrak L(y,\mathcal Z^{(n)})\to\mathfrak L(y,\mathcal Z)$ and proves
the claim. Thus the extension to unbounded multipliers uses only
integrable bounds for the dual criterion.
\end{proof}

\subsection{Proof of Proposition \ref{pro:dimension_reduction}}\label{proof:pro:dimension_reduction}
\begin{proof}
Fix $(t,y,z)$ and set $\xi=y/(1+z)$.  For any
$\mathcal Z\in\Pi_t(z)$ define
\[
 \mathcal D_s:=\frac{1+\mathcal Z_s}{1+z}-1,
 \qquad t\le s\le T.
\]
Because $\mathcal Z_{t-}=z$, we have $\mathcal D_{t-}=0$, so
$\mathcal D\in\Pi_t(0)$.  Conversely, every
$\mathcal D\in\Pi_t(0)$ produces a unique
$\mathcal Z\in\Pi_t(z)$ via
\[
 \mathcal Z_s=(1+z)(1+\mathcal D_s)-1.
\]
Moreover,
\[
 \mathcal Y_s^{t,y}=(1+z)\mathcal Y_s^{t,\xi},\qquad
 \frac{\mathcal Y_s^{t,y}}{1+\mathcal Z_s}
 =\frac{\mathcal Y_s^{t,\xi}}{1+\mathcal D_s},\qquad
 d\mathcal Z_s=(1+z)d\mathcal D_s,
\]
including possible atoms at $t$ and $T$.  Every term in the criterion defining
$\mathfrak J(t,y,z)$ therefore equals $(1+z)$ times the corresponding term for
$\mathfrak J(t,\xi,0)$.  Taking infima yields
\[
 \mathfrak J(t,y,z)=(1+z)Q\!\left(t,\frac y{1+z}\right).
\]
\end{proof}

\subsection{Proof of Lemma \ref{lem:existence-R}}\label{proof:lem:existence-R}
\begin{proof}
The lognormal moment bounds make every payoff in
\eqref{eq:R-stopping-representation} integrable, uniformly over stopping
times. If $0<\gamma<1$, its absolute value is bounded in expectation
by $C(1+e^{av})$. If $\gamma>1$, every reward is negative and immediate
stopping gives $p(\tau)\le R(\tau,v)\le0$. Thus the same growth bound
holds in either case. The Snell envelope construction and dynamic
programming give the continuous obstacle value. The terminal reward
is $\max\{B,p(0)\}$, including the option to enforce terminal welfare.
See \citet{PS} for the stopping formulation.

Fix a bounded cylinder $\mathcal Q_0\Subset\Omega_T$, and let $g_\varepsilon$ be
smooth nonnegative approximations converging uniformly to $R-p$ on
the parabolic boundary $\partial_p\mathcal Q_0$. Choose a smooth nonnegative,
nonincreasing penalty $\omega_\varepsilon$, vanishing on
$[\varepsilon,\infty)$, with
$\omega_\varepsilon(0)=\max\{0,\sup_{\mathcal Q_0}(p'+\beta p-f)\}$.
Let $R_\varepsilon$ solve
\[
\begin{cases}
(\partial_\tau-\mathscr L)R_\varepsilon
=f+\omega_\varepsilon(R_\varepsilon-p)&\text{in }\mathcal Q_0,\\
R_\varepsilon=p+g_\varepsilon&\text{on }\partial_p\mathcal Q_0.
\end{cases}
\]
Comparison gives $R_\varepsilon\ge p$ and
$0\le\omega_\varepsilon(R_\varepsilon-p)\le\omega_\varepsilon(0)$,
where the upper bound is independent of $\varepsilon$.
Since $f$ and the boundary data are uniformly bounded on this fixed
cylinder, comparison also gives an $L^\infty(\mathcal Q_0)$ bound independent of
$\varepsilon$.
Fix a sequence $\varepsilon_n\downarrow0$. The uniform
bound on the right hand side gives interior H\"older estimates for
$R_{\varepsilon_n}$. The boundary values $p+g_{\varepsilon_n}$ converge
uniformly to the continuous function $R$ and are therefore equicontinuous.
Boundary barriers then give a common modulus of continuity at
$\partial_p\mathcal Q_0$ \citep{Lieberman,Ladyzhenskaya}.
The Arzel\`a--Ascoli theorem yields a subsequence, relabelled
$\varepsilon_j$, and $u\in C(\overline{\mathcal Q_0})$ such that
\[
    R_{\varepsilon_j}\longrightarrow u
    \qquad\text{uniformly on }\overline{\mathcal Q_0}.
\]
In particular, $u\ge p$ in $\mathcal Q_0$ and $u=R$ on
$\partial_p\mathcal Q_0$.

For every cylinder $\mathcal Q\Subset\mathcal Q_0$
and $1<q<\infty$, the interior $W_q^{1,2}$ estimate gives
\[
    \|R_{\varepsilon_j}\|_{W_q^{1,2}(\mathcal Q)}
    \le C_{\mathcal Q,q},
\]
where $C_{\mathcal Q,q}$ is independent of $j$.
Applying weak compactness along a countable increasing exhaustion of
$\mathcal Q_0$ by cylinders and a sequence of exponents $q_m\uparrow\infty$,
we obtain a further subsequence, still denoted by $\varepsilon_j$, such that
\[
    R_{\varepsilon_j}\rightharpoonup u
    \qquad\text{in }W_q^{1,2}(\mathcal Q).
\]
This holds for every $\mathcal Q\Subset\mathcal Q_0$ and $1<q<\infty$;
the weak limit is $u$ by uniform convergence.
Thus $u\in W_{q,\mathrm{loc}}^{1,2}(\mathcal Q_0)$ for every
$1<q<\infty$. Since the penalty is nonnegative, testing the penalized
equation against a nonnegative smooth function of compact support and
passing to the weak limit gives
\[
    (\partial_\tau-\mathscr L)u\ge f
    \qquad\text{a.e. in }\mathcal Q_0.
\]
To obtain equality on $\{u>p\}$, fix a compact set
$K\Subset\mathcal Q_0\cap\{u>p\}$ and set
$d_K:=\min_K(u-p)>0$. Uniform convergence gives
$R_{\varepsilon_j}-p\ge d_K/2>\varepsilon_j$ on $K$ for all
sufficiently large $j$. Hence the penalty vanishes there, and passage
to the weak limit yields
\[
    (\partial_\tau-\mathscr L)u=f
    \qquad\text{a.e. in }\mathcal Q_0\cap\{u>p\}.
\]
Together with the boundary values, these properties show that $u$
solves the local obstacle problem.

Verification by the generalized It\^o formula, with stopping at first
contact with the obstacle or at exit, identifies $u$ with the local
stopping value, with discount
rate $\beta$, running reward $f$, reward $p$ on stopping before exit,
and boundary reward $R$. Dynamic programming for
\eqref{eq:R-stopping-representation} therefore gives $u=R$ in $\mathcal Q_0$.
Since $\mathcal Q_0$ is arbitrary, $R\in W_{q,\mathrm{loc}}^{1,2}(\Omega_T)$
for every $1<q<\infty$ and solves \eqref{VI:R}.
Parabolic Sobolev embedding for $q>3$ gives the local H\"older
continuity of $R_v$.

To make uniqueness explicit, take $k>|a|$ and choose $c$ sufficiently
large that
$\Psi(\tau,v)=e^{c\tau}(e^{kv}+e^{-kv})$ satisfies
$(\partial_\tau-\mathscr L)\Psi\ge0$. For two solutions in
\eqref{eq:R-growth-class}, compare the first with the second plus
$\varepsilon\Psi$ on $(-n,n)$. The lateral inequalities hold for all
sufficiently large $n$, uniformly in time. The obstacle comparison
principle applies because, wherever the first function exceeds the
second, the first is strictly above its obstacle. Let $n\to\infty$
and then $\varepsilon\to0^+$, and interchange the solutions.
This proves uniqueness in the stated class.
The comparison assertions follow from the same barrier. For the first, on the open set $\{R>\zeta\}\subset\{R>p\}$ one has $(\partial_\tau-\mathscr L)(R-\zeta)\le0$, so comparing $R$ with $\zeta+\varepsilon\Psi$ on $\mathcal N\cap\{|v|<n\}$, where the lateral inequalities hold for large $n$ by the growth bound, and letting $n\to\infty$, $\varepsilon\to0^+$ gives $R\le\zeta$. For the second, $(\partial_\tau-\mathscr L)R\ge f\ge(\partial_\tau-\mathscr L)\zeta$ a.e. on $\{\zeta>R\}\subset\{\zeta>p\}$, and the same barrier gives $\zeta\le R$.

Finally, put $W=R-p$ and $\mathfrak d=\theta^2/2$. Then
\[
(\partial_\tau-\mathscr L)W
=\{f(v)-p'(\tau)-\beta p(\tau)\}\mathbf1_{\{W>0\}}.
\]
The change of variables $v=\sqrt{\mathfrak d}\,x-b_0\tau$ and multiplication by $e^{\beta\tau}$
reduce this to an obstacle problem for the heat operator. Before multiplication
by the positivity indicator, the transformed source is
\[
 e^{\beta\tau}\bigl[f(\sqrt{\mathfrak d}\,x-b_0\tau)-p'(\tau)-\beta p(\tau)\bigr].
\]
It is locally Lipschitz in $(\tau,x)$ because $f$ is smooth and
Assumption~\ref{as:P_finite} gives the Lipschitz continuity of $p'$.
Theorem~1 of \citet{Optimalregularity} therefore gives
locally bounded first time and second spatial derivatives of the transformed solution.
Transforming back and using $p'\in L^\infty(0,T)$ gives
$R_\tau,R_{vv}\in L^\infty_{\mathrm{loc}}(\Omega_T)$.
Together with the local boundedness of $R$ and $R_v$, this proves
$R\in W_{\infty,\mathrm{loc}}^{1,2}(\Omega_T)$.
\end{proof}

\subsection{Proof of Lemma \ref{lem:necessary-contact-R}}\label{proof:lem:necessary-contact-R}
\begin{proof}
Set $g(\tau):=p'(\tau)+\beta p(\tau)$.
Let $W:=R-p$ and fix a contact point
$(\tau_0,v_0)\in\{R=p\}$.
We first prove $f(v_0)\le g(\tau_0)$.
Suppose, to the contrary, that $f(v_0)>g(\tau_0)$. By continuity,
there are a cylinder $\mathcal Q\Subset\Omega_T$ containing
$(\tau_0,v_0)$ and a constant $\eta>0$ such that
$f(v)-g(\tau)\ge\eta$ in $\mathcal Q$.
Since $(\partial_\tau-\mathscr L)R\ge f$ almost everywhere, we have
\[
    (\partial_\tau-\mathscr L)W
    \ge f(v)-g(\tau)\ge\eta
    \qquad\text{a.e. in }\mathcal Q.
\]
On the other hand, $W\ge0$ and $W(\tau_0,v_0)=0$.
By Lemma~\ref{lem:existence-R}, $W$ is continuous and belongs to
$W_{\infty,\mathrm{loc}}^{1,2}(\Omega_T)$.
The strong maximum principle for strong supersolutions \citep{Lieberman}
therefore forces $W$ to vanish in a backward subcylinder of
$\mathcal Q$, contradicting the strict inequality.
Thus $f(v_0)\le g(\tau_0)$.

Suppose first that $\gamma\in(0,1)$. Since $f>0$, contact is impossible
when $g(\tau_0)\le0$. If $g(\tau_0)>0$, the strict monotonicity of
$f$ gives $v_0\ge\Gamma(\tau_0)$.

We now exclude equality. Suppose that
$v_0=\Gamma(\tau_0)<+\infty$. Since $g(\tau_0)>0$ and
$p\in C^{1,1}([0,T])$, the function $g$ is locally Lipschitz and
remains strictly positive near $\tau_0$. Hence $\Gamma$ is locally
Lipschitz near $\tau_0$. For sufficiently small $\kappa>0$, set
\[D_\kappa:=\{(\tau,v):|\tau-\tau_0|<\kappa^2,
-\kappa<v-\Gamma(\tau)<0\}\Subset\Omega_T.\] Since $v\ge\Gamma(\tau)$ at every contact point $(\tau,v)$, $W>0$ in $D_\kappa$. Moreover, since $f$ is
strictly decreasing and $f(\Gamma(\tau))=g(\tau)$, we have
$(\partial_\tau-\cL)W=f(v)-g(\tau)>0$ in $D_\kappa$.

To flatten the lateral boundary, set $y:=v-\Gamma(\tau)$ and
$\widetilde W(\tau,y):=W(\tau,y+\Gamma(\tau))$. This change of
variables maps $D_\kappa$ onto
$(\tau_0-\kappa^2,\tau_0+\kappa^2)\times(-\kappa,0)$ and maps
$(\tau_0,v_0)$ to $(\tau_0,0)$. Moreover, $\widetilde W$ satisfies
\[
    \partial_\tau\widetilde W
    -\frac{\theta^2}{2}\partial_{yy}\widetilde W
    -\bigl(b_0+\Gamma'(\tau)\bigr)\partial_y\widetilde W
    +\beta\widetilde W
    =f\bigl(y+\Gamma(\tau)\bigr)-g(\tau)>0
\]
in the flat cylinder. Since $\Gamma$ is locally
Lipschitz, $\Gamma'\in L^\infty$ locally. Thus the transformed
operator is uniformly parabolic, its leading coefficient is
constant, and all its lower order coefficients are bounded. The
parabolic Hopf boundary point lemma for strong solutions with bounded measurable lower order coefficients \citep[Section~3]{Nazarov2012} therefore applies at the boundary point $(\tau_0,0)$ on the lateral boundary $\{y=0\}$ and gives
$\partial_y\widetilde W(\tau_0,0)<0$. Since
$\partial_y\widetilde W(\tau_0,0)=\partial_vW(\tau_0,v_0)$, we
obtain $\partial_vW(\tau_0,v_0)<0$.

On the other hand, $W(\tau_0,\cdot)\ge0$ and
$W(\tau_0,v_0)=0$, so $v_0$ is a minimum point of
$W(\tau_0,\cdot)$. Hence $\partial_vW(\tau_0,v_0)=0$, which is a
contradiction. Therefore, $v_0>\Gamma(\tau_0)$.

Suppose next that $\gamma\in(1,\infty)$. If $g(\tau_0)<0$, the same argument
excludes equality and gives $v_0>\Gamma(\tau_0)$. If
$g(\tau_0)\ge0$, then $f(v)<0\le g(\tau_0)$ for every
$v\in\mathbb R$, while $\Gamma(\tau_0)=-\infty$. Thus
$v_0>\Gamma(\tau_0)$ in all cases.
\end{proof}

\subsection{Proof of Lemma \ref{lem:strict-record-condition}}\label{proof:lem:strict-record-condition}
\begin{proof}
Let $(\tau_0,v_0)\in\Omega_T$ be such that
$R(\tau_0,v_0)=p(\tau_0)$. The first assertion follows immediately
from Lemma~\ref{lem:necessary-contact-R}. Indeed,
$f(v_0)<p'(\tau_0)+\beta p(\tau_0)$, and since $f(v_0)>0$ for
$\gamma\in(0,1)$, we obtain
\[
q'(\tau_0)
=
e^{\beta\tau_0}\{p'(\tau_0)+\beta p(\tau_0)\}
>0.
\]

Now fix $s\in[0,\tau_0)$ and define
$\underline R_s(\tau,v):=e^{-\beta(\tau-s)}p(s)$ on
$[s,T]\times\mathbb R$. Since $\underline R_s$ is independent of
$v$, we have $(\partial_\tau-\cL)\underline R_s=0$. Moreover,
$(\partial_\tau-\cL)R\ge f>0$ for $\gamma\in(0,1)$, while at
$\tau=s$ we have
$R(s,v)\ge p(s)=\underline R_s(s,v)$. Hence the comparison principle (in the growth class of Lemma~\ref{lem:existence-R})
gives
\[
R(\tau,v)\ge e^{-\beta(\tau-s)}p(s)
\qquad\text{on }\ [s,T]\times\mathbb R.
\]

The inequality is in fact strict for every $\tau>s$. Otherwise,
$R-\underline R_s$ would attain the value $0$ at some interior point,
while $R-\underline R_s\ge0$ and
$(\partial_\tau-\cL)(R-\underline R_s)\ge f>0$, contradicting the
strong maximum principle. Evaluating the resulting strict inequality
at $(\tau_0,v_0)$ gives
$p(\tau_0)>e^{-\beta(\tau_0-s)}p(s)$, and therefore
$q(\tau_0)>q(s)$.
\end{proof}

\subsection{Proof of Lemma \ref{lem:T0-open}}\label{proof:lem:T0-open}
\begin{proof}
Fix $\tau_0\in\mathcal T_0$. Since $q'(\tau_0)>0$, there is $\kappa_0\in(0,\min\{\tau_0,T-\tau_0\})$
such that $q$ is strictly increasing on $[\tau_0-\kappa_0,\tau_0+\kappa_0]$,
and since $\tau_0$ is a strict running record time,
$\max_{0\le s\le\tau_0-\kappa_0}q(s)<q(\tau_0)$. By continuity there is
$\kappa\in(0,\kappa_0)$ such that $q'(\tau)>0$ and
$q(\tau)>\max_{0\le s\le\tau_0-\kappa_0}q(s)$ whenever $|\tau-\tau_0|<\kappa$.
For such $\tau$ and $0\le s<\tau$, either $s\le\tau_0-\kappa_0$, and then
$q(s)<q(\tau)$ by the preceding inequality, or $s>\tau_0-\kappa_0$, and then
$q(s)<q(\tau)$ by strict monotonicity. Hence
$(\tau_0-\kappa,\tau_0+\kappa)\subset\mathcal T_0$.
\end{proof}

\subsection{Proof of Lemma \ref{lem:contact-existence-gamma-less-one}}\label{proof:lem:contact-existence-gamma-less-one}
\begin{proof}
Fix $\tau_0\in\mathcal T_0$. There exist $\kappa_0\in(0,\min\{\tau_0,T-\tau_0\})$ and
$c_0>0$ such that $q'(\tau)\ge2c_0$ and $q(\tau)>0$ whenever
$|\tau-\tau_0|\le\kappa_0$. Moreover,
$\max_{0\le s\le\tau_0-\kappa_0}q(s)<q(\tau_0)$. By continuity,
we may choose $\kappa\in(0,\kappa_0)$ such that
\[
    d_0
    :=
    \min_{|\tau-\tau_0|\le\kappa/2}
    \left[q(\tau)-\max_{0\le s\le\tau_0-\kappa}q(s)\right]
    >0.
\]
Choose $\lambda_0>0$ sufficiently small that
$\lambda_0 e^{\beta T}\le c_0$ and
$\lambda_0\int_0^T e^{\beta\xi}\,d\xi\le d_0/2$. We claim that,
whenever $|\tau-\tau_0|\le\kappa/2$,
\begin{equation}\label{ineq:q}
    q(\tau)-q(s)
    \ge
    \lambda_0\int_s^\tau e^{\beta\xi}\,d\xi
    \qquad\text{for every }\ 0\le s\le\tau.
\end{equation}
If $s\ge\tau_0-\kappa$, then
$q'(\xi)-\lambda_0 e^{\beta\xi}\ge c_0$ between $s$ and $\tau$, and
the claim follows by integrating both sides.
If $s\le\tau_0-\kappa$, then the
left hand side is at least $d_0$, whereas the integral on the
right hand side is at most $d_0/2$. Thus \eqref{ineq:q} holds.

Fix $\tau\in [\tau_0-\kappa/2,\tau_0+\kappa/2]$ and define
\[
    \rho_\tau(s)
    :=
    e^{\beta(\tau-s)}p(\tau)
    -\lambda_0\int_s^\tau e^{\beta(\xi-s)}\,d\xi,
    \qquad 0\le s\le\tau.
\]
Then $\rho_\tau'(s)+\beta\rho_\tau(s)=\lambda_0$ and
$\rho_\tau(\tau)=p(\tau)$. Moreover, using
$q(s)=e^{\beta s}p(s)$, the inequality \eqref{ineq:q} gives
\[
    e^{\beta s}\{\rho_\tau(s)-p(s)\}
    =q(\tau)-q(s)-\lambda_0\int_s^\tau e^{\beta\xi}\,d\xi
    \ge0.
\]
Hence $\rho_\tau(s)\ge p(s)$ for every $0\le s\le\tau$.

Since $a=1-1/\gamma<0$, the growth bound \eqref{eq:R-growth-class} in Lemma~\ref{lem:existence-R} and the boundedness of $p$ imply that
there exists $M>0$ such that
\[
    R(s,v)-p(s)\le M
    \qquad\text{for }\ (s,v)\in[0,T]\times[0,\infty).
\]
Choose \(\varkappa>0\) sufficiently small that
\[
    \varkappa^{-2}\ge\max\{0,v_B\},\qquad
    \varkappa^{-1}\ge M,\qquad
    f(\varkappa^{-2})\le\frac{\lambda_0}{2},\qquad
    \theta^2\varkappa^3+2|b_0|\varkappa\le\frac{\lambda_0}{2}.
\]
Then \(p(s)+\varkappa^{-1}\ge R(s,\varkappa^{-2})\) for every \(s\in[0,T]\).
Moreover, since \(\varkappa^{-2}\ge v_B\), we have
\(R(0,v)=p(0)\) for every \(v\ge \varkappa^{-2}\).

With such \(\varkappa>0\), define
$\mathcal N_\tau:=(0,\tau)\times(\varkappa^{-2},\infty)$
and
\[\zeta_\tau(s,v)
    :=\rho_\tau(s)+\varkappa^3\bigl[(2\varkappa^{-2}-v)^+\bigr]^2
    \qquad \text{for }\ (s,v)\in \mathcal{N}_\tau.\]
Then $\zeta_\tau\in
    W^{1,2}_{\infty,\mathrm{loc}}(\mathcal N_\tau)
    \cap C(\overline{\mathcal N_\tau})$.
Since $\rho_\tau\ge p$, we have
$\zeta_\tau\ge p$ in $\mathcal N_\tau$.
For $v\ge2\varkappa^{-2}$, $\zeta_\tau(s,v)=\rho_\tau(s)$ and hence
$(\partial_s-\cL)\zeta_\tau=\lambda_0\ge f(v)$ in $\mathcal N_\tau$.
For $\varkappa^{-2}\le v\le2\varkappa^{-2}$, direct computation gives
\[
\begin{aligned}
    (\partial_s-\cL)\zeta_\tau
    &=\lambda_0-\theta^2\varkappa^3
      -2b_0\varkappa^3\left(v-\frac2{\varkappa^2}\right)
      +\beta \varkappa^3\left(v-\frac2{\varkappa^2}\right)^2\\
    &\ge\lambda_0-\theta^2\varkappa^3-2|b_0|\varkappa
    \ge\frac{\lambda_0}{2}\ge f(v).
\end{aligned}
\]

On the initial boundary, since $\varkappa^{-2}\ge v_B$, we have
$\zeta_\tau(0,v)\ge\rho_\tau(0)\ge p(0)=R(0,v)$. On the
lateral boundary $v=\varkappa^{-2}$,
$\zeta_\tau(s,\varkappa^{-2})=\rho_\tau(s)+\varkappa^{-1}
\ge R(s,\varkappa^{-2})$. The comparison principle (in the growth class of Lemma~\ref{lem:existence-R}) on the unbounded
half cylinder $\mathcal N_\tau$ therefore yields
$R\le\zeta_\tau$ in $\overline{\mathcal N_\tau}$. For $v\ge2\varkappa^{-2}$ this gives
$\zeta_\tau(\tau,v)=\rho_\tau(\tau)=p(\tau)$, and hence
$R(\tau,v)=p(\tau)$ for $v\ge2\varkappa^{-2}$ because $R\ge p$. 
The constant $\lambda_0>0$ was chosen uniformly for
$|\tau-\tau_0|\le\kappa/2$, and then \(\varkappa>0\) was chosen, depending on
this $\lambda_0$, so that all the preceding conditions hold uniformly
over the same interval. Hence the assertion follows with
$\nu=\kappa/2$ and $V=2\varkappa^{-2}$.
\end{proof}

\subsection{Proof of Lemma \ref{lem:bounds-gamma-greater-one}}\label{proof:lem:bounds-gamma-greater-one}
\begin{proof}
The first assertion follows from the stopping representation
\eqref{eq:R-stopping-representation}: for $\gamma\in(1,\infty)$ every reward
is nonpositive, so $R\le0$ in $\Omega_T$.

For the second assertion, let $\vartheta$ solve
\[
    \vartheta'(\tau)
    -
    \left[
        \frac{\theta^2}{2}
        \left(1-\frac1\gamma\right)^2
        +b_0\left(1-\frac1\gamma\right)
        -\beta
    \right]\vartheta(\tau)
    =
    \frac{1}{1-\gamma},
    \qquad
    \vartheta(0)
    =
    \frac{\kappa_B^{(1-\gamma)/\gamma}}{1-\gamma}.
\]
Then a direct computation gives
\[
    (\partial_\tau-\cL)
    \left[
        \vartheta(\tau)e^{(1-1/\gamma)v}
    \right] =f(v)
\quad \text{and} \quad
\vartheta(0)e^{(1-1/\gamma)v}
    =B(v)\le R(0,v).\]
Since $(\partial_\tau-\cL)R\ge f$, the comparison principle yields
\[
    R(\tau,v)
    \ge
    \vartheta(\tau)e^{(1-1/\gamma)v}.
\]
For each fixed $\tau\in(0,T)$, since $1-1/\gamma>0$,
\[
    \vartheta(\tau)e^{(1-1/\gamma)v}\to0
    \qquad\text{as }\ v\to-\infty.
\]
Since $p$ is continuous and strictly negative on $[0,T]$, there exists
$m>0$ such that
\[
    p(\tau)\le -m
    \qquad\text{for all }\ \tau\in[0,T].
\]
Moreover, $\vartheta$ is bounded on $[0,T]$, say
$|\vartheta(\tau)|\le M$. Since $1-1/\gamma>0$, we may choose
$v_*\in\mathbb R$ sufficiently negative so that
\[
    M e^{(1-1/\gamma)v_*}<m.
\]
Then, for every $\tau\in[0,T]$ and
$v\le v_*$,
\[
    R(\tau,v)
    \ge
    \vartheta(\tau)e^{(1-1/\gamma)v}
    >
    p(\tau).
\]
In particular,
\[
    (-\infty,v_*]
    \subset
    \{v\in\mathbb R:R(\tau,v)>p(\tau)\}
    \qquad\text{for all }\ \tau\in[0,T].
\]
\end{proof}

\subsection{Proof of Lemma \ref{lem:contact-existence-gamma-greater-one}}\label{proof:lem:contact-existence-gamma-greater-one}
\begin{proof}
Since $p\in C^{1,1}([0,T])$, there exists a constant $C>0$ such that
\[
    p'(\tau)+\beta p(\tau)\ge-C
\]
for every $\tau\in(0,T)$.

Choose \(\varkappa>0\) sufficiently small that
\[
    \varkappa^{-2}\ge v_B,\qquad
    f(\varkappa^{-2})\le-2C,\qquad
    \theta^2\varkappa^3+2|b_0|\varkappa\le C,\qquad
    \varkappa^{-1}\ge\|p\|_{L^\infty(0,T)}
\]
and define
\[
    \zeta_\varkappa(\tau,v)
    :=
    p(\tau)
    +\varkappa^3\bigl[(2\varkappa^{-2}-v)^+\bigr]^2
\]
on \(\mathcal{N}_\varkappa:=(0,T)\times(\varkappa^{-2},\infty)\). 
Clearly, $\zeta_\varkappa\in
W^{1,2}_{\infty,\mathrm{loc}}(\mathcal{N}_\varkappa)
\cap C(\overline{\mathcal{N}_\varkappa})$ and \(\zeta_\varkappa\ge p\) in $\mathcal{N}_\varkappa$.
For \(v\ge2\varkappa^{-2}\), the additional term vanishes, and hence
\[
    (\partial_\tau-\cL)\zeta_\varkappa
    =
    p'+\beta p
    \ge-C
    \ge f(v).
\]
For \(\varkappa^{-2}\le v\le2\varkappa^{-2}\), since $f$ is decreasing,
\[
    (\partial_\tau-\cL)\zeta_\varkappa
    \ge
    -C-\theta^2\varkappa^3-2|b_0|\varkappa
    \ge-2C
    \ge f(v).
\]
Since \(\varkappa^{-2}\ge v_B\), the initial condition gives
\(R(0,v)=p(0)\) for every \(v\ge \varkappa^{-2}\). Hence on the initial boundary of $\mathcal{N}_\varkappa$,
\[
    \zeta_\varkappa(0,v)
    =
    p(0)+\varkappa^3\bigl[(2\varkappa^{-2}-v)^+\bigr]^2
    \ge R(0,v).
\]
On the lateral boundary $v=\varkappa^{-2}$ of $\mathcal{N}_\varkappa$, Lemma~\ref{lem:bounds-gamma-greater-one}
and the choice of \(\varkappa\) yield
\[
    \zeta_\varkappa(\tau,\varkappa^{-2})
    =
    p(\tau)+\varkappa^{-1}
    \ge0
    \ge R(\tau,\varkappa^{-2}).
\]
The comparison principle (in the growth class of Lemma~\ref{lem:existence-R}) therefore gives \(R\le\zeta_\varkappa\) on $\overline{\mathcal{N}_\varkappa}$. 
For \(v\ge2\varkappa^{-2}\), the additional
term vanishes, and hence
\[
    R(\tau,v)
    \le
    \zeta_\varkappa(\tau,v)
    =
    p(\tau).
\]
Since \(R\ge p\), it follows that
\[
    R(\tau,v)=p(\tau)
    \qquad
    \text{on }\ [0,T]\times[2\varkappa^{-2},\infty).
\]
Thus the assertion follows with \(V_0:=2\varkappa^{-2}\).
\end{proof}

\subsection{Proof of Lemma \ref{lem:state-monotonicity-R}}\label{proof:lem:state-monotonicity-R}
\begin{proof}
We first prove that \(R\) is nonincreasing in \(v\). For \(\varkappa>0\),
define \(R^\varkappa(\tau,v):=R(\tau,v+\varkappa)\). Then \(R^\varkappa\) satisfies the same
obstacle problem with source \(f(v+\varkappa)\) and initial condition
\[
    R^\varkappa(0,v)=\max\{B(v+\varkappa),p(0)\}.
\]
Since both \(B\) and \(f\) are strictly decreasing,
\[
    R^\varkappa(0,v)\le R(0,v),
    \qquad
    f(v+\varkappa)\le f(v).
\]
Hence, by the comparison principle
(in the growth class of Lemma~\ref{lem:existence-R}),
\(R^\varkappa\le R\) in \(\overline{\Omega_T}\). Since \(\varkappa>0\) is arbitrary,
\(v\mapsto R(\tau,v)\) is nonincreasing for every \(\tau\in(0,T)\).
The local Sobolev embedding theorem
\cite[Lemma 3.3, p.~80]{Ladyzhenskaya} yields
the H\"older continuity of $\partial_v R$ and hence,
\[
    \partial_vR\le0
    \qquad\text{in }\ \Omega_T.
\]

We next prove strict monotonicity in the noncontact region.
Suppose that
\(\partial_vR(\tau_0,v_0)=0\) for some
\((\tau_0,v_0)\in\{R>p\}\). Choose a cylinder
\(\mathcal{Q}\Subset\{R>p\}\) containing \((\tau_0,v_0)\). By interior
parabolic regularity, \(R\) is smooth in \(\mathcal{Q}\), and differentiating
\((\partial_\tau-\cL)R=f(v)\) with respect to \(v\) gives
\[
    (\partial_\tau-\cL)(\partial_vR)
    =
    f'(v)
    =
    -\frac1\gamma e^{(1-1/\gamma)v}<0
    \qquad\text{in }\ \mathcal{Q}.
\]
Since \(\partial_vR\le0\) and attains its maximum \(0\) at the
interior point \((\tau_0,v_0)\), the strong maximum principle gives
a contradiction. Hence
\[
    \partial_vR<0
    \qquad\text{in }\ \{R>p\}.
\]
\end{proof}

\subsection{Proof of Theorem \ref{thm:contact-geometry-R}}\label{proof:thm:contact-geometry-R}
\begin{proof}
For $\gamma\in(0,1)$, the characterization of the set of contact times
follows from Corollary~\ref{cor:no-contact-outside-T0} and
Lemma~\ref{lem:contact-existence-gamma-less-one}. Moreover,
Lemma~\ref{lem:necessary-contact-R} shows that contact can occur only
for $v>\Gamma(\tau)$ with $\Gamma(\tau)>-\infty$, and hence the noncontact set is nonempty at
every $\tau\in\mathcal T_\gamma$. For $\gamma\in(1,\infty)$,
nonemptiness of the contact set follows from
Lemma~\ref{lem:contact-existence-gamma-greater-one}, while
Lemma~\ref{lem:bounds-gamma-greater-one} shows that the noncontact
set is also nonempty.

Fix $\tau\in\mathcal T_\gamma$. By
Lemma~\ref{lem:state-monotonicity-R}, the map
$v\mapsto R(\tau,v)$ is nonincreasing. If
$R(\tau,v_0)=p(\tau)$ at some $v_0$, then, for every $v\ge v_0$,
$p(\tau)\le R(\tau,v)\le R(\tau,v_0)=p(\tau)$. Hence
$R(\tau,v)=p(\tau)$ for every $v\ge v_0$. Since both the contact
and noncontact sets are nonempty, the contact set is
$[\mathcal V(\tau),\infty)$ with $\mathcal V(\tau)$ finite, by the continuity of $R(\tau,\cdot)-p(\tau)$.

Finally, the inequality involving $\Gamma$ follows from
Lemma~\ref{lem:necessary-contact-R}, since
$(\tau,\mathcal V(\tau))$ itself belongs to the contact set.
\end{proof}

\subsection{Proof of Lemma \ref{lem:free-boundary-lsc}}\label{proof:lem:free-boundary-lsc}
\begin{proof}
Fix $\tau_0\in\mathcal T_\gamma$ and $\varepsilon>0$. By \eqref{non-contact},
$R(\tau_0,\mathcal V(\tau_0)-\varepsilon)>p(\tau_0)$, so by the continuity of
$R-p$ and the openness of $\mathcal T_\gamma$, there is $\delta>0$, with $(\tau_0-\delta,\tau_0+\delta)\subset\mathcal T_\gamma$,
such that $R(\tau,\mathcal V(\tau_0)-\varepsilon)>p(\tau)$ for
$|\tau-\tau_0|<\delta$. By \eqref{non-contact} again,
$\mathcal V(\tau)>\mathcal V(\tau_0)-\varepsilon$ on this interval, which is
the lower semicontinuity of $\mathcal V$.
\end{proof}

\subsection{Proof of Lemma \ref{lem:time-space-derivative-estimate}}\label{proof:lem:time-space-derivative-estimate}
\begin{proof}
Fix $\tau_0\in\mathcal T_\gamma$.
First choose $\kappa>0$ such that $[\tau_0-2\kappa,\tau_0+2\kappa]\subset\mathcal T_\gamma$ and, if $\gamma\in(0,1)$, $2\kappa<\nu(\tau_0)$ with $\nu(\tau_0)$ from Lemma~\ref{lem:contact-existence-gamma-less-one}. Then choose $\overline v>\mathcal V(\tau_0)$ and $\underline v<\mathcal V(\tau_0)$ such that
\[
    R(\tau,\underline v)>p(\tau),
    \qquad
    R(\tau,\overline v)=p(\tau)
    \qquad
    \text{for }\ |\tau-\tau_0|\le2\kappa.
\]
Such an $\underline v$ can be obtained by choosing
$\underline v<
\min_{|\tau-\tau_0|\le2\kappa}\mathcal V(\tau)$,
where the minimum is attained by the lower semicontinuity of
$\mathcal V$. The existence of $\overline v$ follows from
Lemmas~\ref{lem:contact-existence-gamma-less-one} and~\ref{lem:contact-existence-gamma-greater-one}.
Choosing $\overline v$ strictly larger than the resulting
uniform contact threshold, we obtain
$\underline v<\mathcal V(\tau)<\overline v$ for
$|\tau-\tau_0|\le2\kappa$.
Thus, the lateral boundary $v=\underline v$ of
\[
    \mathcal{Q}_{2\kappa}
    :=
    (\tau_0-2\kappa,\tau_0+2\kappa)
    \times(\underline v,\overline v)
\]
lies strictly inside $\{R>p\}$, whereas the lateral boundary $v=\overline v$ lies
in $\{R=p\}$. Since $\partial_vR<0$ in $\{R>p\}$, there exists $c>0$
such that
\begin{equation}\label{strict-left-derivative}
    \partial_vR(\tau,\underline v)\le-c
    \qquad
    \text{for }\ |\tau-\tau_0|\le2\kappa.
\end{equation}

We use a penalized approximation $R_\varepsilon$ to compare
$\partial_\tau R_\varepsilon-p'$ with $-\partial_vR_\varepsilon$
uniformly in $\varepsilon$, and then pass to the limit.
We first regularize the local initial datum, which need not be smooth
at the free boundary.

Set $\tau_1:=\tau_0-2\kappa$ and $g(v):=R(\tau_1,v)$ for $v\in\mathbb R$.
By the local optimal regularity result in \cite{Optimalregularity},
$g\in W^{2,\infty}_{\mathrm{loc}}(\mathbb R)$. Moreover, $g\ge p(\tau_1)$,
$g'\le0$, $g$ is smooth on $(-\infty,\mathcal V(\tau_1))$ with $g'<0$
there, and $g\equiv p(\tau_1)$ on $[\mathcal V(\tau_1),\infty)$.

Fix $\vartheta_0>0$ with $\underline v+2\vartheta_0<\mathcal V(\tau_1)$,
a cutoff $\chi\in C^\infty(\mathbb R)$ with $\chi=0$ on
$(-\infty,\underline v+\vartheta_0]$ and $\chi=1$ on
$[\underline v+2\vartheta_0,\infty)$, and a standard mollifier
$\phi_\varepsilon$, and put
\[
    g_\varepsilon:=\chi\,(g*\phi_\varepsilon)+(1-\chi)\,g .
\]
Then $g_\varepsilon=g$ near $\underline v$, $g_\varepsilon\equiv p(\tau_1)$
on $[\mathcal V(\tau_1)+\varepsilon,\infty)$, $g_\varepsilon\ge p(\tau_1)$,
$g_\varepsilon\to g$ in $W^{2,q}(\underline v,\overline v)$ for every
$1<q<\infty$, and
$\sup_{0<\varepsilon<1}\|g_\varepsilon\|_{W^{2,\infty}(\underline v,\overline v)}<\infty$.
Since $g*\phi_\varepsilon$ is nonincreasing, $g'\le-c'<0$ on
$[\underline v,\underline v+2\vartheta_0]$, and
$|g*\phi_\varepsilon-g|\le C\varepsilon^2$ there,
we also have $g_\varepsilon'\le0$ for all sufficiently small $\varepsilon$.
We also take $\varepsilon<\min\{1,\overline v-\mathcal V(\tau_1)\}$, so that $g_\varepsilon(\overline v)=p(\tau_1)$.
With such $g_\varepsilon$, let $R_\varepsilon$ solve the local penalized problem
\begin{equation}\label{local-penalized-R}
\begin{cases}
    \partial_\tau R_\varepsilon-\mathscr LR_\varepsilon
    =
    f(v)+\omega_\varepsilon(R_\varepsilon-p)
    &\text{in }\mathcal{Q}_{2\kappa},\\
    R_\varepsilon(\tau,\underline v)=R(\tau,\underline v),
    &|\tau-\tau_0|<2\kappa,\\
    R_\varepsilon(\tau,\overline v)=R(\tau,\overline v),
    &|\tau-\tau_0|<2\kappa,\\
    R_\varepsilon(\tau_0-2\kappa,v)=g_\varepsilon(v),
    &\underline v<v<\overline v,
\end{cases}
\end{equation}
where $\omega_\varepsilon\in C^\infty(\mathbb R)$ is a penalty function satisfying
$\omega_\varepsilon\ge0$, $\omega_\varepsilon'\le0$, $\omega_\varepsilon(s)=0$ for all
$s\ge\varepsilon$, and
\[
    \omega_\varepsilon(0)
    =
    \sup_{\mathcal{Q}_{2\kappa}}
    \bigl[p'(\tau)+\beta p(\tau)-f(v)\bigr].
\]
This supremum is positive and independent of $\varepsilon$, because $R(\tau,\overline v)=p(\tau)$ and Lemma~\ref{lem:necessary-contact-R} then gives $f(\overline v)<p'(\tau)+\beta p(\tau)$.
Then $(\partial_\tau-\mathscr L)p-f-\omega_\varepsilon(0)\le0$ in
$\mathcal{Q}_{2\kappa}$. Since $g_\varepsilon\ge p(\tau_0-2\kappa)$ on the initial
boundary and $R\ge p$ on the lateral boundary, the comparison
principle gives
\begin{equation}\label{local-penalty-above-obstacle}
    R_\varepsilon\ge p
    \qquad\text{in }\overline{\mathcal{Q}_{2\kappa}}.
\end{equation}
This implies
$0\le\omega_\varepsilon(R_\varepsilon-p)\le\omega_\varepsilon(0)$ and hence
$\omega_\varepsilon(R_\varepsilon-p)$ is bounded uniformly in $\varepsilon$.

The standard $W^{1,2}_q$ estimate, together with the uniform
$W^{2,\infty}$ bound for $g_\varepsilon$, gives a bound for $R_\varepsilon$ independent
of $\varepsilon$, for each $1<q<\infty$. By compactness and the comparison principle of Lemma~\ref{lem:existence-R} on $\mathcal Q_{2\kappa}$, which identifies every limit with $R$,
\[
    R_\varepsilon\rightharpoonup R
    \qquad\text{in }\ W^{1,2}_q(\mathcal{Q}_{2\kappa}),
\]
and the parabolic Sobolev embedding yields the uniform convergence of
$\partial_vR_\varepsilon$ to $\partial_vR$. Hence, by
\eqref{strict-left-derivative}, for all sufficiently small $\varepsilon>0$,
\begin{equation}\label{local-left-derivative}
    \partial_vR_\varepsilon(\tau,\underline v)\le-\frac c2
    \qquad
    \text{for }\ |\tau-\tau_0|\le2\kappa.
\end{equation}

We first prove the nonpositivity of $\partial_v R_\varepsilon$ in $\mathcal{Q}_{2\kappa}$.
For convenience, define
\[
    \mathscr L_\varepsilon
    :=
    \mathscr L+\omega_\varepsilon'(R_\varepsilon-p),
\]
where the last term is understood as a zeroth order coefficient.
Differentiating \eqref{local-penalized-R} with respect to $v$ gives
\[
    (\partial_\tau-\mathscr L_\varepsilon)\partial_vR_\varepsilon
    =
    f'(v)<0.
\]
On the initial boundary, $\partial_vR_\varepsilon=g_\varepsilon'\le0$. The sign on the
lateral boundary $v=\underline v$ follows from \eqref{local-left-derivative}. On
the lateral boundary $v=\overline v$, $R_\varepsilon(\tau,\overline v)=p(\tau)$ and
\eqref{local-penalty-above-obstacle} imply
$\partial_vR_\varepsilon(\tau,\overline v)\le0$. Since $\omega_\varepsilon'\le0$, the
maximum principle yields
\begin{equation}\label{local-monotonicity-penalty}
    \partial_vR_\varepsilon\le0
    \qquad\text{in }\overline{\mathcal{Q}_{2\kappa}}.
\end{equation}

We next obtain an $\varepsilon$-independent bound for
$\partial_\tau R_\varepsilon-p'$ in $\mathcal Q_{2\kappa}$ without using
any regularity at the corners where the initial and lateral boundaries meet. Let
\[
    K:=\sup_{0<\varepsilon<1}\|\mathscr Lg_\varepsilon\|_{L^\infty(\underline v,\overline v)}
    +\|f\|_{L^\infty(\underline v,\overline v)}+\omega_\varepsilon(0)
    +\|p'\|_{L^\infty(0,T)}
    +\sup_{|\tau-\tau_0|\le2\kappa}\bigl|\partial_\tau R(\tau,\underline v)\bigr|,
\]
which is finite and independent of $\varepsilon$ by the uniform
$W^{2,\infty}$ bound for $g_\varepsilon$ and by the interior Schauder
estimate for $R$ near the lateral boundary $v=\underline v$, which lies compactly
inside $\{R>p\}$ where $R$ solves $\partial_\tau R-\mathscr LR=f$
\citep{Friedman}.

Since $g_\varepsilon\ge p(\tau_1)$ and
$K\ge\|p'\|_{L^\infty(0,T)}$, the upper function
$g_\varepsilon(v)+K(\tau-\tau_1)$ satisfies $g_\varepsilon(v)+K(\tau-\tau_1)\ge p(\tau)$.
Thus its penalty is bounded above by $\omega_\varepsilon(0)$,
while for the lower function we only use $\omega_\varepsilon\ge0$.
By direct computation, the choice of $K$ and $\beta>0$ then imply
that $g_\varepsilon(v)\pm K(\tau-\tau_1)$ are a supersolution and
a subsolution, respectively, of the penalized equation in
$\mathcal Q_{2\kappa}$.
On $\partial_p\mathcal Q_{2\kappa}$ the first lies above and the second lies
below the boundary data in \eqref{local-penalized-R}. On the
initial slice they equal $g_\varepsilon$, on $v=\overline v$ they equal
$p(\tau_1)\pm K(\tau-\tau_1)$, and on $v=\underline v$ they equal
$R(\tau_1,\underline v)\pm K(\tau-\tau_1)$. Because $\omega_\varepsilon$
is nonincreasing, the difference of a supersolution and a solution
satisfies a linear inequality with a nonnegative zeroth order
coefficient, and the comparison principle gives
\begin{equation}\label{local-initial-lipschitz}
    |R_\varepsilon(\tau,v)-g_\varepsilon(v)|\le K(\tau-\tau_1)
    \qquad\text{in }\overline{\mathcal Q_{2\kappa}}.
\end{equation}
Now fix $h\in(0,\kappa)$ and set, on
$\mathcal Q^h:=(\tau_1,\tau_0+2\kappa-h)\times(\underline v,\overline v)$,
\[
    E_h(\tau,v):=\frac{R_\varepsilon(\tau+h,v)-R_\varepsilon(\tau,v)}{h}
    -\frac{p(\tau+h)-p(\tau)}{h}.
\]
Subtracting the penalized equations gives, for some bounded $\lambda_h\le0$,
\[
    (\partial_\tau-\mathscr L-\lambda_h)E_h
    =-\frac{p'(\tau+h)-p'(\tau)}{h}-\beta\,\frac{p(\tau+h)-p(\tau)}{h},
\]
whose right hand side is uniformly bounded since $p'$ is Lipschitz.
By \eqref{local-initial-lipschitz} and the lateral boundary data,
$|E_h|\le2K$ on $\partial_p\mathcal Q^h$.
As $\beta-\lambda_h\ge\beta>0$, the maximum principle gives
$|E_h|\le C$ independently of $\varepsilon$ and $h$.
Letting $h\to0^+$ yields
\begin{equation}\label{local-time-bound}
    \bigl|\partial_\tau R_\varepsilon-p'\bigr|\le C
    \qquad\text{in }\mathcal{Q}_{2\kappa},
\end{equation}
where $C$ is independent of $\varepsilon$. Differentiating
\eqref{local-penalized-R} with respect to $\tau$ gives, almost
everywhere,
\[
    (\partial_\tau-\mathscr L_\varepsilon)
    (\partial_\tau R_\varepsilon-p')
    =
    -p''-\beta p',
\]
with a bounded right hand side.

Consider now the smaller cylinder
\[
    \mathcal{Q}_\kappa
    :=
    (\tau_0-\kappa,\tau_0+\kappa)
    \times(\underline v,\overline v).
\]
Since $f'<0$ on $[\underline v,\overline v]$, we may choose $C_0>0$,
independently of $\varepsilon$, sufficiently large that
\[
    -C_0f'(v)
    \ge
    C+2\kappa\|p''+\beta p'\|_{L^\infty(0,T)}
    \qquad
    \text{for }\underline v\le v\le\overline v.
\]
Increasing $C_0$ further if necessary, we may also assume that
\[
    \frac{C_0c}{2}
    \ge
    2\kappa
    \sup_{|\tau-\tau_0|\le\kappa}
    \bigl|\partial_\tau R(\tau,\underline v)-p'(\tau)\bigr|.
\]
With this constant $C_0$, define
\[\Phi(\tau,v):=(\tau-\tau_0+\kappa)(\partial_\tau R_\varepsilon-p')
    +C_0(-\partial_vR_\varepsilon)
    \qquad \text{in }\ \mathcal{Q}_\kappa.\]
Using \eqref{local-time-bound} and the equations for
$\partial_\tau R_\varepsilon-p'$ and $\partial_vR_\varepsilon$, we have
\[
\begin{aligned}
(\partial_\tau-\cL_\varepsilon)\Phi
&=
(\partial_\tau-\cL_\varepsilon)
\Bigl[
    (\tau-\tau_0+\kappa)(\partial_\tau R_\varepsilon-p')
    +C_0(-\partial_vR_\varepsilon)
\Bigr]
\\
&=
\partial_\tau R_\varepsilon-p'
-(\tau-\tau_0+\kappa)(p''+\beta p')
-C_0f'(v)
\\
&\ge0
\end{aligned}
\]
in $\mathcal{Q}_\kappa$.
We next verify the sign of $\Phi$ on the parabolic boundary of $\mathcal{Q}_\kappa$. On
$\tau=\tau_0-\kappa$, the first term vanishes, so $\Phi$ is
nonnegative by \eqref{local-monotonicity-penalty}. On
$v=\overline v$, we have $\partial_\tau R_\varepsilon-p'=0$ and
$\partial_vR_\varepsilon\le0$, so it is again nonnegative. Finally, on
$v=\underline v$, \eqref{local-left-derivative} and the choice of
$C_0$ give
\[
\begin{aligned}
&
(\tau-\tau_0+\kappa)(\partial_\tau R_\varepsilon-p')
+C_0(-\partial_vR_\varepsilon)
\\
&\qquad\ge
-2\kappa
\bigl|\partial_\tau R(\tau,\underline v)-p'(\tau)\bigr|
+\frac{C_0c}{2}
\ge0.
\end{aligned}
\]
The comparison principle therefore yields
\[
    \partial_\tau R_\varepsilon-p'
    \ge
    -\frac{C_0}{\tau-\tau_0+\kappa}
    \bigl(-\partial_vR_\varepsilon\bigr)
    \qquad\text{in }\mathcal{Q}_\kappa.
\]
Applying the same argument to
   $\Psi(\tau,v):= (\tau-\tau_0+\kappa)(\partial_\tau R_\varepsilon-p')
    -C_0(-\partial_vR_\varepsilon)$,
we also see that $(\partial_\tau -\cL_\varepsilon)\Psi\le 0$ holds in $\mathcal{Q}_\kappa$ with 
$\Psi\le 0$ on the parabolic boundary of $\mathcal{Q}_\kappa$.
This implies
\[
    \partial_\tau R_\varepsilon-p'
    \le
    \frac{C_0}{\tau-\tau_0+\kappa}
    \bigl(-\partial_vR_\varepsilon\bigr)
    \qquad \text{in }\ \mathcal{Q}_\kappa.
\]

Combining these results, we deduce that
\begin{equation}\label{two-sided-derivative-estimate-penalty}
    \bigl|\partial_\tau R_\varepsilon-p'\bigr|
    \le
    \frac{C_0}{\tau-\tau_0+\kappa}
    \bigl(-\partial_vR_\varepsilon\bigr)
    \qquad
    \text{in }\ \mathcal{Q}_\kappa.
\end{equation}
Multiplying each of the two one sided inequalities in
\eqref{two-sided-derivative-estimate-penalty} by an arbitrary
nonnegative test function and integrating over $\mathcal{Q}_\kappa$, we may
pass to the limit as $\varepsilon\to0^+$ by the weak
$W^{1,2}_q(\mathcal{Q}_\kappa)$ convergence. 
Since the test function was chosen arbitrarily, we deduce that
\[
    \bigl|\partial_\tau R-p'\bigr|
    \le
    \frac{C_0}{\tau-\tau_0+\kappa}
    \bigl(-\partial_vR\bigr)
\]
almost everywhere in $\mathcal{Q}_\kappa$.
In the noncontact region $\{R>p\}$, the
interior Schauder regularity gives the continuity of
these derivatives, so the inequality holds pointwise in $\mathcal{Q}_\kappa\cap \{R>p\}$. Finally,
$\tau-\tau_0+\kappa\ge\kappa/2$ whenever
$|\tau-\tau_0|<\kappa/2$, and therefore
\[
    \bigl|\partial_\tau R(\tau,v)-p'(\tau)\bigr|
    \le
    \frac{2C_0}{\kappa}
    \bigl(-\partial_vR(\tau,v)\bigr).
\]

Relabeling $\kappa/2$ as $\kappa$ gives the assertion.
\end{proof}

\subsection{Proof of Lemma \ref{lem:free-boundary-holder}}\label{proof:lem:free-boundary-holder}
\begin{proof}

Fix $\tau_0\in\mathcal T_\gamma$, and take
$\kappa,\underline v,\overline v,C$ from
Lemma~\ref{lem:time-space-derivative-estimate}. Choose an open interval
$I$ containing $\tau_0$ such that
$\overline I\subset(\tau_0-\kappa,\tau_0+\kappa)$, and set
$\mathcal Q:=I\times(\underline v,\overline v)$. Then
$\underline v<\mathcal V(\tau)<\overline v$ for every
$\tau\in\overline I$.

    Note that $R-p$ satisfies
    \[ (\partial_\tau-\cL)(R-p) = f(v)-p'(\tau)-\beta p(\tau) \qquad\text{in }\ \mathcal{Q}\cap\{v<\mathcal{V}(\tau)\}. \]

By lower semicontinuity and the choice of $\underline v$, we may choose
\[
    0<\delta_1<\min_{\tau\in\overline I}
    \bigl(\mathcal V(\tau)-\underline v\bigr).
\]
The estimate in Lemma~\ref{lem:time-space-derivative-estimate}, with the
same constant $C$, then holds whenever
\[
    \tau\in \overline{I},
    \qquad
    \mathcal V(\tau)-\delta_1<v<\mathcal V(\tau).
\] 
The function
$\tau\mapsto p'(\tau)+\beta p(\tau)-f(\mathcal V(\tau))$
is strictly positive by Lemma~\ref{lem:necessary-contact-R} and lower
semicontinuous because $\mathcal V$ is lower semicontinuous and $f$
is continuous and decreasing. Compactness of $\overline I$ therefore gives
\[
    \nu:=
    \min_{\tau\in\overline I}\ 
    \Bigl[\ 
        p'(\tau)+\beta p(\tau)-f\big(\mathcal V(\tau)\big)
    \ \Bigr]
    >0.
\]

Spatial smooth fit and the uniform continuity of
$R$ and $\partial_vR$ on $\overline{\mathcal Q}$ imply that $R-p$ and
$\partial_vR$ tend to zero uniformly for $\tau\in\overline I$ as $v$
approaches $\mathcal V(\tau)$ from the noncontact region.
The estimate $|\partial_\tau R-p'|\le C(-\partial_vR)$ from
Lemma~\ref{lem:time-space-derivative-estimate}, with the same $C$
for all $\tau\in\overline I$, gives the same convergence for
$\partial_\tau R-p'$. Thus
\[
    R-p,\quad \partial_vR,\quad \partial_\tau R-p'\longrightarrow0
\]
uniformly for $\tau\in\overline I$ as
$v\to\mathcal V(\tau)$
from the noncontact region. Substituting these limits into the
equation and using the uniform continuity of $f$ on
$[\underline v,\overline v]$ gives
\[
    \lim_{v\to\mathcal V(\tau)^-}
    \frac{\theta^2}{2}\partial_{vv}R(\tau,v)
    =p'(\tau)+\beta p(\tau)-f(\mathcal V(\tau))
    \ge\nu,
\]
with uniform convergence in $\tau$. Hence, for some
$\delta\in(0,\delta_1)$,
\[\frac{\theta^2}{2}\partial_{vv}R(\tau,v)>\frac{\nu}{2}
\qquad \text{for all }\ v\in (\mathcal V(\tau)-\delta,\mathcal V(\tau))
\ \text{ and }\ \tau\in I. 
\]

Integrating twice and using smooth fit yields
\[R(\tau,v)-p(\tau)>\frac{\nu}{2\theta^2}(v-\mathcal{V}(\tau))^2
\qquad \text{for all }\ v\in (\mathcal V(\tau)-\delta,\mathcal V(\tau))
\ \text{ and }\ \tau\in I.\]
By Lemma~\ref{lem:existence-R} and the boundedness of $p'$, we have
    $\partial_\tau R-p'
    \in L^\infty_{\mathrm{loc}}(\Omega_T)$.
Hence, since $\overline{\mathcal{Q}}\Subset\Omega_T$, there exists $C_1>0$ such that
\[
    \bigl|\partial_\tau R-p'\bigr|\le C_1
    \qquad\text{a.e. in }\mathcal{Q}.
\]
Set $\delta_0:=\nu\delta^2/(8\theta^2C_1)$, and fix
$\tau_1,\tau_2\in I$ such that
$|\tau_1-\tau_2|<\delta_0$. We first claim that
$|\mathcal V(\tau_1)-\mathcal V(\tau_2)|<\delta$.
Suppose, to the contrary and without loss of generality, that
$\mathcal V(\tau_1)-\mathcal V(\tau_2)\ge\delta$. Then
$\mathcal V(\tau_1)-\delta/2>\mathcal V(\tau_2)$, and hence
$R(\tau_2,\mathcal V(\tau_1)-\delta/2)=p(\tau_2)$. On the other
hand, the preceding quadratic estimate yields
\[
    R\left(\tau_1,\mathcal V(\tau_1)-\frac{\delta}{2}\right)
    -p(\tau_1)
    >
    \frac{\nu\delta^2}{8\theta^2}.
\]
Therefore,
\[
\begin{aligned}
    \frac{\nu\delta^2}{8\theta^2}
    &<
    \left|
        (R-p)\left(\tau_1,\mathcal V(\tau_1)-\frac{\delta}{2}\right)
        -
        (R-p)\left(\tau_2,\mathcal V(\tau_1)-\frac{\delta}{2}\right)
    \right| \\
    &\le C_1|\tau_1-\tau_2|
    <
    \frac{\nu\delta^2}{8\theta^2},
\end{aligned}
\]
which is a contradiction. Thus the claim follows.

Interchanging $\tau_1$ and $\tau_2$ if necessary, suppose that
$\mathcal V(\tau_1)>\mathcal V(\tau_2)$. By the preceding claim,
$\mathcal V(\tau_2)\in
(\mathcal V(\tau_1)-\delta,\mathcal V(\tau_1))$. Hence, applying
the quadratic estimate at time $\tau_1$ with
$v=\mathcal V(\tau_2)$ gives
\[
    \frac{\nu}{2\theta^2}
    \bigl|\mathcal V(\tau_1)-\mathcal V(\tau_2)\bigr|^2
    \le
    R\bigl(\tau_1,\mathcal V(\tau_2)\bigr)-p(\tau_1).
\]
Since
$(R-p)(\tau_2,\mathcal V(\tau_2))=0$, we obtain
\[
    \frac{\nu}{2\theta^2}
    \bigl|\mathcal V(\tau_1)-\mathcal V(\tau_2)\bigr|^2
    \le
    (R-p)\bigl(\tau_1,\mathcal V(\tau_2)\bigr)
    -(R-p)\bigl(\tau_2,\mathcal V(\tau_2)\bigr)
    \le
    C_1|\tau_1-\tau_2|.
\]
Thus, we conclude that
\[
    \big\vert\mathcal V(\tau_1)-\mathcal V(\tau_2)\big\vert
    \le
    \left(\frac{2\theta^2C_1}{\nu}\right)^{1/2}
    |\tau_1-\tau_2|^{1/2}
\]
whenever $\tau_1,\tau_2\in I$ and
$|\tau_1-\tau_2|<\delta_0$.

Since $\tau_0$ was arbitrary, this proves
$\mathcal V\in C_{\mathrm{loc}}^{1/2}(\mathcal T_\gamma)$.

Finally, by the continuity of $\mathcal V$, every point
$(\tau,v)$ satisfying $R(\tau,v)=p(\tau)$ with
$v>\mathcal V(\tau)$ admits a neighborhood contained in the contact
region $\{R=p\}$. Hence,
\[
    \operatorname{int}\{R=p\}
    \supset
    \{R=p\}\setminus
    \{(\tau,\mathcal V(\tau)):\tau\in\mathcal T_\gamma\},
\]
and consequently,
\[
    \partial\{R=p\}
    \subset
    \{(\tau,\mathcal V(\tau)):\tau\in\mathcal T_\gamma\}.
\]
Since $(\tau,\mathcal{V}(\tau))$ is a free boundary point for every $\tau\in \mathcal{T}_\gamma$, we conclude that
\[\partial\{R>p\}=\partial\{R=p\}=\{(\tau,\mathcal V(\tau)):\tau\in\mathcal T_\gamma\}.\]
\end{proof}

\subsection{Proof of Proposition \ref{prop:boundary-higher-regularity}}
\begin{proof}
Put $W=R-p$, and choose compact intervals
$\mathcal I\subset\operatorname{int}\mathcal I_1\subset \mathcal I_1\Subset\mathcal T_\gamma$.
Local optimal regularity, Lemma~\ref{lem:time-space-derivative-estimate},
and the positivity argument for $R_{vv}$ in the proof of
Lemma~\ref{lem:free-boundary-holder} give constants $c,M,L,\delta>0$ such that
\[
 c\le W_{vv}\le M,\qquad |W_\tau|\le L(-W_v)
 \quad\text{if }\tau\in \mathcal I_1,\quad
 0<\mathcal V(\tau)-v<\delta.
\]
The constants and the tube width can be chosen
uniformly by a finite covering of the compact boundary graph over
$\mathcal I_1$.

We first prove the Lipschitz assertion.
The quadratic lower bound in the proof of
Lemma~\ref{lem:free-boundary-holder} and strict state monotonicity give,
for sufficiently small $\eta>0$, a
unique level curve $v_\eta$ in this tube defined by
$W(\tau,v_\eta(\tau))=\eta$. The implicit function theorem gives
\[
 v_\eta'(\tau)=
 \frac{W_\tau(\tau,v_\eta(\tau))}{-W_v(\tau,v_\eta(\tau))},
 \qquad |v_\eta'|\le L.
\]
After decreasing $c$ if necessary, this lower bound gives
$0\le\mathcal V-v_\eta\le(2\eta/c)^{1/2}$, uniformly on $\mathcal I_1$.
Thus $\mathcal V$ is $L$-Lipschitz on $\mathcal I_1$.
In particular, this part of the proof already supplies the Lipschitz
property of $\log b$ used in Lemma~\ref{lem:reflection-moments}.

For the higher regularity, set $u_1=W_\tau$, $u_2=-W_v$, and
$\mathfrak d=\theta^2/2>0$. In the continuation part of the tube,
\[
\begin{aligned}
 (\partial_\tau-\mathfrak d\partial_{vv}-b_0\partial_v+\beta)u_1
 &=-p''-\beta p',\\
 (\partial_\tau-\mathfrak d\partial_{vv}-b_0\partial_v+\beta)u_2
 &=-f'.
\end{aligned}
\]
Both right hand sides are bounded uniformly in this tube, where
$p''\in L^\infty$ is understood almost everywhere.
The function $u_2$ extends continuously as zero to the free boundary
by smooth fit. The estimate $|u_1|\le Lu_2$ then shows
that $u_1$ also extends continuously as zero there. Thus both
equations have continuous zero boundary data.

Flatten the boundary by setting
$U_i(\tau,x)=u_i(\tau,\mathcal V(\tau)-x)$ for $0<x<\delta$.
Since $\mathcal V$ is Lipschitz, these functions satisfy, almost everywhere,
\[
 U_{i,\tau}-\mathfrak d U_{i,xx}
 +(b_0+\mathcal V'(\tau))U_{i,x}+\beta U_i=g_i,
 \qquad U_i(\tau,0)=0,
\]
where $g_1=-p''-\beta p'$ and
$g_2(\tau,x)=-f'(\mathcal V(\tau)-x)$.
The drift and the sources are bounded, and $U_2(\tau,x)\ge cx$
by $W_{vv}\ge c$ and spatial smooth fit.

By localization and the boundary $L^q$ theory
\citep[Chapter~IV, Theorem~9.1]{Ladyzhenskaya}, we obtain
$U_i\in W_q^{1,2}(\mathcal I\times(0,\delta/2))$ for every $q>3$,
with regularity up to the flat boundary $x=0$.
Choose smooth cutoffs $\chi_0(\tau)$, supported in
$\operatorname{int}\mathcal I_1$ and equal to one near $\mathcal I$,
and $\chi_1(x)$, equal to one on $[0,\delta/2]$ and vanishing near
$x=\delta$. Then $Z_i=\chi_0\chi_1U_i$ has zero initial and spatial
boundary data. In its equation, the terms containing derivatives of
$\chi_1$ are supported away from $x=0$ and bounded by interior
regularity, while the time cutoff term $\chi_0'\chi_1U_i$ is bounded
because $U_i$ is bounded. Thus the right hand side is bounded, and
the boundary $L^q$ existence theorem gives a $W_q^{1,2}$ solution
with this right hand side and zero data. The function $Z_i$ is
continuous up to the boundary and is a strong solution in the interior.
Uniqueness of continuous strong solutions with bounded measurable drift identifies the localized function with the $W^{1,2}_q$ solution given by the boundary $L^q$ theory. The same localization is used for the Schauder estimates below.

Write $C_p^s=C^{s,s/2}$ for the parabolic H\"older class.
The parabolic Sobolev embedding
\citep[Chapter~II, Lemma~3.3]{Ladyzhenskaya} yields
$U_{i,x}\in C_p^{1-3/q}$ up to $x=0$.
Since $U_{2,x}(\tau,0)\ge c$, the function
\[
 \chi(\tau)=\frac{U_{1,x}(\tau,0)}{U_{2,x}(\tau,0)}
 \quad\text{belongs to }C^{(1-3/q)/2}(\mathcal I).
\]
The identity
$U_i(\tau,x)/x=\int_0^1U_{i,x}(\tau,sx)\,ds$
shows that $U_1/U_2\to\chi$ uniformly on $\mathcal I$ as $x\to0^+$.

The uniform convergence $v_\eta\to\mathcal V$ and this quotient
estimate imply $v_\eta'\to\chi$ uniformly on $\mathcal I$.
Passing to the limit in
$v_\eta(\tau)-v_\eta(s)=\int_s^\tau v_\eta'(u)\,du$
shows that $\mathcal V'=\chi$.
For any $0<\alpha<1/2$, choose $q>3/(1-2\alpha)$.
Then $\mathcal V\in C^{1,\alpha}(\mathcal I)$ for every $\alpha<1/2$.
All constants may depend on $\mathcal I_1$ and on its distance from the endpoints
of the component of the active time set that contains it. The estimates need not
be uniform near times at which the boundary appears or disappears.

Suppose $p\in C^{m+1}((0,T))$ for an integer $m\ge2$,
and fix $0<\alpha<\tfrac12$. If $\mathcal V'\in C^\rho$ locally,
with $0<\rho<m-1$ and $2\rho$ noninteger, then the coefficients and
sources in the equations for $U_i$ belong locally to $C_p^{2\rho}$.
The local boundary Schauder estimates on smaller cylinders away from
the initial time
\citep[Chapter~IV, Theorems~5.2 and~10.1]{Ladyzhenskaya}
give $U_i\in C_p^{2\rho+2}$ up to $x=0$.
Thus $U_{i,x}(\cdot,0)\in C^{\rho+1/2}$, and the formula
$\mathcal V'=U_{1,x}(\cdot,0)/U_{2,x}(\cdot,0)$ gives
$\mathcal V'\in C^{\rho+1/2}$ locally.
Starting with $\rho=\alpha$ and applying this step $2(m-1)$ times yields
$\mathcal V\in C^{m,\alpha}$ locally. The largest input exponent is
$m-\tfrac32+\alpha<m-1$, so the stated regularity of $p$ suffices
at every step. The smooth case follows by taking $m$ arbitrary.
\end{proof}

\subsection{Proof of Lemma \ref{lem:boundary-T0}}\label{proof:lem:boundary-T0}
\begin{proof}
Suppose, to the contrary, that there exist
$\tau_n\in\mathcal T_0$ with $\tau_n\to\tau_*$ and
$M\in\mathbb R$ such that
\[
    \mathcal V(\tau_n)\le M
    \qquad\text{for all }\ n\in \mathbb N.
\]
Since the contact set at time $\tau_n$ is
$[\mathcal V(\tau_n),\infty)$, we have
\[
    R(\tau_n,M)=p(\tau_n)
    \qquad\text{for all }\ n\in \mathbb N.
\]
By the continuity of $R$ and $p$, taking $n\to \infty$ yields
    $R(\tau_*,M)=p(\tau_*)$.
Moreover, since $\mathcal T_0$ is open,
we have $\tau_*\notin\mathcal T_0$. Hence
Corollary~\ref{cor:no-contact-outside-T0} gives
\[
    R(\tau_*,v)>p(\tau_*)
    \qquad\text{for every }v\in\mathbb R,
\]
which is a contradiction.
\end{proof}

\subsection{Proof of Lemma \ref{lem:endpoint-zero}}\label{proof:lem:endpoint-zero}
\begin{proof}
\emph{Lower bound.} For $v<v_B$ we have $R(0,v)=B(v)>p(0)$, so
$R(\tau,v)>p(\tau)$, and hence $\mathcal V(\tau)>v$, for all sufficiently
small $\tau>0$. Letting $v\to v_B^-$ gives
$\liminf_{\tau\to0^+}\mathcal V(\tau)\ge v_B$. If $\Gamma(0)$ is finite,
then $\mathcal V(\tau)>\Gamma(\tau)\to\Gamma(0)$ for $\tau\in\mathcal T_\gamma$,
so $\liminf_{\tau\to0^+}\mathcal V(\tau)\ge\Gamma(0)$. If $\gamma\in(0,1)$
and $\Gamma(0)=+\infty$, then $\Gamma(\tau)\to+\infty$, and
$\mathcal V\ge\Gamma$ in the extended sense (with $\mathcal V=+\infty$ off the
active set) gives $\mathcal V(\tau)\to+\infty$. It remains to prove the upper
bound when $\max\{v_B,\Gamma(0)\}<+\infty$.

\emph{Upper bound.} Fix $v_0>\max\{v_B,\Gamma(0)\}$. Since $v_0>\Gamma(0)$,
the definition of $\Gamma(0)$ and the strict monotonicity of $f$ give
$f(v_0)<p'(0)+\beta p(0)$, and this is immediate when $\Gamma(0)=-\infty$.
Choose $\kappa>0$ so small that $v_0-\kappa>v_B$ and $f<p'(0)+\beta p(0)$ on
$[v_0-\kappa,v_0+\kappa]$. By continuity there are $\delta,\eta>0$ such that
\begin{equation}\label{endpoint-negative-source}
    f(v)-p'(\tau)-\beta p(\tau)\le-\eta
    \qquad\text{for }0\le\tau\le\delta,\ |v-v_0|\le\kappa.
\end{equation}
Put $W:=R-p$, so that $W\ge0$ and
$(\partial_\tau-\mathscr L)W\ge f-p'-\beta p$ almost everywhere, with
equality in $\{W>0\}$. For $c>0$ sufficiently small,
\[
    (\partial_\tau-\mathscr L)\bigl[c(v-v_0)^2\bigr]
    =-\theta^2c-2cb_0(v-v_0)+\beta c(v-v_0)^2
    \ge-\eta
    \qquad\text{for }|v-v_0|\le\kappa,
\]
so by \eqref{endpoint-negative-source} the function $c(v-v_0)^2$ is a
supersolution of the problem for $W$ on $(0,\delta)\times(v_0-\kappa,v_0+\kappa)$.
Since $v_0-\kappa>v_B$, the initial condition gives $W(0,\cdot)=0$ on
$[v_0-\kappa,v_0+\kappa]$. With $c$ fixed, continuity gives
$W(\tau,v_0\pm\kappa)\le c\kappa^2$ for $0\le\tau\le\delta$ after decreasing $\delta$ if necessary. Thus
$W\le c(v-v_0)^2$ on the parabolic boundary of the cylinder, and the
comparison principle of Lemma~\ref{lem:existence-R} on this bounded cylinder, applied to $\zeta=p+c(v-v_0)^2$, yields $0\le W\le c(v-v_0)^2$ throughout. Evaluating at
$v=v_0$ gives $R(\tau,v_0)=p(\tau)$, hence $\mathcal V(\tau)\le v_0$, for
$0<\tau\le\delta$. Since $v_0>\max\{v_B,\Gamma(0)\}$ was arbitrary,
$\limsup_{\tau\to0^+}\mathcal V(\tau)\le\max\{v_B,\Gamma(0)\}$, which
together with the lower bound proves the result.
\end{proof}

\subsection{Proof of Proposition \ref{prop:construct-Q}}
\begin{proof}
The unconstrained marginal value $G^0$ solves
$G_t^0+\mathcal LG^0+u(I(\xi))=0$ with terminal value $B(\log\xi)$.
Subtract its stopped martingale identity from
\eqref{eq:R-stopping-representation}. This represents $D$ as an
optimal stopping gain with reward $P-G^0$ before the horizon and terminal
reward $(P(T)-B)^+$. In particular $D\ge0$.

Since $K$ is positive and bounded away from zero and $P$ has a fixed
strict sign, there is $d_*>0$ such that $P(s)-G^0(s,\eta)<0$
for every $s\in[0,T]$ and $0<\eta<d_*$. The positive part of either
stopping reward is bounded by $C(1+\eta^{a^+})$ and is zero below
$d_*$. Write $\mathcal Y_s^{t,\xi}=\xi E_s^t$ and
$S=\sup_{t\le s\le T}E_s^t$. Every positive stopping gain is thus
bounded by
\[
C\{1+(\xi S)^{a^+}\}\mathbf1_{\{\xi S\ge d_*\}}.
\]
All moments of $S$ are bounded uniformly over $t$. For any $m>0$,
use $\mathbf1_{\{\xi S\ge d_*\}}\le(\xi S/d_*)^m$ and
$\xi\le1$ to obtain \eqref{eq:D-small-state}. These estimates also
give $D\le C(1+\xi^{a^+})$ for $\xi\ge1$.

The integral in \eqref{eq:Q-constructive} therefore converges at zero.
Direct differentiation proves
$\widehat Q-\xi\widehat Q_\xi=G$ and
$\widehat Q_{\xi\xi}=-G_\xi/\xi$. Since $a<1$, integration at infinity
and the bound for small states prove \eqref{eq:Q-growth-explicit}.
At $T$ the same construction from $G(T,\xi)=\max\{B(\log\xi),P(T)\}$
gives exactly $\widetilde U_B$. The normalization for small states fixes
the otherwise arbitrary linear term.

For the time equation put
\[
k(t,\eta):=\beta P(t)-P'(t)-u(I(\eta)),\qquad
\mathcal{C}_t:=\{\eta:G(t,\eta)=P(t)\}.
\]
The obstacle equation and smooth fit give, in distributions,
\[
D_t+\mathcal LD=-k(t,\eta)\mathbf1_{\mathcal{C}_t}.
\]
Here $k\ge0$ on $\mathcal{C}_t$, by Lemma~\ref{lem:necessary-contact-R}.
The operator $\phi\mapsto-\xi\int_0^\xi\phi(\eta)\eta^{-2}\,d\eta$
commutes with $\partial_t+\mathcal L$. One verifies this by integration
by parts on $(\varepsilon,\xi)$ and lets $\varepsilon\to0^+$, and
\eqref{eq:D-small-state} removes the boundary terms at the lower end. For weak
derivatives the same calculation is made against a compactly supported
test function. Local parabolic estimates applied below $d_*$ justify
the corresponding derivative bounds at the lower endpoint.
Consequently,
\begin{equation}\label{eq:Q-HJB-residual}
\widehat Q_t+\mathcal L\widehat Q+\widetilde u(\xi)+\epsilon\xi
=\xi\int_0^\xi\frac{k(t,\eta)\mathbf1_{\mathcal{C}_t}}{\eta^2}\,d\eta.
\end{equation}
The right side is nonnegative and is zero in the continuation region
and at its boundary. It is continuous at interior dates, since
$\mathcal{C}_t=[b(t),\infty)$ when nonempty and the reciprocal boundary is
continuous, including at disappearance times. Also $G_\xi$ is
continuous. Thus \eqref{eq:Q-HJB-residual} identifies a continuous
time derivative and proves the asserted $C^{1,2}$ regularity.
Together with $G-P\ge0$, this proves the reduced variational inequality.
\end{proof}

\subsection{Proof of Lemma \ref{lem:reflection-moments}}
\begin{proof}
For $0<\gamma<1$, contact requires
$u(I(\xi))<\beta P(t)-P'(t)$. The right hand side is bounded,
whereas the left side tends to $+\infty$ as $\xi\to0^+$.
For $\gamma>1$, the uniform estimate for $\tau\in[0,T]$
in the proof of Lemma~\ref{lem:bounds-gamma-greater-one} yields
$b(t)\ge e^{v_*}$ for $0\le t<T$.
Include the positive terminal threshold $q_B$ in the lower bound to
obtain $b_*>0$. The boundary regularity and endpoint results give the
assertions about $1/b$, including at $t=0$ by Remark~\ref{rem:endpoint-T}.

The running supremum now gives \eqref{eq:A-moment-bound}.
Both $\sup_s\mathcal Y_s^{t,y}$ and
$(\inf_s\mathcal Y_s^{t,y})^{-1}$ have moments of every order on
a finite horizon. Hence
\[
\frac{\inf_s\mathcal Y_s^{t,y}}
{1+z+b_*^{-1}\sup_s\mathcal Y_s^{t,y}}
\le\Xi_s^{*,t,y,z}\le\frac{\mathcal Y_s^{t,y}}{1+z}.
\]
The assertions on the moments of consumption and terminal wealth follow.

On any compact interval of finite interior boundary values,
$\log b$ is Lipschitz by the part of the proof of
Proposition~\ref{prop:boundary-higher-regularity} that concerns level curves. The
higher regularity is not needed here. Thus
$U_s=\log\Xi_s^*-\log b(s)$ is a continuous semimartingale
with quadratic variation $d\langle U\rangle_s=\theta^2ds$.
The occupation density formula gives
$\int\mathbf1_{\{U_s=0\}}\,d\langle U\rangle_s=0$.
The active time set is open and has a countable compact exhaustion, and
its complement has no finite contact boundary. Initial and terminal
atoms do not affect a time integral. This proves
\eqref{eq:boundary-zero-occupation} (see \citet{RY1999}).
\end{proof}

\subsection{Proof of Proposition \ref{prop:dual-shadow-reflection}}\label{proof:prop:dual-shadow-reflection}
\begin{proof}
Set
\[
 \Upsilon_s^{t,y}:=\frac{\mathcal Y_s^{t,y}}{b(s)}.
\]
By definition,
$A_s^*=(1+z)\vee\sup_{t\le u\le s}\Upsilon_u^{t,y}$ is the smallest
right continuous nondecreasing process with value $1+z$ before intervention that
dominates $\Upsilon^{t,y}$.  Continuity of $1/b$ before $T$ and its terminal left limit show that $A^*$ is c\`adl\`ag, with no interior jumps. Lemma~\ref{lem:reflection-moments} gives its moments, so $\mathcal Z^*=A^*-1$ belongs to $\Pi_t(z)$.  Since $A_s^*\ge \Upsilon_s^{t,y}$,
\[
 \Xi_s^*=\frac{\mathcal Y_s^{t,y}}{A_s^*}\le b(s),
\]
which proves \eqref{eq:state-constraint-reflection}.

The running supremum can increase only when the current value
$\Upsilon_s^{t,y}$ attains the current running maximum.  Therefore, on the support of
$dA_s^*=d\mathcal Z_s^*$,
\[
 A_s^*=\Upsilon_s^{t,y}=\frac{\mathcal Y_s^{t,y}}{b(s)},
\]
whenever the increment is nonzero, and hence $\Xi_s^*=b(s)$.  This remains
true for atoms. If the boundary falls or the initial state lies above the
boundary, $A^*$ jumps to the smallest value that restores
$\Xi_s^*=b(s)$.  Thus \eqref{eq:Skorokhod-support} holds.  Since
$N(s,b(s))=0$ whenever $b(s)<\infty$, while no increase occurs when
$b(s)=+\infty$, \eqref{eq:N-complementarity-exact} follows.
\end{proof}

\subsection{Proof of Theorem \ref{thm:dual-verification}}\label{proof:thm:dual-verification}
\begin{proof}
Fix $(t,y,z)$ and let $\mathcal Z\in\Pi_t(z)$ be arbitrary.  Write
\[
 A_s:=1+\mathcal Z_s,\qquad
 \Xi_s:=\frac{\mathcal Y_s^{t,y}}{A_s},\qquad
 F(s,Y,A):=A \widehat Q\!\left(s,\frac YA\right).
\]
For fixed $(s,Y)$,
\begin{equation}\label{eq:F-A-derivative-proof}
 \partial_A F(s,Y,A)
 =\widehat Q(s,Y/A)-(Y/A)\widehat Q_\xi(s,Y/A)
 =G(s,Y/A)\ge P(s).
\end{equation}
Between jumps of $A$, the generalized It\^o formula and
\eqref{eq:HJB_Q_reduced} give, after discounting,
\begin{align*}
 d\bigl(e^{-\beta(s-t)}F(s,\mathcal Y_s^{t,y},A_s)\bigr)
 &\ge-e^{-\beta(s-t)}
 \Bigl(A_s\widetilde u(\Xi_s)+\mathcal Y_s^{t,y}\epsilon\Bigr)ds\\
 &\quad+e^{-\beta(s-t)}P(s)\,dA_s^c+dM_s,
\end{align*}
where $A^c$ is the continuous finite variation part and $M$ is a local
martingale.  At a jump time $s$,
\begin{align*}
&F(s,\mathcal Y_s^{t,y},A_s)
 -F(s,\mathcal Y_s^{t,y},A_{s-})\\
&\qquad=\int_{A_{s-}}^{A_s}
 G\!\left(s,\frac{\mathcal Y_s^{t,y}}\zeta\right)d\zeta
 \ge P(s)\Delta A_s,
\end{align*}
by \eqref{eq:F-A-derivative-proof}.  Thus, after localization, integration,
and expectation,
\begin{align*}
 (1+z)\widehat Q\!\left(t,\frac y{1+z}\right)
 \le\mathbb E_t\bigg[&
 \int_t^T e^{-\beta(s-t)}
 \Bigl((1+\mathcal Z_s)\widetilde u(\Xi_s)
       +\mathcal Y_s^{t,y}\epsilon\Bigr)ds\\
 &+e^{-\beta(T-t)}(1+\mathcal Z_T)\widetilde U_B(\Xi_T)
 -\int_{[t,T]}e^{-\beta(s-t)}P(s)\,d\mathcal Z_s
 \bigg].
\end{align*}
Taking the infimum over $\mathcal Z$ yields the lower bound for the dual value.

Now take $\mathcal Z=\mathcal Z^{*,t,y,z}$.  By
Proposition~\ref{prop:dual-shadow-reflection}, the control increases only when
$\Xi^*=b$, so $G=P$ on the support of its continuous finite variation part.
At any jump of $A^*$, the state moves through the interval
$\xi\ge b(s)$, which lies entirely in the intervention region. Hence
$G(s,\xi)=P(s)$ throughout the jump and the jump inequality above is an
equality.  The reflected state always lies in the closure of the waiting region. The explicit residual \eqref{eq:Q-HJB-residual} vanishes throughout this closure, so the time inequality is also an equality.  Therefore every inequality in the preceding It\^o
calculation becomes an equality for $\mathcal Z^*$.  Using
$\widehat Q(T,\xi)=\widetilde U_B(\xi)$ proves that its criterion equals
$(1+z)\widehat Q(t,y/(1+z))$.  Hence $\mathcal Z^*$ is optimal and
\eqref{eq:dual-verification-value} follows.  The localization is removed by \eqref{eq:Q-growth-explicit}, the explicit multiplier moment condition, and Lemma~\ref{lem:reflection-moments}. Indeed, after discounting, the absolute diffusion coefficient is bounded by $C(A_s+\mathcal Y_s+A_s^{1/\gamma}\mathcal Y_s^a)$ and is square integrable. Thus the local martingale is a true martingale. This identifies $\widehat Q=Q$ and supplies all of the analytic properties used subsequently.
\end{proof}

\subsection{Proof of Proposition \ref{prop:shadow-welfare-identity}}\label{proof:prop:shadow-welfare-identity}
\begin{proof}
Let $G=Q-\xi Q_\xi$.  In the waiting region, differentiate the continuation
equation in \eqref{eq:HJB_Q_reduced} with respect to $\xi$ and subtract
$\xi$ times the differentiated equation from the original equation.  This
gives
\begin{equation}\label{eq:G-linear-equation-proof}
 G_t+\mathcal L G+\bigl(\widetilde u(\xi)-\xi\widetilde u'(\xi)\bigr)=0.
\end{equation}
Fenchel equality and $\widetilde u'(\xi)=-I(\xi)$ imply
\begin{equation}\label{eq:fenchel-G-running-proof}
 \widetilde u(\xi)-\xi\widetilde u'(\xi)=u(I(\xi)).
\end{equation}
At terminal time, the constrained envelope formula
$\widetilde U_B'(\xi)=-\hat x_B(\xi)$ yields
\begin{equation}\label{eq:G-terminal-proof}
 G(T,\xi)
 =\widetilde U_B(\xi)-\xi\widetilde U_B'(\xi)
 =U_B(\hat x_B(\xi)).
\end{equation}
This identity holds on both sides of the terminal kink because
$\widetilde U_B$ is $C^1$ there.

We justify the generalized It\^o formula before applying it to
$e^{-\beta(s-t)}G(s,\Xi_s^y)$. Localize to a compact subset of
$(0,T)\times(0,\infty)$ and stop when $A^y$ exceeds a fixed level.
The change of variables $G(t,\xi)=R(T-t,\log\xi)$ gives
$G\in W^{1,2}_{\infty,\mathrm{loc}}$ by Lemma~\ref{lem:existence-R}.
Spatial smooth fit gives $G_\xi=0$ on the contact side and on the
boundary. Lemma~\ref{lem:time-space-derivative-estimate} gives
$W_\tau\to0$ from the continuation side, and hence
$G_t\to P'(t)$ there. Thus both first derivatives extend continuously
across every finite interior free boundary. Local optimal regularity
also bounds $G_{\xi\xi}$ on the localized set.

The diffusion coefficient of $\Xi^y$ is $-\theta\Xi^y$, whose
absolute value is bounded away from zero after this localization.
These are the local Sobolev and nondegeneracy conditions of the
generalized It\^o formula in \citet[Chapter~2, \S10]{Krylov1980}.
To include the possibly singular continuous finite variation part
of the cumulative multiplier, mollify $G$ on a slightly larger
compact set and apply the classical formula. The first derivatives
of the mollifications converge uniformly, so the finite variation
integrals converge, as do the localized stochastic integrals in
$L^2$. Off the free boundary the second derivatives converge
pointwise and remain uniformly bounded. By
\eqref{eq:boundary-zero-occupation}, the reflected state spends
zero Lebesgue time on the finite free boundary, so dominated convergence
also applies to the second order time integral.
Consequently the formula has no local time correction. Initial and
terminal jumps are treated by their actual function increments,
separately from this continuous calculation.

When the cumulative multiplier is inactive,
\eqref{eq:G-linear-equation-proof}--\eqref{eq:fenchel-G-running-proof}
give the discounted running term
$-e^{-\beta(s-t)}u(c_s^y)\,ds$. The zero occupation property ensures
that this identifies the drift for almost every time.
When the cumulative multiplier acts continuously, $\Xi_s^y=b(s)$
and $G_\xi=N_\xi=0$, so the finite variation reflection term
vanishes. At a jump of the cumulative multiplier, the states before and
after intervention lie in the intervention region, where
$G(s,\xi)=P(s)$. Thus the function increment over the jump is zero.
Finally, Lemma~\ref{lem:reflection-moments} and the growth bound for
$G$ justify the limits at the initial and terminal dates and the
removal of localization. Taking conditional expectations yields
\[
 G(t,\Xi_t^y)
 =\mathbb E_t\!\left[
 \int_t^T e^{-\beta(s-t)}u(c_s^y)\,ds
 +e^{-\beta(T-t)}U_B(X_T^y)
 \right],
\]
which is \eqref{eq:shadow-welfare-identity}.  Since $N\ge0$, this immediately
gives $\mathcal W_t^y\ge P(t)$.  Finally,
Proposition~\ref{prop:dual-shadow-reflection} says that $d\mathcal Z^{*,y}$ is
supported on $N(t,\Xi_t^y)=0$, so
\eqref{eq:welfare-complementarity} follows.
\end{proof}

\subsection{Proof of Proposition \ref{prop:dynamic-credit-limit}}
\label{proof:prop:dynamic-credit-limit}
\begin{proof}
Fix $t\in[0,T]$ and $\xi>0$.  Let $\Phi_t(\xi,\mathcal Z)$ denote the dual
criterion in \eqref{eq:three_dim_singular_finite} with initial state
$(t,\xi,0)$ and a fixed admissible control $\mathcal Z\in\Pi_t(0)$.  Since
\[
 \mathcal Y_s^{t,\xi}=\xi e^{\beta(s-t)}\mathcal H_s^t,
\]
differentiation with respect to $\xi$ gives, by the explicit multiplier moments and the CRRA bounds,
\begin{align}
 \partial_\xi\Phi_t(\xi,\mathcal Z)
 &=\mathbb E_t\!\left[
 \int_t^T \mathcal H_s^t
 \left\{\widetilde u'\!\left(
 \frac{\mathcal Y_s^{t,\xi}}{1+\mathcal Z_s}\right)
 +\epsilon\right\}ds\right.\notag\\
 &\hspace{3.1cm}\left.
 +\mathcal H_T^t\widetilde U_B'\!\left(
 \frac{\mathcal Y_T^{t,\xi}}{1+\mathcal Z_T}\right)
 \right].\label{eq:dynamic-envelope-fixed-control-proof}
\end{align}
At the reflected optimizer, Fenchel equality yields
$\widetilde u'(q)=-I(q)$ and
$\widetilde U_B'(q)=-\hat x_B(q)$.  Hence
\begin{equation}\label{eq:dynamic-envelope-cost-proof}
 \partial_\xi\Phi_t(\xi,\mathcal Z^{*,t,\xi})
 =-\mathcal B_t(\xi).
\end{equation}
The explicit formula
\[
 1+\mathcal Z_s^{*,t,\xi}
 =1\vee\xi\sup_{t\le u\le s}
 \frac{e^{\beta(u-t)}\mathcal H_u^t}{b(u)}
\]
shows that the optimal reflected control depends continuously on $\xi$,
pathwise and uniformly on $[t,T]$.  Applying the usual two sided
envelope inequalities with the optimizer at $\xi$ and at $\xi+h$, followed by
dominated convergence, gives
\[
 Q_\xi(t,\xi)=-\mathcal B_t(\xi).
\]
This proves \eqref{eq:dynamic-budget-envelope}.

We next identify the geometry of this map.  Put $\tau=T-t$ and $v=\ln\xi$.
Since
\[
 R(\tau,v)=G(t,\xi)=Q(t,\xi)-\xi Q_\xi(t,\xi),
\]
we have in the waiting region
\begin{equation}\label{eq:R-v-Qxx-proof}
 \partial_vR(\tau,v)
 =\xi\partial_\xi G(t,\xi)
 =-\xi^2Q_{\xi\xi}(t,\xi).
\end{equation}
Lemma~\ref{lem:state-monotonicity-R} gives $\partial_vR<0$ in the noncontact
region.  Therefore $Q_{\xi\xi}>0$ there, and
$x(t,\xi)=-Q_\xi(t,\xi)$ is strictly decreasing throughout the waiting region.
Continuity follows from Proposition~\ref{prop:construct-Q}.

Suppose first that $b(t)<\infty$.  On the intervention region
$[b(t),\infty)$ we have
\[
 Q(t,\xi)-\xi Q_\xi(t,\xi)=P(t).
\]
Equivalently,
\[
 \frac{d}{d\xi}\left(\frac{Q(t,\xi)}{\xi}\right)
 =-\frac{P(t)}{\xi^2}.
\]
Thus $Q(t,\xi)=P(t)+\varpi(t)\xi$ throughout the intervention region.  Hence
$Q_\xi(t,\xi)=\varpi(t)$ there, and continuity at $b(t)$ gives
\[
 -Q_\xi(t,\xi)=-Q_\xi(t,b(t))=\underline x(t),
 \qquad \xi\ge b(t).
\]
Together with strict decrease on $(0,b(t))$, this proves (i).  If
$b(t)=\infty$, the whole finite state space is a waiting region, so the same
strict monotonicity holds for every finite $\xi$.  The lower bound proved below
shows that $x(t,\xi)$ is bounded from below, so its decreasing limit exists
and equals the second line of \eqref{eq:dynamic-credit-limit}.  This proves
(ii).

At $t=T$, $Q(T,\xi)=\widetilde U_B(\xi)$ and the constrained envelope formula
gives
\[
 -Q_\xi(T,\xi)=-\widetilde U_B'(\xi)=\hat x_B(\xi).
\]
Since $b(T)=q_B$ and $\hat x_B(q_B)=\underline x_B$, assertion (iii) follows.

For every finite $\xi$, \eqref{eq:dynamic-budget-envelope} and
$X_T^{t,\xi}\ge\underline x_B$ imply
\begin{align*}
 x(t,\xi)
 &=\mathbb E_t\!\left[
 \int_t^T\mathcal H_s^t(c_s^{t,\xi}-\epsilon)ds
 +\mathcal H_T^tX_T^{t,\xi}\right]\\
 &>-\int_t^T e^{-r(s-t)}\epsilon\,ds
   +e^{-r(T-t)}\underline x_B\\
 &=-h(t)+e^{-r(T-t)}\underline x_B=\ell(t),
\end{align*}
where we used $c^{t,\xi}>0$ and
$\mathbb E_t[\mathcal H_s^t]=e^{-r(s-t)}$.  Passing to the boundary or to the
limit $\xi\to\infty$ yields
$\underline x(t)\ge\ell(t)>-h(t)$.

If $b(s)=\infty$ for every $s\in[t,T)$, the cumulative multiplier does not act before the
terminal time.  As $\xi\to\infty$,
$c_s^{t,\xi}$ decreases to $0$ for $s<T$, while the terminal reflection at
$b(T)=q_B$ implies that $X_T^{t,\xi}$ decreases to $\underline x_B$.  Monotonicity provides
an integrable dominating claim, so dominated convergence in
\eqref{eq:dynamic-budget-map} yields
\[
 \underline x(t)=-h(t)+e^{-r(T-t)}\underline x_B.
\]
This proves (iv).  Taking $t=0$ gives the abbreviation
\eqref{eq:initial-credit-limit-abbrev} for the initial credit limit.
\end{proof}

\subsection{Proof of Corollary \ref{cor:feasible-wealth-domain}}
\label{proof:cor:feasible-wealth-domain}
\begin{proof}
Fix $t<T$.  The weak duality argument restarts verbatim at time $t$. For every
admissible continuation strategy with current wealth $x$ and every $\xi>0$,
its continuation utility is bounded above by
\begin{equation}\label{eq:dynamic-weak-duality-feasible-proof}
 Q(t,\xi)+\xi x.
\end{equation}
Suppose $x<\underline x(t)$.  Since
$-Q_\xi(t,\xi)=x(t,\xi)\to\underline x(t)$ as the effective dual state
approaches the intervention boundary (or infinity when $b(t)=\infty$), there
exist $\delta>0$ and $\xi_0$ such that
\[
 \frac{d}{d\xi}\bigl(Q(t,\xi)+\xi x\bigr)
 =Q_\xi(t,\xi)+x\le-\delta
\]
for all sufficiently large effective $\xi$.  Hence
$Q(t,\xi)+\xi x\to-\infty$.  This contradicts
\eqref{eq:dynamic-weak-duality-feasible-proof}, because an admissible strategy
has finite continuation utility.  Therefore no feasible continuation exists
below $\underline x(t)$.

Conversely, Proposition~\ref{prop:dynamic-credit-limit} shows that
$x(t,\xi)=\mathcal B_t(\xi)$ is continuous, strictly decreasing on the waiting
range, and tends to $\underline x(t)$ at its upper endpoint.  The same CRRA
estimate used in Proposition~\ref{prop:dual-envelope-budget} gives
$x(t,\xi)\to\infty$ as $\xi\to0^+$.  Thus, for every
$x>\underline x(t)$ there is a unique $\xi$ in the waiting region with
$\mathcal B_t(\xi)=x$.  Repeating the wealth and portfolio reconstruction on
$[t,T]$ produces an admissible continuation strategy with initial wealth $x$.
Therefore the infimum of the feasible wealth domain is exactly
$\underline x(t)$.  If $b(t)<\infty$, the same reflected construction with $\xi=b(t)$ finances the endpoint and has initial utility $G(t,b(t))=P(t)$. Weak duality proves its optimality. The terminal assertion follows directly from $X_T\ge\underline x_B$.
\end{proof}

\subsection{Proof of Proposition \ref{prop:dual-envelope-budget}}
\label{proof:prop:dual-envelope-budget}
\begin{proof}
Since $J(y)=Q(0,y)$ and $Q_\xi$ is continuous, Proposition
\ref{prop:dynamic-credit-limit} immediately yields
\[
 J'(y)=Q_\xi(0,y)=-\mathcal B_0(y)=-\mathcal B(y).
\]
The same proposition shows that $\mathcal B$ is continuous and nonincreasing,
and that it is strictly decreasing on the waiting range $y<b(0)$.  It also
gives
\[
 \lim_{y\to\infty}\mathcal B(y)=\underline x(0)=\underline x
 \ge-h(0)+e^{-rT}\underline x_B>-h(0).
\]

It remains to verify the behavior of the budget map as $y\to0^+$.  Because
$1+\mathcal Z_t^{*,y}\ge1$,
\[
 \Xi_t^y\le\mathcal Y_t^y
 =y e^{\beta t}\mathcal H_t,
\]
and hence, using $I(q)=q^{-1/\gamma}$,
\[
 c_t^y=(\Xi_t^y)^{-1/\gamma}
 \ge y^{-1/\gamma}e^{-\beta t/\gamma}
 \mathcal H_t^{-1/\gamma}.
\]
The lognormal state price density has all finite moments on the finite horizon,
so for some deterministic constant $C>0$,
\[
 \mathbb E\!\left[\int_0^T\mathcal H_t c_t^y\,dt\right]
 \ge C y^{-1/\gamma}.
\]
The income term and the terminal floor term remain bounded from below, and therefore
$\mathcal B(y)\to\infty$ as $y\to0^+$.

Consequently $\mathcal B$ maps the waiting interval $(0,b(0))$ continuously
and strictly decreasingly onto $(\underline x,\infty)$, with the obvious
interpretation $(0,b(0))=(0,\infty)$ when $b(0)=\infty$.  Thus each
$x>\underline x$ determines a unique $y^*$ satisfying
$\mathcal B(y^*)=x$.  Since
\[
 \frac{d}{dy}\bigl(J(y)+yx\bigr)=x-\mathcal B(y),
\]
the derivative changes sign from negative to positive at $y^*$, so $y^*$
is the global minimizer of the dual objective.  This proves all assertions.
\end{proof}

\subsection{Proof of Proposition \ref{prop:wealth-credit-reconstruction}}
\label{proof:prop:wealth-credit-reconstruction}
\begin{proof}
Let $A_t^*:=1+\mathcal Z_t^*$ and fix $t\in[0,T]$.  From the running supremum
representation of the optimal cumulative multiplier,
\[
 A_s^*=1\vee\sup_{0\le u\le s}\frac{\mathcal Y_u^{y^*}}{b(u)},
 \qquad s\ge t.
\]
Since $A_t^*$ already contains the entire running maximum up to time $t$, and
$\Xi_t^*=\mathcal Y_t^{y^*}/A_t^*\le b(t)$, division by $A_t^*$ gives
\begin{equation}\label{eq:future-scaling-reflection-proof}
 \frac{A_s^*}{A_t^*}
 =1\vee\sup_{t\le u\le s}
 \frac{\Xi_t^*e^{\beta(u-t)}\mathcal H_u^t}{b(u)},
 \qquad t\le s\le T.
\end{equation}
Thus, conditional on $\mathcal F_t$, the future normalized dual state and the
future increments of the cumulative multiplier are exactly those of the restarted
verified problem with initial normalized state $\Xi_t^*$ and zero
multiplier before intervention.  In particular,
\[
 c_s^*=c_s^{t,\Xi_t^*},\qquad
 X_T^*=X_T^{t,\Xi_t^*},\qquad s\in[t,T].
\]
Using \eqref{eq:replicating-wealth}, \eqref{eq:dynamic-budget-map}, and
Proposition~\ref{prop:dynamic-credit-limit}, we therefore obtain
\[
 X_t^*=\mathcal B_t(\Xi_t^*)
 =-Q_\xi(t,\Xi_t^*)=x(t,\Xi_t^*),
\]
which proves \eqref{eq:wealth-Qxi-representation}.

By the Skorokhod state constraint, $\Xi_t^*\le b(t)$.  Proposition
\ref{prop:dynamic-credit-limit} shows that the wealth image of the waiting
region lies above $\underline x(t)$ and that the boundary $b(t)$ maps to
$\underline x(t)$.  Hence
$X_t^*\ge\underline x(t)$, and the lower bound in
\eqref{eq:optimal-wealth-credit-limit} follows from
\eqref{eq:dynamic-credit-limit-lower-bound}.  The terminal statement follows
from \eqref{eq:terminal-wealth-dual-map}.

Finally, Proposition~\ref{prop:dual-shadow-reflection} implies that
$d\mathcal Z_t^*$ is carried by the set $\{\Xi_t^*=b(t)\}$ (with no increase
when $b(t)=\infty$).  On this support,
\eqref{eq:credit-limit-collapse} and
\eqref{eq:wealth-Qxi-representation} give
$X_t^*=\underline x(t)$.  Therefore
\[
 \int_{[0,T]}(X_t^*-\underline x(t))\,d\mathcal Z_t^*=0,
\]
which is \eqref{eq:wealth-credit-complementarity}.
\end{proof}

\subsection{Proof of Theorem \ref{thm:strong-duality}}\label{proof:thm:strong-duality}
\begin{proof}
Choose $y^*$ such that $\mathcal B(y^*)=x$.  By
\eqref{eq:replicating-wealth}, $X_0^*=\mathcal B(y^*)=x$ and
$X_T^*$ is the prescribed terminal wealth.  Define $M^X$ by
\eqref{eq:martingale-reconstruction}.  The martingale representation theorem
gives $dM_t^X=\varphi_t dB_t$, and comparison of
\eqref{eq:MX-HX} with It\^o's formula for $\mathcal H_tX_t^*$ gives
\[
 d(\mathcal H_tX_t^*)
 =-\mathcal H_t(c_t^*-\epsilon)dt
 +\mathcal H_t(\sigma\pi_t^*-\theta X_t^*)dB_t,
\]
where $\pi^*$ is exactly \eqref{eq:pi-star-explicit}.  Dividing by
$\mathcal H_t$ yields \eqref{eq:wealth_finite}.  Equations \eqref{eq:wealth-Qxi-representation}--
\eqref{eq:optimal-wealth-credit-limit} give
$X_t^*\ge\underline x(t)>-h(t)$ for $t<T$ and
$X_T^*\ge\underline x_B=\underline x(T)$.  The complementarity in wealth space
\eqref{eq:wealth-credit-complementarity} follows from the Skorokhod support
property and Proposition~\ref{prop:dynamic-credit-limit}.  Proposition
\ref{prop:shadow-welfare-identity} gives
$\mathcal W_t^{y^*}\ge P(t)$.  Lemma~\ref{lem:reflection-moments} gives the utility and consumption integrability. The terminal random variable defining $M^X$ has moments of every order, so its martingale representation satisfies $\int_0^T\varphi_t^2dt<\infty$ a.s. Since $\mathcal H^{-1}$ and the continuous replicating wealth are bounded on each sample path over $[0,T]$, formula \eqref{eq:pi-star-explicit} gives $\int_0^T(\pi_t^*)^2dt<\infty$ a.s. Thus every condition defining $\mathcal A(x)$ is verified.

By Proposition~\ref{prop:shadow-welfare-identity}, the utility of
$(c^*,\pi^*)$ is $\mathcal W_0^{y^*}=G(0,\Xi_0^*)$.  Since $y^*<b(0)$ by
Proposition~\ref{prop:dual-envelope-budget}, no intervention occurs at time
$0$, so $\Xi_0^*=y^*$ and, by \eqref{eq:J-Q-initial} and
\eqref{eq:initial-budget-Qxi},
\[
 \mathcal W_0^{y^*}
 =Q(0,y^*)-y^*Q_\xi(0,y^*)
 =J(y^*)+y^*\mathcal B(y^*)
 =J(y^*)+y^*x.
\]
Weak duality \eqref{eq:weak_duality_finite} gives $V(x)\le J(y)+yx$ for every
$y>0$, and Proposition~\ref{prop:dual-envelope-budget} shows that $y^*$
minimizes the right hand side.  Hence
\[
 V(x)=J(y^*)+y^*x=\inf_{y>0}\{J(y)+yx\},
\]
and $(c^*,\pi^*)$ is primal optimal.  The budget constraint binds because
$\mathcal B(y^*)=x$, and \eqref{eq:welfare-complementarity} is part of
Proposition~\ref{prop:shadow-welfare-identity}. Uniqueness follows from strict concavity. The admissible class is convex in $(c,X_T)$, since the continuation value is concave in these variables and the wealth equation is linear, so two optimal pairs $(c,X_T)$ that differ on a set of positive $dt\otimes d\mathbb P$ measure would have a strictly better convex combination. The replicating wealth and portfolio are then determined by \eqref{eq:replicating-wealth} and \eqref{eq:pi-star-explicit}.
\end{proof}

\end{document}